\documentclass[12pt]{amsart}

\usepackage{pdfsync}         
\usepackage[active]{srcltx}  
\usepackage{tikz}
\usepackage{enumerate}
\usepackage{comment}
\usepackage{amssymb}
\usepackage{subcaption}
\usepackage{mathtools}

\usepackage[OT1]{fontenc}    %
\usepackage{graphicx}

\usepackage{palatino}
\usepackage[left=3cm,top=3.8cm,right=3cm]{geometry}        

\usepackage{xcolor}
\usepackage[pagebackref]{hyperref}
\hypersetup{
   colorlinks,
    linkcolor={red!60!black},
    citecolor={blue!60!black},
    urlcolor={blue!90!black}
}

\numberwithin{equation}{section}

\theoremstyle{plain}
\newtheorem{theorem}{Theorem}[section]
\newtheorem{corollary}[theorem]{Corollary}
\newtheorem{prop}[theorem]{Proposition}
\newtheorem{lemma}[theorem]{Lemma}

\theoremstyle{remark}
\newtheorem{remark}[theorem]{Remark}

\theoremstyle{definition}
\newtheorem{definition}[theorem]{Definition}

\newcommand{\e}{\varepsilon}
\newcommand{\N}{\mathbb{N}}
\newcommand{\R}{\mathbb{R}}
\newcommand{\Z}{\mathbb{Z}}
\newcommand{\dist}{\mathrm{dist}}

\newcommand{\cP}{\mathcal{P}}

\newcommand{\cN}{\mathcal{N}}
\newcommand{\cD}{\mathcal{D}}

\newcommand{\cE}{\mathcal{E}}

\newcommand{\al}{\alpha}
\newcommand{\de}{\delta}
\newcommand{\eps}{\varepsilon}
\newcommand{\su}{\subset}

\DeclareMathOperator{\hdim}{dim_H}
\DeclareMathOperator{\pdim}{dim_P}

\DeclareMathOperator{\ubdim}{\overline{\dim}_B}

\DeclareMathOperator{\supp}{supp}

\DeclareMathOperator{\bad}{\mathbf{Bad}}

\DeclareMathOperator{\error}{Error}

\newcommand{\wh}{\widehat}
\newcommand{\wt}{\widetilde}

\newcommand{\M}{\mathbf{M}}
\newcommand{\T}{\mathbf{T}}
\newcommand{\D}{\mathbf{D}}

\newcommand{\Arrow}[1]{%
\parbox{#1}{\tikz{\draw[->](0,0)--(#1,0);}}
}

\newcommand{\sto}{\Arrow{.25cm}}
\newcommand{\ssto}{\Arrow{.2cm}}

\title{New bounds on the dimensions of planar distance sets}

\author{Tam\'as Keleti and Pablo Shmerkin}
\address
{Institute of Mathematics, E\"otv\"os Lor\'and University,
P\'az\-m\'any P\'e\-ter s\'et\'any 1/c, H-1117 Budapest, Hungary}
\email{tamas.keleti@gmail.com}
\urladdr{http://web.cs.elte.hu/analysis/keleti}
\thanks{T.K is grateful to the Alfr\'ed R\'enyi Institute of Mathematics
  of the Hungarian Academy of Sciences, where he was a visiting researcher
  during this project.
  He was partially supported by the Hungarian National Research, Development and Innovation Office –- NKFIH, 104178 and 124749}

\address{Department of Mathematics and Statistics, Torcuato Di Tella University, and CONICET, Buenos Aires, Argentina}
\email{pshmerkin@utdt.edu}
\urladdr{http://www.utdt.edu/profesores/pshmerkin}
\thanks{P.S. was partially supported by Projects CONICET-PIP 11220150100355 and PICT 2014-1480 (ANPCyT). Both authors thank Institut Mittag Leffler for hospitality and financial support.}


\begin{document}

\begin{abstract}
We prove new bounds on the dimensions of distance sets and pinned distance sets of planar sets. Among other results, we show that if $A\subset\R^2$ is a Borel set of Hausdorff dimension $s>1$, then its distance set has Hausdorff dimension at least $37/54\approx 0.685$. Moreover, if $s\in (1,3/2]$, then outside of a set of exceptional $y$ of Hausdorff dimension at most $1$, the pinned distance set $\{ |x-y|:x\in A\}$ has Hausdorff dimension $\ge \tfrac{2}{3}s$ and packing dimension at least $ \tfrac{1}{4}(1+s+\sqrt{3s(2-s)}) \ge 0.933$. These estimates improve upon the existing ones by Bourgain, Wolff, Peres-Schlag and Iosevich-Liu for sets of Hausdorff dimension $>1$. Our proof uses a multi-scale decomposition of measures in which, unlike previous works,  we are able to choose the scales subject to certain constrains. This leads to a combinatorial problem,  which is a key new ingredient of our approach, and which we solve completely by optimizing certain variation of Lipschitz functions.
\end{abstract}

\maketitle

\tableofcontents

\section{Introduction and statement of results}

\subsection{Introduction}

Given $A\subset\R^d$, its \emph{distance set} is $\Delta(A)=\{|x-y|:x,y\in A\}$. K. Falconer \cite{Falconer85} pioneered the study of the relationship between the Hausdorff dimensions of $A$ and $\Delta(A)$. He proved that if $d\ge 2$ and $A\subset\R^d$ is a Borel (or even analytic) set then $\hdim(\Delta(A)) \ge \min(\hdim(A)-\tfrac{1}{2}(d-1),1)$, where $\hdim$ stands for Hausdorff dimension. Falconer also constructed compact sets $A\subset\R^d$ (based on lattices) of any Hausdorff dimension such that $\hdim(\Delta(A)) \le \min(2\hdim(A)/d,1)$. Although it is not explicitly stated in \cite{Falconer85}, the conjecture that these lattice constructions are extremal, in the sense that one should have $\hdim(\Delta(A))=1$ if $\hdim(A)\ge d/2$, has become known as the Falconer distance set problem.

Falconer's problem is a continuous version of the celebrated P. Erd\H{o}s distinct distances problem \cite{Erdos46}, asserting (in the plane) that if $|A|=N$, $A\subset\R^2$, then $|\Delta(A)|\ge c N/\sqrt{\log N}$. L. Guth and N. Katz \cite{GuthKatz15} (building up on work of Gy. Elekes and M. Sharir \cite{ElekesSharir11}) famously solved this problem, up to logarithmic factors, by showing that $|\Delta(A)| \ge c N/\log N$. However, the approach of Guth and Katz and, indeed, all previous methods developed to tackle Erd\H{o}s' problem, do not appear to be able to yield progress on Falconer's problem.

From now on, we focus on the case $d=2$, which is the first non-trivial case,  the best understood, and the focus of this article. T. Wolff \cite{Wolff99}, based on a method of P. Mattila \cite{Mattila87} and extending ideas of J. Bourgain \cite{Bourgain94}, proved that if $A\subset\R^2$ is a Borel set with $\hdim(A)\ge 4/3$, then $\hdim(\Delta(A))=1$. In fact, he proved that $\hdim(A)>4/3$ ensures that $\Delta(A)$ has positive length, and established the more general dimension formula
\begin{equation} \label{eq:bound-wolff}
\hdim(\Delta(A)) \ge\min\left( \frac{3}{2}\hdim(A)-1,1\right),
\end{equation}
whenever $\hdim(A)>1$. The method developed by Mattila and Wolff is strongly Fourier-analytic, depending on difficult estimates for the decay of circular averages of the Fourier transform of measures.

Later Bourgain \cite{Bourgain03}, crucially relying on earlier work of N. Katz and T. Tao \cite{KatzTao01}, proved that if $A\subset\R^2$ satisfies $\hdim(A)\ge 1$, then
\begin{equation} \label{eq:Bourgain-bound}
\hdim(\Delta(A)) > \frac{1}{2} + \delta,
\end{equation}
where $\delta>0$ is a universal constant. Although non-explicit, it is clear from the proof that the value of $\delta$ one would get is extremely small. The method of Katz-Tao and Bourgain is based on additive combinatorics, and it seems difficult for this type of arguments to yield reasonable values of $\delta$.

A related problem concerns the dimensions of \emph{pinned} distance sets
\[
\Delta_y(A) = \{ |x-y|:x\in A\}.
\]
Y. Peres and W. Schlag \cite[Theorem 8.3]{PeresSchlag00} proved that if $A\subset\R^2$ is a Borel set with $\hdim(A)=s$, then for all $0<t\le \min(s,1)$,
\begin{equation} \label{eq:PeresSchlag}
\hdim\{ y\in\R^2: \hdim(\Delta_y(A)) < t \} \le 2+t-\max(s,1).
\end{equation}
Recently, A. Iosevich and B. Liu \cite{IosevichLiu17} proved that \eqref{eq:PeresSchlag} remains true with $3+3t-3s$ in the right-hand side. This is an improvement in some parts of the parameter region. Both results imply that if $\hdim(A)>3/2$, then there is $y\in A$ such that $\hdim(\Delta_y A)=1$, and it is unknown whether $3/2$ can be replaced by a smaller number. We remark that the results of both \cite{PeresSchlag00} and \cite{IosevichLiu17} extend to higher dimensions.

These were the best known results towards Falconer's problem in the plane for general sets prior to this article. For some special classes of sets, better results are known. In particular, the second author proved in \cite{Shmerkin17} that if $A\subset\R^2$ is a Borel set of equal Hausdorff and packing dimension, and this value is $>1$, then $\hdim(\Delta_y(A))=1$ for all $y$ outside of a set of exceptions of Hausdorff dimension at most $1$, and in particular for many $y\in A$. This verifies Falconer's conjecture for this type of sets, outside of the endpoint. We remark that T. Orponen \cite{Orponen17b} and the second author \cite{Shmerkin17b} had previously proved weaker results of the same kind. See also \cite{Mattila87, IosevichLiu16} for other results on the distance sets of special classes of sets.

\subsection{Main results}

In this article we prove new lower bounds on the dimensions of (pinned) distance sets, which in particular greatly improve the best previously known estimates when $\hdim(A)=1+\delta$, $\delta>0$ small.

\begin{theorem} \label{thm:main-simple}
If $A$ is a Borel subset of $\R^2$ with $\hdim A=s$, then
\begin{equation} \label{eq:claim-main-thm}
\hdim\left\{ y\in\R^2: \hdim(\Delta_y(A)) <  \min\left(\frac{2}{3}s,1\right) \right\} \le \max(1,2-s).
\end{equation}
In particular, if $s>1$, then one can find many $y\in A$ such that
\[
 \hdim(\Delta_y(A)) \ge  \min\left(\frac{2}{3}s,1\right).
\]
\end{theorem}

We remark that we get better bounds for the dimension of the full distance set, see Theorem \ref{thm:full-distance-set} below.

The last claim in Theorem \ref{thm:main-simple} improves the previously known bounds for the dimensions of pinned distance sets $\Delta_y(A)$ with $y\in A$ for all $s\in (1,3/2]$.  The bound \eqref{eq:claim-main-thm}  also improves upon \eqref{eq:PeresSchlag} (and the variant of Iosevich and Liu) in large regions of parameter space, and in particular for   $t=\min(\tfrac{2}{3}s,1)$ and all $s\in (3/5,5/3)$.

Theorem \ref{thm:main-simple} is a special case of a more general result that takes into account the Hausdorff and also the packing dimension of $A$. We refer to \cite[\S 3.5]{Falconer14} for the definition and main properties of packing dimension $\pdim$, and simply note that it satisfies $\hdim(A)\le\pdim(A)\le \ubdim(A)$, where $\ubdim$ denotes the upper box-counting (or Minkowski) dimension.  For our method, the worst case is that in which $A$ has maximal packing dimension $2$, and we get better bounds for the distance set under the assumption that the packing dimension is smaller:

\begin{theorem} \label{thm:main}
Let
\[
\chi(s,u) = \begin{dcases}
\frac{s(2+u-2s)}{2+2u-3s}  &\text{ if } u>2s-1\\
1 & \text{ if } u\le 2s-1
\end{dcases}.
\]

Given $0< s\le u \le 2$, the following holds: if $A$ is a Borel subset of $\R^2$ with $\hdim A\ge s$ and $\pdim A \le u$, then
\[
\hdim\left\{ y\in\R^2: \hdim(\Delta_y(A)) <  \chi(s,u)  \right\} \le \max(1,2-s).
\]
In particular, if $s>1$ then there are many $y\in A$ such that
\[
\hdim(\Delta_y(A)) \ge  \chi(s,u),
\]
and hence if $\hdim(A)>1$ and $\pdim A\le 2\hdim A-1$, then $\hdim(\Delta_y(A))=1$ for many $y\in A$.
\end{theorem}

Note that Theorem \ref{thm:main-simple} follows immediately by taking $u=2$.  A simple calculation shows that if $0\le s\le u \le 2$ and $s<2$, then
\[
\chi(s,u) = \min\left(\frac{s(2+u-2s)}{2+2u-3s},1\right).
\]
We remark that, taking $u=s$, this theorem recovers the main result of \cite{Shmerkin17} mentioned above, namely that if $\hdim(A)=\pdim(A)>1$, then $\hdim(\Delta_y A)=1$ for many $y\in A$. On the other hand, it was known from \eqref{eq:PeresSchlag} that if $\hdim(A)>3/2$ then there is $y\in A$ such that $\hdim(\Delta_y A)=1$. The last claim in Theorem \ref{thm:main} can be seen as interpolating between these two situations, and hence provides a new, more general, geometric condition under which Falconer's conjecture is known to hold.

When $\hdim(A)>1$, we are able to get much better lower bounds for the \emph{packing} dimension of the pinned distance sets:
\begin{theorem} \label{thm:packing}
Let $A$ be a Borel subset of $\R^2$ with $s=\hdim(A)\in (1,3/2)$. Then
\[
\hdim\left\{ y\in \R^2: \pdim(\Delta_y(A)) < \frac{1+s+\sqrt{3s(2-s)}}{4} \right\} \le 1.
\]
In particular, there is $y\in A$ such that
\[
\pdim(\Delta_y(A)) \ge \frac{1+s+\sqrt{3s(2-s)}}{4} > \frac{2+\sqrt{3}}{4} =  0.933013\ldots.
\]
\end{theorem}

We recall that since upper box-counting dimension is at least as large as packing dimension, the above theorem also holds for upper box-counting dimension. Even though Falconer's conjecture is about the Hausdorff dimension of the distance set, this result presents further evidence towards its validity.

\begin{figure}
\centering
\includegraphics[width=0.9\textwidth]{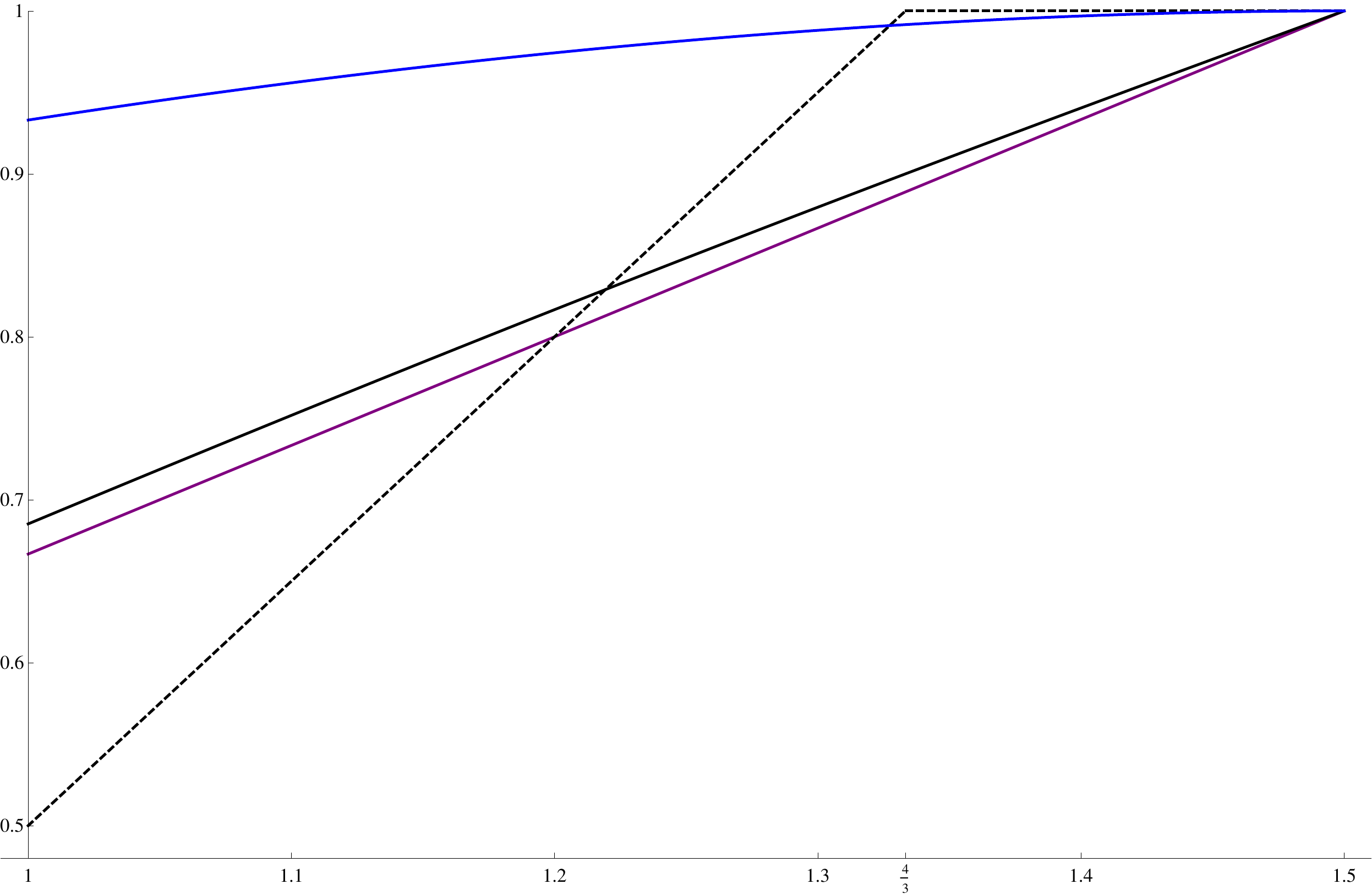}
\caption{The three solid graphs show, from top to bottom: (1) the lower bound given by Theorem \ref{thm:packing} for $\pdim(\Delta_y(A))$ for $y$ outside of a one dimensional set of $y$, (2) the lower bound for $\hdim(\Delta(A))$ given by Theorem \ref{thm:full-distance-set}, (3) the lower bound given by Theorem \ref{thm:main-simple} for $\hdim(\Delta_y(A))$ outside of a one dimensional set of $y$. The dashed line is Wolff's lower bound for $\hdim(\Delta(A))$ (which was the previously known best bound, outside of a tiny interval to the right of $1$). In all cases the variable is $\hdim(A)$.}
\label{fig:graphs}
\end{figure}

Finally, as anticipated above, we get a better bound for the dimension of the full distance set when $\hdim(A)$ is slightly larger than $1$:
\begin{theorem}  \label{thm:full-distance-set}
If $A\su\R^2$ is a Borel set with $\hdim(A)=s\in (1,4/3)$, then
\[
\hdim(\Delta(A))\ge \frac{s(147-170s+60s^2)}{18(12-14s+5s^2)}
\ge  \frac{37}{54}=0.6851851\ldots.
\]
\end{theorem}

A calculation shows that this indeed improves upon Wolff's bound \eqref{eq:bound-wolff} for the dimension of the full distance set for $s\in (1,1.21931\ldots)$ (and upon Bourgain's bound \eqref{eq:Bourgain-bound}  for all $s>1$). We remark that this theorem is obtained by combining the idea of the proof of Theorem \ref{thm:main} with a known effective variant of Wolff's bound \eqref{eq:bound-wolff}. Although achieving this combination takes quite a bit of work, Theorem \ref{thm:main} should perhaps be considered the most basic result, since its proof is shorter and already contains most of the main ideas, and the improvement given by Theorem \ref{thm:full-distance-set} is relatively modest. Note also that already applying Theorem \ref{thm:main-simple} for the full distance set improves upon \eqref{eq:bound-wolff} for $s\in (1,6/5)$. See Figure \ref{fig:graphs} for a comparison of the lower bounds from Theorems \ref{thm:main-simple}, \ref{thm:packing} and \ref{thm:full-distance-set} and Wolff's lower bound.

After this paper was made public, B. Liu \cite{Liu18} posted a preprint extending Wolff's result to pinned distance sets. In particular, he shows that if $A\subset\R^2$ is a Borel set with $\hdim(A)>4/3$, then $\Delta_x(A)$ has positive Lebesgue measure for some $x\in A$ (with bounds on the dimension of the exceptional set). This is stronger than our Theorem \ref{thm:main-simple} for $s>4/3$ (other than the exceptional set being larger).

\subsection{Strategy of proof}

Our approach is completely different to those of Wolff, Bourgain, Peres and Schlag and Iosevich and Liu. Rather, it can be seen as a continuation of the ideas successively developed in \cite{Orponen17b, Shmerkin17b, Shmerkin17} to attack the distance set problems for sets with certain regularity. Thus, one of the main points of this paper is extending the strategy of these papers so that it can be applied to general sets.

At the core of our method is a lower box-counting estimate for pinned distance sets $\Delta_y A$ in terms of a multi-scale decomposition of $A$ or, rather, a Frostman measure $\mu$ supported on $A$. See Section \ref{sec:box-counting} for precise statements. A key aspect of these estimates is that they recover a \emph{global} lower box-counting estimate for $\Delta_y A$ from bounds on \emph{local}, discretized and linearized estimates for the pinned distance measures $\Delta_y \mu$.

The general philosophy of obtaining lower bounds for the dimension of projected sets and measures, in terms of multi-scale averages of local projections is behind a large number of results in fractal geometry in the last few years, see e.g. \cite{HochmanShmerkin12, Hochman14} and references there. The insight that this approach can be used also to study distance sets is due to Orponen \cite{Orponen12, Orponen17b}.

Up until the paper \cite{Shmerkin17}, the scales in the multi-scale decomposition behind all the variants of the method described above were of the form $2^{-N j}$ for some fixed $N$. One of the innovations of \cite{Shmerkin17} was to modify the method so that it could handle also scales of the form $2^{-(1+\e)^j}$ (the point being that $(1+\e)^j$ is exponential in $j$, rather than linear). Although this was flexible enough to handle sets of equal Hausdorff and packing dimension (as opposed to Ahlfors-regular sets as in \cite{Orponen17b, Shmerkin17b}), it was still too restrictive for dealing with general sets.

One of the main innovations of this paper is that we are able to work with scales $2^{-M_j}$ where the $M_j$ only need to satisfy $\tau M_j \le M_{j+1}-M_j\le M_j+T$ (where $\tau>0, T\in\N$ are fixed parameters). This provides a major degree of flexibility. In particular, a crucial point is that we are able to pick the sequence $(M_j)$ depending on the set $A$ (or the Frostman measure $\mu$), while in all previous works the scales in the multi-scale decomposition were basically fixed. See Proposition \ref{prop:lower-bound-box-counting}. This leads us to the combinatorial problem of optimizing the choice of $(M_j)$ for each measure $\mu$. We solve this problem completely, up to negligible error terms, in Section \ref{sec:combinatorial}.

In fact, we deduce the combinatorial statements we need from several statements about the variation of Lipschitz functions, which might be of independent interest. More precisely, given a $1$-Lipschitz function $f:[0,a]\to\R$ satisfying certain additional assumptions, we seek to minimize
\[
 \sum_{n=1}^\infty f(a_n)-\min_{[a_n,a_{n-1}]} f,
\]
where $(a_n)_{n=0}^\infty$ is a strictly decreasing sequence tending to $0$ with $a=a_0$ and $a_{n} \le 2 a_{n+1}$. Conversely, we also study the structure of functions $f$ for which these sums are (for some sequence $(a_i)$) close to the minimum possible value. We underline that this part of the method is completely new as the combinatorial problem does not arise for fixed multi-scale decompositions.

Another obstacle to dealing with arbitrary sets and measures is that energies of measures (which play a key role throughout) do not have a nice multi-scale decomposition in general. We deal with this by decomposing a general measure supported on $[0,1)^2$ as a superposition of measures with a regular Cantor structure, plus a small error term: see Corollary \ref{cor:bourgain}. This step is an adaptation of some ideas of Bourgain we learned from \cite{Bourgain10}. After some technical difficulties, this reduces our study to those regular measures for which a suitable multi-scale expression of the energy does exist, see Lemma \ref{lem:energy-regular}.

The strategy just discussed is behind the proofs of Theorems \ref{thm:main}, \ref{thm:packing} and \ref{thm:full-distance-set}. However (as briefly indicated above), the proof of Theorem \ref{thm:full-distance-set} is based on merging these ideas with a more quantitative version of Wolff's result that if $\hdim(A)\ge 4/3$ then $\hdim(\Delta(A))=1$, see Theorem \ref{thm:mattila-wolff} below. The fact that one can improve upon Theorem \ref{thm:main-simple} (for the full distance set) is based on the observation that for some sets $A\subset\R^2$ of Hausdorff dimension $s>1$ for which the method of the proof of Theorem \ref{thm:main-simple} cannot give anything better than $\hdim(\Delta(A))\ge 2s/3$, the quantitative version of Wolff's Theorem can give a much better bound. The fact that these two methods are based on totally different techniques and also have different ``enemies'' that one must overcome, suggests that neither of them (or even in combination as we do here) provides a definitive line of attack on Falconer's problem.

\subsection{Sets of directions, and the case of dimension $1$}

Although Theorem \ref{thm:main-simple} does provide new information on the pinned distance sets $\Delta_y A$ when $\hdim A=1$, it gives no information whatsoever on $\hdim(\Delta(A))$ in this case. There are some well-known ``enemies'' that one must handle in order to improve upon the easy bound $\hdim(\Delta(A))\ge 1/2$ when $\hdim A=1$. One is that the corresponding fact is false over the complex numbers: $\R^2$ is a subset of $\mathbb{C}^2$ of half the dimension of the ambient space for which the (squared) distance set
\[
\Delta^2(\R^2) = \{ (x_1-y_1)^2 + (x_2-y_2)^2 : (x_1,x_2), (y_1,y_2)\in\R^2\}
\]
also has half the dimension of the ambient space. Hence any improvements over $1/2$ in the real case must take into account the order structure of $\R$. The other obstacle is a well-known counterexample to a naive discretization of the problem: see \cite[Eq. (2) and Figure 1]{KatzTao01}. These enemies do not arise when $\hdim(A)>1$. Despite these conceptual differences, we underline that, with the exception of the work of Katz and Tao \cite{KatzTao01} underpinning Bourgain's bound \eqref{eq:Bourgain-bound}, none of the other methods developed so far make any distinction between the cases $\hdim(A)=1$ and $\hdim(A)=1+\delta$.

From the point of view of our strategy, the key significance of the assumption $\hdim(A)>1$ is that in this case the sets of directions determined by points in $A$ has positive Lebesgue measure. In fact, we need a far more quantitative ``pinned'' version of this fact, which is due to Orponen \cite{Orponen17}, improving upon a related result by Mattila and Orponen \cite{MattilaOrponen16} (see Proposition \ref{prop:orponen} below). However, even the fact that the direction set has positive measure clearly fails if $\hdim A=1$ when $A$ is contained in a line. Since $\hdim(\Delta_y A) = \hdim(A)$ trivially when $A$ is contained in a line, this does not rule out an extension of our approach to the case $\hdim(A)=1$. However, this would require some variant of Proposition \ref{prop:orponen} when both $s,u$ are slightly less than $1$, under a suitable hypothesis of non-concentration on lines, and this appears to be very hard. In \cite[Corollary 1.8]{Orponen17}, Orponen also proved that the direction set of a planar set of Hausdorff dimension $1$ which is not contained in a line has Hausdorff dimension $\ge 1/2$, but this is very far from positive measure, let alone from anything resembling Proposition \ref{prop:orponen}.

To understand why directions arise naturally, we recall that our whole approach is based on bounding the size of pinned distance sets in terms of a multi-scale average of local \emph{linearized} pinned distance measures. The derivative of the distance function $x\mapsto |x-y|$ is precisely the direction spanned by $x$ and $y$. Thus we are led to study orthogonal projections of certain measures localized around $x$, where the angle is given by the direction determined by $x$ and $y$. The fact that these directions are ``well distributed'' in a suitable sense can then be used in conjunction with a finitary version of Marstrand's projection theorem (see Lemma \ref{lem:marstrand}) and several applications of Fubini to conclude that one can choose $y$ such that for ``many'' $x$ the direction determined by $x$ and $y$ is good in the sense that the $L^2$ norm of the projection is controlled by the $1$-energy of the measure being projected.

\subsection{Structure of the paper} In Section \ref{sec:notation} we introduce notation to be used in the rest of the paper. Section \ref{sec:preliminary} contains some preliminary definitions and results that will be repeatedly used in the later proofs. In Section \ref{sec:box-counting} we establish a lower bound for the box-counting numbers of pinned distance sets that will be at the heart of the proofs of all main theorems. Section \ref{sec:combinatorial} contains a number of optimization results about Lipschitz functions on the line, as well as corollaries of these results for discrete $[-1,1]$-sequences; these corollaries play a key role in the proofs of the main theorems. Theorems \ref{thm:main}, \ref{thm:packing} and \ref{thm:full-distance-set} are proved in Section \ref{sec:proofs}. We conclude with some remarks on the sharpness of our results in Section \ref{s:sharpness}.

We remark that \S\ref{subsec:initial} and \S\ref{subsec:stability} are not needed for the proof of Theorem \ref{thm:main} (the results from \S\ref{subsec:initial} are required only in the proof of Theorem \ref{thm:packing}, and \S\ref{subsec:stability}  is needed only for the proof of Theorem \ref{thm:full-distance-set}).

\subsection{Acknowledgments} This project was started while the authors were staying at Institut Mittag-Leffler as part of the program \emph{Fractal Geometry and Dynamics}. We are grateful to the organizers for the opportunity to take part, and to the organizers, staff, and fellow participants for the pleasant stay.

We also wish to thank for T. Orponen for many useful discussions at the early stage of this project, and an anonymous referee for several suggestions that improved the paper, and in particular for suggesting a simplification of the statement and proof of Proposition \ref{prop:radial}.

\section{Notation}
\label{sec:notation}

We use Landau's $O(\cdot)$ notation: given $X>0$, $O(X)$ denotes a positive quantity bounded above by $C X$ for some constant $C>0$. If $C$ is allowed to depend on some other parameters, these are denoted by subscripts. We sometimes write $X\lesssim Y$ in place of $X=O(Y)$ and likewise with subscripts. We write $X\gtrsim Y$,  $X\approx Y$ to denote $Y\lesssim X$,  $X\lesssim Y\lesssim X$ respectively.

Throughout the rest of the paper, we work with three parameters that we assume fixed: a large integer $T$ and small positive numbers $\e,\tau$. We briefly indicate their meaning:
\begin{enumerate}
\item We will decompose sets and measures in the base $2^T$. In particular, we will work with sets and measures that have a regular tree (or Cantor) structure when represented in this base: see Definition \ref{def:regular}.
\item The parameter $\tau$ will be used to define sets of bad projections: see Definition \ref{def:bad-projections}. The fact that $\tau>0$ is required to ensure that these sets have small measure. It also keeps some error terms negligible, see Proposition \ref{prop:lower-bound-box-counting}.
\item Finally, $\e$ will denote a generic small parameter; it can play different roles at different places.
\end{enumerate}

We will use the notation  $o_{T,\e,\tau}(1)=o_{T\to\infty,\e\to 0^+,\tau\to 0^+}(1)$ to denote any function $f(T,\e,\tau)$ such that
\[
f(T,\e,\tau)\ge 0 \quad\text{and}\quad  \lim_{T\to\infty,\e\to 0^+,\tau\to 0^+} f(T,\e,\tau)=0.
\]
If a particular instance of $o(1)$ is independent of some of the variables, we drop these variables from the notation.  Different instances of the $o(1)$ notation may refer to different functions of $T,\e,\tau$, and they may depend on each other, so long as they can always be made arbitrarily small.

Note that e.g. $O_\e(1)$ denotes any (finite) function of $\e$, while $o_\e(1)$ denotes a function of $\e$ that tends to $0$ as $\e\to 0^+$.

We will often work at a scale $2^{-T\ell}$;  it is useful to think that $\ell\to\infty$ while $T,\e,\tau$ remain fixed.

The family of Borel probability measures on a metric space $X$ is denoted by $\cP(X)$. If $\mu(A)>0$, then $\mu_A$ denotes the normalized restriction $\mu(A)^{-1}\mu|_A$. If $f:X\to Y$ is a Borel map, then by $f\mu$ we denote the push-forward measure, i.e. $f\mu(A)= \mu(f^{-1}A)$.

We let $\cD_j$ be the half-open $2^{-jT}$-dyadic cubes in $\R^d$ (where $d$ is understood from context), and let $\cD_j(x)$ be the only cube in $\cD_j$ containing $x\in \R^d$. Given a measure $\mu\in\cP(\R^d)$, we also let $\cD_j(\mu)$ be the cubes in $\cD_j$ with positive $\mu$-measure. Note that these families depend on $T$. Given $A\subset \R^d$, we also denote by $\cN(A,j)$ the number of cubes in $\cD_j$ that intersect $A$.

A $2^{-m}$-measure is a measure in $\cP([0,1)^d)$ such that the restriction to any $2^{-m}$-dyadic cube $Q$ is a multiple of Lebesgue measure on $Q$, i.e. a measure defined down to resolution $2^{-m}$. Likewise, a $2^{-m}$-set is a union of $2^{-m}$ dyadic cubes. If $\mu\in\cP(\R^d)$ is an arbitrary measure, then we denote
\[
R_\ell(\mu) = \sum_{Q\in\cD_\ell} \mu(Q) \text{Leb}_Q,
\]
that is $R_\ell(\mu)$ is the $2^{-T\ell}$-measure that agrees with $\mu$ on all dyadic cubes of side length $2^{-T\ell}$. We also define the corresponding analog for sets:  given $A\subset\R^d$, $R_\ell(A)$ denotes the union of all cubes in $\cD_\ell$ that intersect $A$.

Due to our use of dyadic cubes,  it will often  be convenient to deal with supports in the dyadic metric, i.e. given $\mu\in\cP([0,1)^d)$ we let
\[
\supp_{\mathsf{d}}(\mu) = \{ x: \mu(\cD_j(x))>0 \text{ for all } j\in\N\}.
\]
Note that $\mu(\supp_{\mathsf{d}}(\mu))=1$ and that $\supp_{\mathsf{d}}(\mu)\subset \supp(\mu)$.

If a measure $\mu\in\cP(\R^d)$ has a density in $L^p$, then its density is sometimes also denoted by $\mu$, and in particular  $\|\mu\|_p$ stands for the $L^p$ norm of its density.

We make some further definitions. Let $\mu\in\cP([0,1)^d)$. If $Q$ is a dyadic cube and $\mu(Q)>0$, then we denote $\mu^Q = \text{Hom}_Q\mu_Q$, where  $\text{Hom}_Q$ is the homothety renormalizing $Q$ to $[0,1)^d$. If $M<N$ be integers, then for $x\in\supp_{\mathsf{d}}(\mu)$, we define
\[
\mu(x;M \sto N) = R_{N-M} \mu^{\cD_M(x)}.
\]
In other words, $\mu(x;M\sto N)$ is the conditional measure on $\cD_M(x)$, rescaled back to the unit cube, and then stopped at resolution $2^{-(N-M)T}$.  Likewise, for $Q\in\cD_M$ with $\mu(Q)>0$ we define
\[
\mu(Q;N) = R_{N-M} \mu^Q.
\]
Note that $\mu(x;M\sto N)$ and $\mu(Q;N)$ are $2^{-(N-M)T}$-measures.

Logarithms are always to base $2$.

\section{Preliminary results}
\label{sec:preliminary}

\subsection{Regular measures and energy}

In this section we define some important notions and prove some preliminary results.

Recall that the $s$-energy of $\mu\in\cP(\R^d)$ is
\[
\cE_s(\mu) = \iint \frac{d\mu(x)d\mu(y)}{|x-y|^s}.
\]

\begin{lemma} \label{lem:energy}
For any Borel probability measure $\mu$ on $[0,1]^d$, if $s>0$ then
\[
\cE_s(\mu) \approx_{s,d,T} \sum_{j=1}^\infty 2^{s T j} \sum_{Q\in\cD_j} \mu(Q)^2.
\]
If $\mu$ is a $2^{-{T\ell}}$-measure and $0<s<d$, then the sum runs up to $\ell$ (in particular, the $s$-energy is finite).
\end{lemma}
\begin{proof}
First of all, by \cite[Theorem 3.1]{PemantlePeres95}, we can replace $\cE_s(\mu)$ by the $s$-energy on the $2^T$-ary tree, i.e. by
\[
\iint 2^{s T|x\wedge y|} \,d\mu(x)d\mu(y),
\]
where $|x\wedge y| = \max\{ j: y\in \cD_j(x)\}$ (both energies are comparable up to a $O_{T,d}(1)$ factor). The formula for $\cE_s(\mu)$ now follows from a standard calculation, see e.g. \cite[Lemma 3.1]{Shmerkin17} for the case $T=1$ (the proof of the general case is identical).

Finally, the case in which $\mu$ is a $2^{-{T\ell}}$-measure follows again from another simple calculation, see e.g. \cite[Lemma 3.2]{Shmerkin17} for the case $T=1$.
\end{proof}

One of the key steps in the proof of the main theorems is to decompose an arbitrary $2^{-T\ell}$-measure in terms of measures which have a uniform tree structure when represented in base $2^T$. This notion (which is inspired by some constructions of Bourgain \cite{Bourgain10}) is made precise in the next definition.

\begin{definition} \label{def:regular}
Given a sequence $\sigma=(\sigma_1,\ldots,\sigma_{\ell})\in [-1,d-1]^\ell$, we say that $\mu\in \cP([0,1)^d)$ is \emph{$\sigma$-regular} if it is a $2^{-{T\ell}}$-measure, and for any $Q\in \cD_{j}(\mu)$, $1\le j\le\ell$, we have
\[
\mu(Q) \le 2^{-T(\sigma_j+1)} \mu(\wh{Q}) \le 2\mu(Q),
\]
where $\wh{Q}$ is the only cube in $\cD_{j-1}$ containing $Q$.
\end{definition}
The expression $2^{-T(\sigma_j+1)}$ in the definition may appear strange, but it turns out to be a convenient normalization. The key point in this definition is that a measure is $\sigma$-regular if all cubes of positive mass have roughly the same mass, and the sequence $(\sigma_j)$ helps quantify this common mass.

\begin{lemma} \label{lem:energy-regular}
If $\nu\in\cP([0,1)^d)$ is $\sigma$-regular for some $\sigma\in\R^\ell$ and $s\in (0,d)$, then
\[
\left|\log \cE_s(\nu)-\left( T \max_{j=0}^{\ell} \sum_{i=1}^j (s-1)-\sigma_i \right)\right|\le O(\ell) + O_{d,s,T}(1).
\]
\end{lemma}
\begin{proof}
We use crude bounds which are enough for our purposes. From the definition it is clear that if $Q\in\cD_j(\nu)$ then
\[
2^{-\ell} 2^{-T(\sigma_1+1)}  \cdots 2^{-T(\sigma_j+1)}   \le 2^{-j} 2^{-T(\sigma_1+1)}  \cdots 2^{-T(\sigma_j+1)}  \le \nu(Q) \le 2^{-T(\sigma_1+1)}  \cdots 2^{-T(\sigma_j+1)} .
\]
This implies, in particular, that
\begin{equation} \label{eq:box-counting-regular}
2^{T(\sigma_1+1)}  \cdots 2^{T(\sigma_j+1)}  \le |\cD_j(\nu)|\le 2^\ell 2^{T(\sigma_1+1)}  \cdots 2^{T(\sigma_j+1)} .
\end{equation}
From the two displayed equations and  Lemma \ref{lem:energy} it follows that
\[
2^{-2\ell} \sum_{j=1}^{\ell} 2^{-T(\sigma_1+\ldots+\sigma_j+j)} \cdot  2^{s T j}  \lesssim_{d,s,T} \cE_s(\nu) \lesssim_{d,s,T} 2^\ell \sum_{j=1}^{\ell}  2^{-T(\sigma_1+\ldots+\sigma_j+j)}\cdot 2^{sTj}.
\]
Write $\mathcal{M}_s(\sigma) := T \max_{j=1}^{\ell} \sum_{i=1}^j (s-1)-\sigma_j$. Bounding  $\sum_{j=1}^\ell$ by $\ell$ times the maximal term in the right-hand side, we deduce that
\[
\mathcal{M}_s(\sigma) -2\ell-O_{d,s,T}(1) \le \log \cE_s(\nu) \le  \mathcal{M}_s(\sigma) + \ell +\log\ell +O_{d,s,T}(1).
\]
This yields the claim.
\end{proof}
Heuristically, the previous lemma says that for $\log\cE_s(\nu)$ to be small, it must hold that
\[
\sum_{i=1}^j \sigma_i \ge (s-1)j, \quad j=0,\ldots,\ell.
\]
Recalling the connection of $\sigma_i$ to branching numbers, this means that the average branching number over any initial set of scales has to be sufficiently large, in a manner depending on $s$.

The following is a variant of Bourgain's regularization argument (see e.g. \cite[Section 2]{Bourgain10} for a clean example). Recall that  $\supp_{\mathsf{d}}(\mu)$ denotes the dyadic support of $\mu$.
\begin{lemma} \label{lem:bourgain}
Let $\mu$ be a $2^{-{T\ell}}$-measure on $[0,1)^d$ for some $\ell\ge 1$. There exists a $2^{-{T\ell}}$-set $X$, contained in $\supp_{\mathsf{d}}(\mu)$ and satisfying $\mu(X) \ge (2Td+2)^{-\ell}$, such that $\mu_X$ is $\sigma$-regular for some sequence $\sigma\in [-1,d-1]^\ell$.
\end{lemma}
\begin{proof}
Recall that $\wh{Q}$ is the only cube in $\cD_{j-1}$ containing $Q\in\cD_j$. For each $k\in [0,Td]\cap\Z$, let
\[
X_{\ell}^{(k)} = \bigcup \{ Q\in \cD_{\ell}:   \mu(Q) \le  2^{-k} \mu(\wh{Q}) < 2\mu(Q) \},
\]
and set
\[
X_{\ell}^{(>Td)} = \bigcup \{ Q\in \cD_{\ell}:   \mu(Q) \le  2^{-(Td+1)} \mu(\wh{Q})  \}.
\]
Note that
\[
\mu\left(X_{\ell}^{(>Td)}\right) \le 2^{-(Td+1)} 2^{Td} = \frac12,
\]
and that $\supp_{\mathsf{d}}(\mu)$ is the union of the $X_\ell^{(k)}$ together with $X_{\ell}^{(>Td)}$. Pick the smallest $k=k(\ell)\in [0,Td]$ which maximizes $\mu(X_{\ell}^{(k)})$ and set $\sigma_\ell = k/T-1\in [-1,d-1]$. Then
\[
\mu\left(X_{\ell}^{(k)} \right) \ge \frac{1}{2(Td+1)} .
\]
Set $X_{\ell}:= X_{\ell}^{(k)}$ and $\mu_{\ell}=\mu_{X_{\ell}}$.

Now continue inductively, replacing $\ell$ by $\ell-1$ and $\mu$ by $\mu_{\ell}$, until we eventually get a set $X_1$ and a sequence $(\sigma_1,\ldots,\sigma_\ell)\in [-1,d-1]^\ell$. Note that for $Q\in\cD_j(\mu_i)$ the value of $\mu_i(Q)/\mu_i(\wh{Q})$ remains constant for $i\le j$ and, in particular, for $i=1$. Hence $X=X_1$ has the desired properties.
\end{proof}

The set $X$ given by the lemma will have far too little measure for our purposes: later we will need $\mu_X(A)$ to be large (in particular nonzero) for certain sets $A$ of mass roughly $\ell^{-2}$. By iterating the construction, we are able to get a moderately long sequence of sets $X_i$ such that $\mu(\R^d\setminus\cup_i X_i)\ll \ell^{-2}$; by pigeonholing we will then be able to select some $X_i$ with $\mu_{X_i}(A)$ suitably large.
\begin{corollary} \label{cor:bourgain}
Fix $\ell\ge 1$, write $m=T\ell$, and let $\mu$ be a $2^{-m}$-measure on $[0,1)^d$.  There exists a family of pairwise disjoint $2^{-m}$-sets $X_1,\ldots, X_N$  with $X_i\subset\supp_{\mathsf{d}}(\mu)$, and such that:
\begin{enumerate}[(\rm i)]
\item $\mu\left(\bigcup_{i=1}^N X_i\right) \ge 1-2^{-\e m}$. In particular, if $\mu(A)> 2^{-\e m}$, then there exists $i$ such that $\mu_{X_i}(A)\ge \mu(A)-2^{-\e m}$.
\item $\mu(X_i) \ge 2^{-(\e+(1/T)\log(2d T+2)) m} \ge 2^{-o_{T,\e}(1) m}$,
\item Each $\mu_{X_i}$ is $\sigma(i)$-regular for some $\sigma(i)\in  [-1,d-1]^\ell$.
\end{enumerate}
Moreover, the family $(X_i)_{i=1}^N$ may be constructed so that it is determined by $d, T,\e,\ell$ and $\mu$ (even though there may be other families satisfying the above properties).
\end{corollary}
\begin{proof}
Let $X_1$ be the set given by Lemma \ref{lem:bourgain}, and put $B_1=[0,1]^d\setminus X_1$. Continue inductively: once $X_j,B_j$ are defined, let $X_{j+1}$ be the set given by  Lemma \ref{lem:bourgain} applied to $\mu_{B_j}$, and set $B_{j+1}=B_j\setminus X_{j+1}$. Then (setting $B_0=[0,1)^d$)
\begin{equation} \label{eq:X-has-large-measure}
\mu(B_j) \ge 2^{-\e m} \quad \Longrightarrow \quad \mu(X_{j+1}) \ge 2^{-\e m} (2d T+2)^{-\ell}.
\end{equation}
Let $N$ be the smallest integer such that $\mu(B_N) \le  2^{-\e m}$; such $N$ exists thanks to \eqref{eq:X-has-large-measure}.

It is clear that in this construction the family $X_1,\ldots,X_N$ is determined by $d,T,\e,\ell,\mu$ since the set $X$ constructed in the proof of Lemma \ref{lem:bourgain} is determined by $d,T,\ell,\mu$.

The first part of claim (i) is immediate. Then note that
\[
\mu(A)-2^{-\e m} \le \sum_{i=1}^N \mu(X_i \cap A) = \sum_{i=1}^N \mu(X_i) \mu_{X_i}(A),
\]
so there must be $i$ such that $\mu_{X_i}(A) \ge \mu(A)-2^{-\e m}$, as claimed.

Finally, (ii) is immediate from \eqref{eq:X-has-large-measure} and the definition of $N$, and (iii) is clear since the sets $X_i$ were provided by Lemma \ref{lem:bourgain}.
\end{proof}

\subsection{Sets of bad projections}
\label{subsec:badprojections}

In this subsection, we introduce sets of ``bad'' multi-scale projections for a measure $\mu$ around a point $x$. The simple fact that these sets can be taken to have small measure (independently of $\mu$ and $x$) will play a crucial role later. Although a similar notion was introduced in \cite{Shmerkin17}, the sets of bad projections we use here are far more flexible and also more involved, depending on the decomposition into regular measures provided by Corollary \ref{cor:bourgain}.

Given $\theta\in S^1$, we denote the orthogonal projection $x\mapsto x\cdot \theta$ by $\Pi_\theta$.  Normalized Lebesgue measure on $S^1$ will be denoted by $|\cdot|$. We recall the following consequence of the energy version of Marstrand's projection theorem.
\begin{lemma} \label{lem:marstrand}
Let $\mu\in\cP([0,1)^2)$ have finite $1$-energy. Then, for any $R>0$,
\[
|\{ \theta\in S^1: \|\Pi_\theta\mu \|_2^2 \ge R  \cE_1(\mu) \}| \lesssim R^{-1}.
\]
\end{lemma}
\begin{proof}
This is just a consequence of Markov's inequality and the identity
\[
\int_{S^1} \|\Pi_\theta\mu \|_2^2\,d\theta \lesssim \cE_1(\mu).
\]
see e.g. \cite[Equation 1.7]{Mattila2004}.
\end{proof}

We restate  \cite[Lemma 3.7]{Shmerkin17} using our notation, for later reference.
\begin{lemma} \label{lem:discretize-project}
For any $\nu\in\cP(\R^2)$, $k\in\N$ and $\theta\in S^1$,
\[
\| R_k \Pi_\theta \nu \|_2^2 \approx \| \Pi_\theta R_k \nu \|_2^2.
\]
\end{lemma}

Next, we define the various sets of ``bad projections''.
\begin{definition} \label{def:bad-projections}
Given $\mu\in\cP([0,1)^2)$,  $x\in\supp_{\mathsf{d}}(\mu)$ and non-negative integers $j,k, j_0, \ell$, we let
\begin{align*}
\bad(\mu,x,j, k) &= \left\{ \theta\in S^1: \|\Pi_\theta \mu(x;j\sto j+k) \|_2^2 \ge 2^{\e T k} \cE_1(\mu(x;j\sto j+k)) \right\}, \\
\bad_{j_0\ssto\ell}(\mu,x) &= \bigcup \big\{ \bad(\mu,x,j,k) : k \ge \tau j, \,\, j_0\le j \le j+k\le \ell  \big\}.
\end{align*}
\end{definition}
We underline that the definition of $\bad_{j_0\ssto\ell}(\mu,x)$ depends on the parameters $T, \e$ and $\tau$. Note that, since $\mu(x;j\sto j+k)$ has a bounded density by definition, both quantities in the definition of $\bad(\mu,x,j, k)$ are finite.

Our next goal is to combine Lemma \ref{lem:marstrand} with the decomposition given by Corollary \ref{cor:bourgain}. Starting with a $2^{-T\ell}$-measure $\mu\in\cP([0,1)^2)$ and $x\in\supp_{\mathsf{d}}(\mu)$, we define
\begin{equation} \label{eq:def-badprime}
\bad'_{j_0\ssto\ell}(\mu,x) = \begin{cases}
 \bad_{j_0\ssto\ell}(\mu_{X_j},x)  & \text{ if }x\in X_j\\
 \varnothing & \text{ if } x\in \supp_{\mathsf{d}}(\mu)\setminus \bigcup_i X_i
\end{cases},
\end{equation}
where $(X_i)_{i=1}^N$ are the sets given by Corollary \ref{cor:bourgain}. Note that $\supp_{\mathsf{d}}(\mu_{X_j})=X_j$.

\begin{lemma} \label{lem:badx1}
There exists a further constant $\e'=\e'(T,\e,\tau)>0$ such that, for any $2^{-T\ell}$-measure $\mu\in\cP([0,1)^2)$,
\[
|\bad'_{\e \ell\ssto \ell}(\mu,x)|  \lesssim_{T,\e,\tau} 2^{-\e' \ell}  \quad\text{for all }x\in\supp_{\mathsf{d}}(\mu).
\]
\end{lemma}
\begin{proof}
According to the definitions and Lemma \ref{lem:marstrand}, for any $\nu\in\cP([0,1)^2)$ and $x\in\supp_{\mathsf{d}}(\nu)$,
\[
\left|\bad_{j_0\ssto\ell}(\nu,x)\right| \lesssim \sum_{j=j_0}^\infty \sum_{k=\lfloor \tau j\rfloor}^\infty 2^{-\e T k} \lesssim_{T,\e,\tau}  2^{-\e T \tau j_0} .
\]
The point here is that the bound does not depend on $\nu$ or $x$. Hence the claim follows with $\e'  = \e^2 T\tau$.
\end{proof}

Finally, if $\mu\in\cP([0,1)^2)$ and $x\in\supp_{\mathsf{d}}(\mu)$, we let
\begin{equation} \label{eq:def-badprimeprime}
\bad''_{\ell_0}(\mu,x) =\bigcup_{\ell=\ell_0}^\infty \bad'_{\e \ell\ssto\ell}(R_\ell\mu,x).
\end{equation}
We record the following immediate consequence of Lemma \ref{lem:badx1} for later use.
\begin{lemma} \label{lem:badx}
\[
|\bad''_{\ell_0}(\mu,x)| \lesssim_{T,\e,\tau} 2^{-\e' \ell_0},
\]
for all $x\in\supp_{\mathsf{d}}(\mu)$, where $\e'=\e'(T,\e,\tau)>0$ is the constant from Lemma \ref{lem:badx1}.
\end{lemma}

\subsection{Radial projections}
\label{subsec:radial-projections}

The following result was recently established by T. Orponen \cite{Orponen17}. We state it only in the plane. We denote the radial projection with center $y$ by $P_y$, i.e. $P_y(x)=(y-x)/|y-x|\in S^1$ is the (oriented) direction determined by $x$ and $y$.
\begin{prop} \label{prop:orponen}
Let $\mu,\nu\in\cP([0,1)^2)$ be measures with disjoint supports, such that $\cE_s(\mu)<\infty$, $\cE_u(\nu)<\infty$ for some $u>1$, $2-u<s<1$. Then there is $p=p(s,u)>1$ such that $P_x\nu$ is absolutely continuous with a density in $L^p(S^1)$ for $\mu$ almost all $x$. Moreover,
\[
\int \|P_x\nu\|_p^p \,d\mu(x) <\infty.
\]
\end{prop}
\begin{proof}
This is stated in \cite[Equation (3.5)]{Orponen17}, except that Orponen deals with weighted measures $\mu_y=|x-y|^{-1}d\mu$ instead of $\mu$ (note that the roles of $\mu$ and $\nu$ are interchanged in \cite{Orponen17}). Since the weight $|x-y|^{-1}$ is bounded away from $0$ and $\infty$ by the assumption that the supports of $\mu$ and $\nu$ are bounded and disjoint, the claim also holds for $\mu$.
\end{proof}

We point out that Proposition \ref{prop:orponen} uses the Fourier transform, and is the only point in the proofs of Theorems \ref{thm:main} and \ref{thm:packing} that does (on the other hand, the proof of Theorem \ref{thm:full-distance-set} relies heavily on the strongly Fourier-analytic approach of Mattila-Wolff).

Proposition \ref{prop:orponen} has the following key consequence. A similar statement was obtained in \cite{Shmerkin17} using a slightly more involved argument. We recall that $|\cdot|$ stands for normalized Lebesgue measure on the circle.

\begin{prop} \label{prop:radial}
Let $\mu,\nu\in\cP([0,1)^2)$ have disjoint supports and satisfy $\cE_s(\mu), \cE_u(\nu)<\infty$ for some $s\in (0,2), u>\max(1,2-s)$. Then there exists $\kappa=\kappa(\mu,\nu)>0$ such that the following holds:

Suppose that $\Theta\subset [0,1)^2\times S^1$ is a Borel set such that
\[
(\mu\times\mathcal{H}^1)(\Theta) \le \kappa.
\]
Then
\[
(\mu\times\nu)\{ (x,y): P_y(x) \not\in\Theta_x \} > \frac{2}{3}.
\]
\end{prop}
\begin{proof}
Since $\cE_s(\mu)<\infty$ implies that $\cE_{s'}(\mu)<\infty$ for all $s'<s$, we may assume that $s<1$. By Proposition \ref{prop:orponen}, there is $p>1$ such that
\[
\int \|P_x\nu\|_p^p \,d\mu(x) =: C <\infty.
\]
Denote $\Theta_x = \{ \theta\in S^1: (x,\theta)\in\Theta\}$
and $-\Theta_x=\{ -\theta:\theta\in\Theta_x\}$. Using Fubini and H\"{o}lder, each twice, we estimate
\begin{align*}
(\mu\times\nu)\{(x,y): P_y(x)\in\Theta_x\} &= \int P_x\nu(-\Theta_x)\,d\mu(x)\\
&\le \int \mathcal{H}^1(\Theta_x)^{1/p'} \|P_x\nu\|_p \,d\mu(x)\\
&\le \left(\int \mathcal{H}^1(\Theta_x) d\mu(x)\right)^{1/p'} \left(\int \|P_x\nu\|_p^p d\mu(x)\right)^{1/p}\\
&\le \kappa^{1/p'}  C^{1/p}.
\end{align*}
The claim follows by choosing $\kappa$ so that $\kappa^{1/p'} C^{1/p}\le 1/3$.
\end{proof}

\section{Box-counting estimates for pinned distance sets}
\label{sec:box-counting}

In this section we derive a lower bound on box-counting numbers of pinned distance sets that will be crucial in the proofs of Theorems \ref{thm:main} ,\ref{thm:packing} and \ref{thm:full-distance-set}. Our estimate will be in terms of a multiscale decomposition where, unlike previous works in the literature, we are allowed to choose the sequence of scales (depending on the set or measure for which we are seeking estimates). This additional flexibility will ultimately allow us to improve upon the easy bounds on the dimensions of distance sets.

To begin, we recall some basic facts about entropy. If $\nu\in\cP(\R^d$) and $\mathcal{A}$ is a finite partition of $\R^d$ (or of a set of full $\nu$-measure), then the entropy of $\nu$ with respect to $\mathcal{A}$ is given by
\[
H(\nu,\mathcal{A}) = -\sum_{A\in\mathcal{A}} \nu(A) \log(\nu(A)),
\]
with the usual convention $0\cdot \log 0=0$. It follows from the concavity of the logarithm that one always has
\[
H(\nu,\mathcal{A}) \le \log|\mathcal{A}|.
\]
Hence, a lower bound for $H(\nu,\cD_j)$ provides a lower bound for $\cN(A,j)$ if $A$ is a Borel set of full measure (recall that $\cN(A,j)$ denotes the number of elements in $\cD_j$ that intersect $A$). We will apply this when $\nu$ is supported on a pinned distance set. Although box-counting numbers in principle give bounds only for box dimension, together with standard mass pigeonholing arguments we will be able to get bounds also for Hausdorff and packing dimension.

The following proposition is the key device that will allow us to bound from below the entropy of pinned distance measures (and hence also the box-counting numbers of pinned distance sets). Roughly speaking, we bound the entropy of the projection of a measure $\mu$ under the pinned distance map by an average over both scales and space (the latter, weighted by $\mu$) of a quantity involving the $L^2$ norms of projected \emph{local} pinned distance measures. We emphasize that this method to bound the dimension of (linear or nonlinear) projections from below goes back in various forms to \cite{HochmanShmerkin12, Hochman14, Orponen17b}, although the use of projected $L^2$ norms (rather than projected entropies) was first used in \cite{Shmerkin17}.

Before stating the proposition we introduce some definitions. Given $L\in\N$, a \emph{good partition} of $(0,L]$ is an integer sequence $0=N_0<\ldots<N_q=L$ such that $N_{j+1}-N_j\le N_j+1$. We write $\Delta_y(x)=|x-y|$ for the pinned distance map, and $\theta(x,y)=P_y(x)=(x-y)/|x-y|$.

\begin{prop} \label{prop:entropy-of-pinned-dist-measures}
Let $\mu\in\cP([0,1)^d)$, let $y\in \R^d$ be at distance $\ge \e$ from $\supp(\mu)$, and fix a good partition $(N_i)_{i=0}^q$ of $(0,\ell]$. Then
\begin{equation} \label{eq:lower-bound-entropy}
 T\ell -  H(\Delta_y\mu,\cD_\ell) \le O_{T,\e}(q)  + \sum_{i=0}^{q-1} \sum_{Q\in\cD_{N_i}} \mu(Q)  \log\| \Pi_{\theta(x_Q,y)}\mu(Q;N_{i+1}) \|_2^2,
\end{equation}
where  $x_Q$ is an arbitrary point in $Q$.
\end{prop}
\begin{proof}
Write $D_i=N_{i+1}-N_i$. Note that our $\cD_i$ correspond to $\cD_{T i}$  and our $T N_i$ to $m_i$ in \cite{Shmerkin17}. Recall also that $\mu^Q$ denotes the magnification of $\mu_Q$ to the unit cube. It is shown in \cite[Proposition 3.8 and Remark 3.10]{Shmerkin17} that
\begin{equation} \label{eq:lower-bound-entropy-from-entropy}
H(\Delta_y\mu,\cD_\ell) \ge -O_{T,\e}(q) +\sum_{i=0}^{q-1} \sum_{Q\in\cD_{N_i}} \mu(Q) H\left(\Pi_{\theta(y,x_Q)} \mu^Q ,\cD_{D_i}\right).
\end{equation}

Applying Lemma \ref{lem:discretize-project} to $\nu = \mu^Q$ for some $Q\in\cD_{N_i}$ and $k=D_i$, we get that
\begin{equation} \label{eq:discretize-project}
\| R_{D_i} \Pi_{\theta(y,x_Q)} \mu^Q \|_2^2 \approx  \|\Pi_{\theta(y,x_Q)}  \mu(Q;N_{i+1}) \|_2^2.
\end{equation}
On the other hand, a simple convexity argument (see \cite[Lemma 3.6]{Shmerkin17}) yields that, for any $\nu\in\cP(\R)$ and $k\in\N$,
\[
H(\nu,\cD_k) \ge  Tk - \log\|R_k \nu\|_2^2.
\]
Applying this with $k=D_i$ and $\nu = \Pi_{\theta(y,x_Q)} \mu^Q$, and recalling \eqref{eq:discretize-project}, we deduce that
\[
  H\left(\Pi_{\theta(y,x_Q)} \mu^Q ,\cD_{D_i}\right) \ge   T D_i - \log\|\Pi_{\theta(y,x_Q)} \mu(Q;N_{i+1}) \|_2^2 - O(1).
\]
Using this bound in each term in the right-hand side of \eqref{eq:lower-bound-entropy-from-entropy}, and absorbing the sum of the $q$ $O(1)$ terms into $O_{T,\e}(q)$, we get the claim.
\end{proof}

We remark that the assumption that $N_{j+1}-N_j\le N_j+1$ in the definition of good partition (which will play a crucial role later) arises from the linearization of the distance function, and cannot be substantially weakened. The key advantage of having $L^2$ norms instead of entropies in this proposition is that the estimate one gets is robust under passing to subsets of moderately large measure:
\begin{prop} \label{prop:good-bound-from-below}
With the assumptions and notation from Proposition \ref{prop:entropy-of-pinned-dist-measures}, let us write $\mathcal{F}(\mu)$ for the right-hand side of \eqref{eq:lower-bound-entropy} (we assume $y$ and the partition $(N_i)$ are fixed). If $\mu\in\cP([0,1)^2)$, $\nu=\mu_A$ where $A$ is Borel and $\mu(A)>0$,  then
\[
\mathcal{F}(\nu) \le O_{T,\e}(q) + 2q\log\left(\tfrac{T\ell}{\mu(A)}\right)  + \sum_{i=0}^{q-1} \sum_{Q\in\cD_{N_i}} \nu(Q)  \log\|\Pi_{\theta(y,x_Q)}\mu(Q;N_{i+1}) \|_2^2.
\]
\end{prop}
\begin{proof}
We start with the trivial observation that if $\rho,\rho'\in\cP(\R^d)$ have an $L^2$ density and $\rho'(S)\le K\rho(S)$ for all Borel sets $S$, then the same bound transfers over to the densities for a.e. point, and so $\|\rho'\|_2^2 \le K^2 \|\rho\|_2^2$.

Let $\zeta=1/(T\ell)\in (0,1)$. Fix $i\in \{0,\ldots,q-1\}$, and note that
\begin{equation} \label{eq:low-density-has-small-measure}
\sum\{  \nu(Q): Q\in\cD_{N_i}, \nu(Q) < \zeta \mu(Q)\} < \zeta.
\end{equation}
Suppose $\nu(Q) = \mu(A\cap Q)/\mu(A) \ge \zeta\mu(Q)>0$ for a given $Q\in\cD_{N_i}$. Then
\[
\nu_Q(S) = \frac{\mu(A\cap Q\cap S)}{\mu(A\cap Q)}  \le \frac{\mu(Q\cap S)}{\zeta\mu(A)\mu(Q)} = \frac{1}{\zeta\mu(A)} \mu_Q(S)
\]
for any Borel set $S\subset [0,1)^2$. This domination is preserved under push-forwards and the action of $R_{D_i}$ (where as before $D_i = N_{i+1}-N_i$), so in light of our initial observation we get
\[
\| \Pi_{\theta(y,x_Q)} \nu(Q;N_{i+1})  \|_2^2 \le \frac{1}{(\zeta\mu(A))^2} \| \Pi_{\theta(y,x_Q)} \mu(Q;N_{i+1}) \|_2^2 ,
\]
always assuming that $\nu(Q)\ge \zeta\mu(Q)>0$ and $Q\in \cD_{N_i}$. Also, since the measure $\Pi_{\theta(y,x_Q)} \mu(Q;N_{i+1})$ is supported on an interval of length $\sqrt{2}$, it follows from Cauchy-Schwarz that
\begin{equation} \label{eq:L2-norm-not-too-small}
\| \Pi_{\theta(y,x_Q)} \mu(Q;N_{i+1}) \|_2^2 \ge 2^{-1/2}.
\end{equation}

On the other hand, for any $2^{-T D}$-measure $\rho$ on $\R$ one has $\|\rho\|_2^2 \le 2^{T D}$. In light of Lemma \ref{lem:discretize-project}, this implies that
\begin{equation} \label{eq:trivial-bound-L2-norm}
\|\Pi_{\theta(y,x_Q)}\nu(Q;N_{i+1}) \|_2^2 \lesssim 2^{T D_i}.
\end{equation}
Splitting (for each $i$) the sum $\sum_{Q\in\cD_{N_i}}$ in Proposition \ref{prop:entropy-of-pinned-dist-measures} into the cubes with $\nu(Q) \ge \zeta\mu(Q)$ and $\nu(Q)< \zeta\mu(Q)$, and recalling \eqref{eq:low-density-has-small-measure}, we arrive at the estimate
\[
\mathcal{F}(\nu) \le O_{T,\e}(q) + \zeta T\ell - 2q\log(\zeta\mu(A))  + \sum_{i=0}^{q-1} \sum_{\substack{Q\in\cD_{N_i} \\ \nu(Q)\ge \zeta\mu(Q)}}
  \nu(Q)   \log\|\Pi_{\theta(y,x_Q)}\mu(Q;N_{i+1}) \|_2^2,
\]
where we merged the sum of the ($\log$ of the) implicit constants in \eqref{eq:trivial-bound-L2-norm} into $O_{T,\e}(q)$. Recalling that $\zeta=1/(T\ell)$  and using  \eqref{eq:L2-norm-not-too-small} we get the desired result.
\end{proof}

Our next goal is to get a simpler lower bound in the context of Proposition \ref{prop:good-bound-from-below} when $\mu$ is $\sigma$-regular (recall Definition \ref{def:regular}), and $\nu$ is the restriction of $\mu$ to the set of points which are not bad in the sense of \S\ref{subsec:badprojections}. Combining the results of \S\ref{subsec:badprojections} and \S\ref{subsec:radial-projections}, we will later be able to deal with general measures via a reduction to this special case.

We require some additional definitions:
\begin{definition} \label{def:integer-partition}
We say that $0=N_0<N_1<\ldots<N_q=L$ is a \emph{$\tau$-good partition} of $(0,L]$ if
\begin{equation}\label{e:integertau}
\tau N_j \le N_{j+1}-N_j \le N_j+1
\end{equation}
for every $0\le j< q$. In other words $(N_j)$ is a good partition and additionally $N_{j+1}\ge (1+\tau) N_j$.

Given a finite sequence $(\sigma_1,\ldots,\sigma_L)\in \R^L$, let
\[
\mathcal{S}(\sigma) = -\min_{j=0}^L \sigma_1+\cdots+\sigma_j \ge 0.
\]

For any good partition $\mathcal{P}=(N_j)_{j=0}^q$ of $(0,L]$ and any $\sigma \in \R^L$, we denote
\[
\M(\sigma,\mathcal{P}) = \sum_{j=0}^{q-1} \mathcal{S}(\sigma|(N_j,N_{j+1}]),
\]
where $\sigma|I$ denotes the restriction of the sequence $\sigma$ to the interval $I$.

Finally, given $\sigma\in \R^L$ and $\tau\in (0,1)$, we let
\[
\M_\tau(\sigma) = \min \{ \M(\sigma,\mathcal{P}) : \mathcal{P} \text{ is a $\tau$-good partition of } (0,L]\}.
\]
\end{definition}

Recall that $o_{T,\e}(1)$ denotes a function of $T$ and $\e$ which tends to $0$ as $T\to\infty,\e\to 0^+$.
\begin{prop} \label{prop:lower-bound-box-counting}
Suppose that $\rho\in\cP([0,1)^2)$ is a $(\sigma_1,\ldots,\sigma_\ell)$-regular measure. Assume that there are a Borel set $A\subset [0,1)^2$, a point $y\in\R^2$ and a number $\beta\in (0,1)$ satisfying that $\rho(A)>0$, $\dist(y,\supp(\rho))\ge \e$, and for all $x\in A\cap \supp_{\mathsf{d}}(\rho)$ there is $\wt{x} \in\cD_\ell(x)$ such that
\[
\theta(\wt{x},y) \notin \bad_{\beta\ell\ssto\ell}(\rho,\wt{x}).
\]
Then
\[
\frac{\log \cN(\Delta_y A,\ell)}{T\ell} \ge 1 - \frac{\M_\tau(\sigma)}{\ell} -  \error.
\]
where
\[
\error = 2\beta + o_{T,\e}(1) +  O_{T,\e,\tau}\left(\frac{\log^2\ell}{\ell}\right) + \frac{O_\tau(\log\ell)\log(1/\rho(A))}{\ell}.
\]
\end{prop}
\begin{proof}
Let $\cP=(N_i)_{i=0}^q$ be a $\tau$-good partition of $(0,\ell]$. We have to show that
\[
\frac{\log \cN(\Delta_y A,\ell)}{T\ell} \ge 1 - \frac{\M(\sigma,\mathcal{P})}{\ell}  - \error .
\]

Fix $i_0$ as the smallest value of $i$ such that $N_i\ge \beta\ell$, and note that $N_{i_0}<  2\beta\ell+1$.

Let us rewrite the inequality from Proposition \ref{prop:good-bound-from-below} applied to $\rho$ and $\rho_A$ in the form
\[
\mathcal{F}(\rho_A) \le E + \Sigma_{\text{I}} + \Sigma_{\text{II}},
\]
where
\begin{align*}
E &=  O_{T,\e}(q) + 2q\log\left(\tfrac{T\ell}{\rho(A)}\right),\\
\Sigma_{\text{I}} &=  \sum_{i=0}^{i_0-1} \sum_{Q\in\cD_{N_i}: \rho(A\cap Q)>0} \rho_A(Q)  \log\| \Pi_{\theta(y,x_Q)}\rho(Q;N_{i+1}) \|_2^2,\\
\Sigma_{\text{II}} &=  \sum_{i=i_0}^{q-1} \sum_{Q\in\cD_{N_i}: \rho(A\cap Q)>0} \rho_A(Q)  \log\| \Pi_{\theta(y,x_Q)}\rho(Q;N_{i+1}) \|_2^2,
\end{align*}
where $x_Q$ are arbitrary points in $Q$. By assumption, we may choose these points so that
\begin{equation} \label{eq:x_Q-good-set}
\theta(x_Q,y)\notin \bad_{\beta\ell\ssto\ell}(\rho,x_Q).
\end{equation}
Using that  $(1+\tau)^q\le \ell$, we bound
\begin{equation} \label{eq:bound-E}
E \le O_{T,\e,\tau}(\log^2\ell)+ O_\tau(\log\ell) \log(1/\rho(A)).
\end{equation}
Write $D_i=N_{i+1}-N_i$. To estimate $\Sigma_{\text{I}}$, we use the trivial bound $\|R_{D_i}(\cdot)\|_2^2 \le 2^{D_i T}$ together with Lemma \ref{lem:discretize-project} and the bounds $N_{i_0}<2\beta\ell +1$, $(1+\tau)^q \le \ell$, so that
\begin{align}
\label{eq:bound-Sigma-I}
\begin{split}
\Sigma_{\text{I}} &\le \sum_{i=0}^{i_0-1} \sum_{Q\in\cD_{N_i}} \rho_A(Q) (D_i T + O(1)) \\
&\le N_{i_0} T + O(i_0) \\
&\le 2\beta T\ell  + O_{T,\tau}(\log\ell).
\end{split}
\end{align}
Now, to estimate the main term $\Sigma_{\text{II}}$, we need to go back to Definition \ref{def:bad-projections}. By \eqref{eq:x_Q-good-set}, and using that $\cP$ is a $\tau$-good partition of $(0,\ell]$, we have $\theta(x_Q,y)\notin \bad(\rho,x_Q,N_i,D_i)$ for $i_0 \le i < q$. We deduce that
\[
\log\| \Pi_{\theta(x_Q,y)}\rho(Q;N_{i+1}) \|_2^2 \le \e T D_i +\log \cE_1(\rho(Q;N_{i+1})),
\]
for $i_0 \le i < q$. On the other hand,  by the assumption that $\rho$ is $(\sigma_1,\ldots,\sigma_\ell)$-regular, and since $\cP$ is a good partition of $(0,\ell]$, the measure  $\rho(Q;N_{i+1})$ is $(\sigma_{N_i+1},\ldots,\sigma_{N_{i+1}})$-regular. Hence, using Lemma \ref{lem:energy-regular}, we obtain
\[
\log \cE_1(\rho(Q;N_{i+1})) \le O(D_i) + O_T(1)  + T \mathcal{S}(\sigma|(N_i,N_{i+1}]).
\]
Combining the last two displayed formulas, we deduce that
\[
\log \|\Pi_{\theta(x_Q,y)} \rho(Q;N_{i+1}) \|_2^2 \le o_{T,\e}(1) T D_i+ O_T(1) +  T \mathcal{S}(\sigma|(N_i,N_{i+1}]).
\]
Adding up from $i=i_0$ to $q-1$ and again using $q= O_{\tau}(\log\ell)$, we get
\begin{equation}  \label{eq:bound-Sigma-II}
\Sigma_{\text{II}} \le o_{T,\e}(1) T \ell + O_{T,\tau}(\log\ell)+ T \M(\sigma,\cP).
\end{equation}
Combining \eqref{eq:bound-E}, \eqref{eq:bound-Sigma-I} and \eqref{eq:bound-Sigma-II}, we conclude that
\[
\frac{1}{T\ell}\mathcal{F}(\rho_A) \le  \frac{1}{\ell}\M(\sigma,\cP) + \error,
\]
where $\error$ is as in the statement. Recall that $\mathcal{F}(\mu)$ denotes the right-hand side of \eqref{eq:lower-bound-entropy} in Proposition \ref{prop:entropy-of-pinned-dist-measures}. Now Proposition \ref{prop:entropy-of-pinned-dist-measures} guarantees that
\[
\frac{1}{T\ell} H(\Delta_y\rho_A,\cD_\ell) \ge 1 -  \frac{1}{\ell}\M(\sigma,\cP) - \error.
\]
Since $H(\mu,\mathcal{A}) \le \log|\mathcal{A}|$ for any finite Borel partition $\mathcal{A}$ of a set of full $\mu$-measure, this finishes the proof.
\end{proof}

Note that in this proposition, the sequence $\sigma$ depends on the measure $\rho$ and the bound is in terms of $M_\tau(\sigma)$ (we will be able to make the error term arbitrarily small). Thus we are led to the combinatorial problem of minimizing $\M(\sigma,(N_i))$ over all $\tau$-good partitions for a given $\sigma\in [-1,1]^\ell$. This problem will be tackled in the next section: see Proposition \ref{prop:combinatorial}, and also Proposition \ref{prop:combinatorial-packing} for the case in which we are allowed to restrict $\sigma$ to $(0,L]$ for some large $L$.

\section{Finding good scale decompositions: combinatorial estimates}
\label{sec:combinatorial}

\subsection{An optimization problem for Lipschitz functions}

We begin by defining suitable analogs of the concepts from Definition \ref{def:integer-partition} for Lipschitz functions, instead of $[-1,1]$-sequences.

\begin{definition}\label{d:many}
A sequence $(a_n)_{n=0}^\infty$ is a \emph{partition} of
the interval $[0,a]$ if $a=a_0>a_1>\ldots>0$ and $a_n\to 0$;
it is a \emph{good partition} if
we also have $a_{k-1} / a_{k} \le 2$ for every $k\ge 1$.

A sequence $(a_n)_{n=0}^{\infty}$ is a \emph{$\tau$-good partition} for a given $0< \tau <1$
if it is a good partition and we also have $a_{k-1} / a_{k} \ge 1+\tau$ for every $k\ge 1$.

  Let $f:[0,a]\to\R$ be continuous and $(a_n)$ be a  partition of $[0,a]$. By the
  \emph{total drop of} $f$ \emph{according to} $(a_n)$ we mean
$$
\T(f,(a_n))=\sum_{n=1}^\infty f(a_n)-\min_{[a_n,a_{n-1}]} f,
$$
and we also introduce the notation
$$
\T(f)=\inf\{\ \T(f,(a_n))\ :\ (a_n) \textrm{ is a good partition of } [0,a]\ \},
$$
$$
\T_\tau(f)=\inf\{\ \T(f,(a_n))\ :\ (a_n)
\textrm{ is a $\tau$-good partition of } [0,a]\ \}.
$$

We call the interval $[a_{n},a_{n-1}]$ \emph{increasing}
if $\min_{[a_n,a_{n-1}]} f=f(a_n)$ and \emph{decreasing} if
$\min_{[a_n,a_{n-1}]} f=f(a_{n-1})$. (Note that $f$ needs not be increasing or decreasing on $[a_n,a_{n-1}]$.)
\end{definition}

In this section we investigate the following question: given a $1$-Lipschitz function $f:[0,a]\to\R$ satisfying certain bounds,
how large can $\T(f)$ and $\T_\tau(f)$ be?

First we study $\T(f)$.
Later we show (see Corollary~\ref{c:tau})
that for small $\tau$ the quantities $\T(f)$ and
$\T_\tau(f)$ are close.
Finally, from the bounds on $T_\tau(f)$ we deduce corresponding bounds  on $\M_\tau(\sigma)$: see for example Proposition \ref{prop:combinatorial}. Hence this problem is closely related to that of minimizing the dimension loss when estimating the dimension of the pinned distance set via
 Proposition \ref{prop:lower-bound-box-counting}. Dealing first with Lipschitz functions rather than $[-1,1]$-sequences allows us to avoid certain technicalities and make the arguments more transparent.

The basic result is the following.

\begin{prop}\label{p:basic}
  Let $a>0$,
$-1\le D < C \le 1$
be given parameters such that $C\ge 2D$.
  Let $f:[0,a]\to\R$ be a $1$-Lipschitz function such that
  $Dx\le f(x)\le Cx$ for every $x\in[0,a]$.
Then
\begin{equation}\label{e:basict}
    \T(f)\le \frac{(a-f(a))(C-2D)}{1+2C-3D}
    \le  a\cdot\frac{(1-D)(C-2D)}{1+2C-3D}.
    \end{equation}

\end{prop}

\begin{proof}
Since $f(a)\ge Da$ and $a>0$,
the second inequality of
\eqref{e:basict} is clear, so it enough to
prove the first inequality.

Let
$$
h=\frac{C-2D}{1+C-D} \quad \textrm{and} \quad
\rho=\frac{C-2D}{1+2C-3D}.
$$
Note that
\begin{equation}\label{e:hrho}
h=\frac{\rho}{1-\rho} \quad \textrm{and} \quad \rho=\frac{h}{h+1}
\end{equation}
and $h,\rho\ge 0$ since we assumed $C\ge 2D$ and $C\ge D$, so $2C\ge 3D$.

We will construct a good partition $(a_n)$ with the following two extra properties:

(*) every interval $[a_n,a_{n-1}]$ ($n=1,2,\ldots$) is either increasing or decreasing (recall Definition~\ref{d:many}), and

(**) if $[a_k,a_{k-1}],\ldots,[a_{l+1},a_l]$ ($ k \ge l+1 \ge 1$)
is a maximal block of consecutive decreasing intervals, then
\[
\frac{f(a_k)-f(a_l)}{a_l-a_k} \le h=\frac{C-2D}{1+C-D}.
\]

First we show that this is enough to prove
our claim.
Let $a=a'_0>a'_1>\ldots$ be the endpoints of the union of each maximal block of consecutive intervals of the same type (increasing or decreasing).
It easily follows from the definitions and telescoping that
$\T(f,(a_n))=\T(f,(a'_k))$.
Hence to obtain \eqref{e:basict} it is enough to
prove
\begin{equation}\label{e:Tona'_k}
\T(f,(a'_k)) \le \rho\cdot(a-f(a)).
\end{equation}

We claim that
\begin{equation}\label{e:f - min}
f(a'_k)-\min_{[a'_k,a'_{k-1}]} f \le
\rho((a'_{k-1}-f(a'_{k-1})) - (a'_k-f(a'_k))) \qquad (k=1,2,\ldots).
\end{equation}

Indeed, by construction, the interval $[a'_{k},a'_{k-1}]$ is either increasing or decreasing.
If it is increasing then
$$
f(a'_k)-\min_{[a'_k,a'_{k-1}]} f = 0 \le \rho((a'_{k-1}-f(a'_{k-1})) - (a'_k-f(a'_k)))
$$
since $f$ is $1$-Lipschitz and $a'_k<a'_{k-1}$.

If  $[a'_{k+1},a'_{k}]$ is decreasing then, using first (**) and the fact that $\rho<1$, and then \eqref{e:hrho}, we get
\begin{align*}
f(a'_k)-\min_{[a'_k,a'_{k-1}]} f &=
f(a'_{k})-f(a'_{k-1}) \\
&\le
\rho(f(a'_{k})-f(a'_{k-1})) + (1-\rho)h(a'_{k-1}-a'_{k})
\\
&= \rho(f(a'_{k})-f(a'_{k-1})) + \rho (a'_{k-1}-a'_{k}),
\end{align*}
which completes the proof of \eqref{e:f - min}.

By adding up \eqref{e:f - min} for $k=1,2,\ldots$ and
using that $a'_0=a$, $a'_k\to 0$ and $f(a'_k)\to 0$
we get \eqref{e:Tona'_k}, which implies \eqref{e:basict}.

Therefore it is enough to construct a good 
partition $(a_n)$ with properties (*) and (**).
Let $a_0=a$ and suppose that $a_0>\ldots>a_{n}>0$ are already constructed
with properties (*) and (**) (up to $n$).

We distinguish three cases.

\noindent\emph{Case 1.} $\min_{[a_n/2,a_n]} f < f(a_n)$.

In this case let $a_{n+1}\in [a_n/2,a_n]$ be the smallest
number such that
$f(a_{n+1})=\min_{[a_n/2,a_n]}f$. Then $[a_{n+1},a_n]$
is an increasing interval and so (*) and (**) still hold and we can continue the procedure.

\noindent\emph{Case 2.} $\min_{[a_n/2,a_n]} f = f(a_n)$ and
$f(a_n/2)-f(a_n)\le h\cdot (a_n - a_n/2)$.

In this case let $a_{n+1}=a_n/2$, and again (*), (**) hold for the extended sequence and we can continue the procedure.

\noindent\emph{Case 3.} $\min_{[a_n/2,a_n]} f = f(a_n)$ and
$f(a_n/2)-f(a_n)> h\cdot (a_n - a_n/2)$.

First we claim that $h\ge -D$. Indeed, since $-1\le D\le C$ we have
$$
0\le (C-D)(D+1)=-D+CD-D^2+C,
$$
which  implies that
$$
-D(1+C-D) \le C-2D ,
$$
and this implies $h\ge -D$.

Since $h\ge -D$ and $f(x)\ge Dx$ we have
$f(a_n)\ge Da_n\ge -ha_n$ and so
$$
f(0)-f(a_n)=-f(a_n)\le h a_n = h (a_n-0).
$$

This and the assumption
$f(a_n/2)-f(a_n)> h\cdot (a_n - a_n/2)$
implies that there exists
a largest $b\in[0,a_n/2)$ be such that
\begin{equation}\label{e:b}
f(b)-f(a_n)=h(a_n-b).
\end{equation}

Now our goal is to find a sequence $b=b_0<b_1<\ldots<b_M=a_n$ with $M\ge 2$
such that
\begin{equation}\label{e:b_i}
\min_{[b_{i-1},b_i]}f=f(b_i),\ b_i/b_{i-1}\le 2\ (i=1,\ldots, M),
\quad b_i/b_{i-2}\ge 2\ (i=2,\ldots,M).
\end{equation}

The sequence $(b_i)$ is constructed by induction.
Let $b_0=b$.
Suppose that $m\ge 0$, $b=b_0<\ldots<b_m< a_n$
are already constructed and \eqref{e:b_i} holds for $M=m$.
If $b_m\ge a_n/2$ then we can take $b_{m+1}=a_n$ and $M=m+1$.
Then the construction is completed and \eqref{e:b_i} holds.

Now consider the case $b_m< a_n/2$.
Let $b_{m+1}\in[b_m,2b_m]$ be maximal
such that $f(b_{m+1})=\min_{[b_m,2b_m]}f$.
Our goal is to show that $b_{m+1}>b_m$.
For this it is enough to show that $f(2b_m)\le f(b_m)$.

Using that $b$ is the largest number in $[0,a_n/2]$
for which \eqref{e:b} holds,
$b_m\ge b$ and  $f(a_n/2)-f(a_n)> h\cdot (a_n - a_n/2)$,
we get
\begin{equation}\label{e:b_m}
f(b_m)-f(a_n)\ge h(a_n-b_m).
\end{equation}

Hence to get
$f(2b_m)\le f(b_m)$ it is enough to show that
\begin{equation}\label{e:2b_m}
f(2b_m)-f(a_n)\le h(a_n-b_m).
\end{equation}
Using \eqref{e:b_m} and $Dx\le f(x)\le Cx$ we get
$$
h(a_n-b_m)\le f(b_m)-f(a_n)\le Cb_m-Da_n,
$$
which implies that
$$
(D+h)a_n\le (C+h)b_m.
$$
Direct calculation shows that $D+h=(C-D)(1-h)$ and $C+h=(C-D)(2-h)$.
Thus the last inequality
 and $D<C$ imply
that
$$
(1-h)a_n\le (2-h)b_m.
$$
Hence, using also that $f$ is $1$-Lipschitz
and $b_m<a_n/2$, we obtain
$$
h(a_n-b_m)\ge a_n - 2b_m \ge f(2b_m)-f(a_n).
$$
This completes the proof of \eqref{e:2b_m} and so also the proof
of $b_{m+1}>b_m$.
It is easy to see that \eqref{e:b_i} holds for $M=m+1$.
Note also that the property $b_i/b_{i-2}\ge 2$ implies that
the construction of the sequence $(b_i)$ is completed after
finitely many steps.

Now, to finish Case 3 we take $a_{n+j}=b_{M-j}$ for $j=1,\ldots,M$.
Then (*) and (**) hold (up to $n+m$) and so the procedure can be continued.

This way we obtain a sequence
$a=a_0>a_1>\ldots>0$ that forms
a good partition with (*) and (**), provided $a_n\to 0$.
Therefore it remains to prove that $a_n\to 0$.

Since $a_{n+1}=a_{n}/2$ when Case 2 is applied and
$a_{n+M}=b_0=b\le a_n/2$ in Case 3, we are done if Case 2 or Case 3
is applied infinitely many times.
It is easy
to see that if both $a_{n+1}$ and $a_{n+2}$ were
obtained from Case 1, then we have $a_{n}/a_{n+2}\ge 2$.
Thus $a_n\to 0$, which completes the proof.
\end{proof}

\subsection{Small drop on initial segments}
\label{subsec:initial}
The results in this subsection are required in the proof of Theorem \ref{thm:packing}. We  aim to minimize $\T(f|[0,u])/u$, where $u>0$ is a new parameter that we are allowed to choose, subject to not being too small. The analysis will be strongly based on the study of hard points which we now define:
\begin{definition} \label{def-hardpoint}
If $f:[0,a]\to\R$ is a function, we say that $p\in[0,a]$ is a \emph{hard point of} $f$ if $\min_{[p/2,p]} f=f(p)$.
\end{definition}

We will say that a function $f$ defined on an interval $I$ is \emph{piecewise linear} if $I$ can be decomposed
into finitely many intervals such that $f$ is linear on each of them.
\begin{lemma}\label{l:hardpoints}
Let $f:[0,a]\to\R$ be a $1$-Lipschitz function,
which is piecewise linear on every closed subinterval of $(0,a]$. Then:

(i) The set of hard points of $f$ can be written as a
(possibly empty) finite or infinite union of closed
(possibly degenerate) intervals
$H = \cup_j [u_j,v_j]$
such that
$v_1\ge u_1>v_2\ge u_2>\ldots$ and
every closed subinterval of $(0,a]$
intersects only finitely many $[u_j,v_j]$.

(ii) We have
\begin{equation}\label{e:formula}
\T(f)=\sum_{j} f(u_j)-f(v_j),
\end{equation}
where the empty sum is meant to be zero.
\end{lemma}

\begin{proof}
The first statement is easy,
using that $f$ is piecewise linear.

First we prove $\ge$ in \eqref{e:formula}.
Let
 $(a_n)$ be a good partition of $[0,a]$ and let
$a=a'_0>a'_1>\ldots$ be an ordered enumeration
of the set $\{a_n\} \cup \{u_j\}
\cup \{v_j\}$. It is easy to check that $(a'_n)$
is also a good partition of $[0,a]$, and that by inserting a hard point of
$f$ into a good partition $(a_n)$, the value of $\T(f,(a_n))$ is
not changed. Thus $\T(f,(a'_n))=\T(f,(a_n))$.
Now every $[u_j,v_j]$ is of the form
$[u_j,v_j]=\cup_{n=n_j}^{m_j} [a'_{n},a'_{n-1}]$.
Since $f$ must be nonincreasing on
any interval $[u_j,v_j]$ we obtain
$$
f(u_j)-f(v_j)=
\sum_{n=n_j}^{m_j} f(a'_n)-\min_{[a'_n,a'_{n-1}]} f
$$
for every $j$.
Adding up, and using that
$f(a'_n)-\min_{[a'_n,a'_{n-1}]} f\ge 0$ and
$\T(f,(a'_n))=\T(f,(a_n))$ we get the claim.

To prove the other inequality we
construct by induction a good partition of $[0,a]$
such that $\T(f,(a_n))\le  \sum_{j} f(u_j)-f(v_j)$.
Let $a_0=a$.
Suppose that $a_0,\ldots,a_n$ are already defined.

\emph{Case 1.} If $a_n\in (u_j,v_j]$ for some $j$ then choose $k\ge 1$ and
$a_n>a_{n+1}>\ldots>a_{n+k}=u_j$ so that
$a_{n+i}/a_{n+i-1}\le 2$ for $i=1,\ldots,k$.

\emph{Case 2.} Otherwise let $a_{n+1}\in[a_n/2,a_n]$ be the
smallest number for which
$f(a_{n+1})=\min_{[a_n/2,a_n]}f$.
We claim that $a_{n+1}<a_n$.
If $a_n\not\in H$ then this is clear from the definition.
Since the only points of $H$ that are not handled
in the previous case are the left endpoints of the
intervals $[u_j,v_j]$ we can suppose that $a_n=u_j$
for some $j$.
By the piecewise linearity of $f$, there exists  $w\in (u_j/2, u_j)$
such that $f$ is linear on $[w,u_j]$ and
$w>v_{j+1}$.
Since $u_j$ is a hard point, $f$ cannot be increasing
on $[w,u_j]$. If $f$ is constant on $[w,u_j]$
then $a_{n+1}\le w<u_j=a_n$, so we are done.
So we can suppose that $f$ is decreasing on $[w,u_j]$.
Since $w>v_{j+1}$, every  $x\in [w,u_j)$ is
not hard, so there exists an $x'\in[x/2,x)$ such that
$f(x')<f(x)$.
Since $f$ is decreasing on $[w,u_j]$, $x'<w$.
By the continuity of $f$, this implies that there exists
$x_0\in[u_j/2,w]$ such that $f(x_0)\le f(u_j)$. Thus indeed $a_{n+1}<u_j=a_n$.

Note that if Case 2 was applied to obtain both $a_{n+1}$ and
$a_{n+2}$ then $a_n/a_{n+2}\ge 2$.
This implies that $a_n\to 0$, so $(a_n)$ is a good
partition of $[0,a]$. It remains to show that
$\T(f,(a_n))\le \sum_j f(u_j)-f(v_j)$.

If $a_n$ was obtained in Case 1 then
$[a_n,a_{n-1}]$ is a subinterval of some $[u_j,v_j]$ and
$f(a_n)-\min_{[a_n,a_{n-1}]} f=f(a_n)-f(a_{n-1})$.
If $a_n$ was obtained in Case 2 then
$f(a_n)-\min_{[a_n,a_{n-1}]} f=0$.
Note also that $f$ is nonincreasing on each $[u_j,v_j]$
since all points of $[u_j,v_j]$ are hard points of $f$.
These show that indeed $\T(f,(a_n))\le \sum_j f(u_j)-f(v_j)$,
which completes the proof.
\end{proof}

The next proposition (or rather, the discrete corollary given in Proposition \ref{prop:combinatorial-packing} below) will be crucial to get estimates on the packing dimension of the pinned distance sets.

\begin{prop}\label{p:initial}
  Let $a>0$ and  $0\le D < 1/2$ be given parameters.
  Let $f:[0,a]\to\R$ be a $1$-Lipschitz function,
  which is piecewise linear on every closed subinterval of $(0,a]$,
  and suppose that
 $f(0)=0$ and
  $Dx\le f(x)$ for every $x\in[0,a]$. Let
$$
\Phi(D)=\frac{2-D-\sqrt{3-3D^2}}{4}.
$$

Then for every $\delta\in(0,1/2)$
there exists  $u\in[3a\Phi(D)2^{-1/\delta},a]$
such that
   \begin{equation}\label{e:initiale}
   \T(f|[0,u]) < u\cdot (\Phi(D) +
    \delta(2-4\log\delta)). 
    \end{equation}

\end{prop}

\begin{proof}
Let $H\su[0,a]$ be the set of hard points of $f$.
If $H=\emptyset$ then by Lemma~\ref{l:hardpoints},
$\T(f)=0$, so $u=a$ is clearly a good choice in
this case. So suppose that $H$ is nonempty.

First we briefly explain the idea of the proof in this
nontrivial case. For simplicity, suppose that $a=1$ and
$D=0$, which is the most interesting case anyway.
Assume that the maximum of $f(x)/x$ on  $H\cap (0,1]$ exists and is attained at $u$, and let $B$ be this
maximum. Since $u$ is a hard point, $f(x)\ge f(u)=Bu$
on $[u/2,u]$, and a calculation using that $f$ is $1$-Lipschitz shows that
that
\begin{equation} \label{eq:sketch}
f(x)> Bx \quad \text{if} \quad u'<x<u,
\quad\text{where }u'=\frac{1/2-B}{1-B}u.
\end{equation}
Let $F(x)=\min(f(x),2B x)$. Then it is not hard to show
(see below for details) that every
$p\in H\cap [0,u]$ is also a hard point of $F$
and that $F=f$ on $H\cap [0,u]$.
By Lemma~\ref{l:hardpoints} this implies that
$\T(f|[0,u])=\T(F|[0,u])$, so we can study
$F|[0,u]$ instead of $f|[0,u]$.
Let $v$ be the largest number in $[0,u)$ such that
$F(v)=Bv$. It follows from \eqref{eq:sketch} that also
$F(x)> Bx$ if $u'<x<u$,
so we must have $v\le u'$, and hence
\[
v-F(v)=v(1-B)\le u(1/2-B)
\]
and $F(x)>Bx$ on $(v,u)$. Since $F(x)\le 2Bx$,
for any hard point $y$ of $F$ we must have
$F(y)\le By$, and this implies that $F$ has no
hard point in $(v,u)$.
By Lemma~\ref{l:hardpoints} this implies that
$\T(F|[0,u])=\T(F|[0,v])$.
Again using that $F(x)\le 2Bx$, we can apply Proposition~\ref{p:basic}
on $[0,v]$ to obtain
$$
\T(f|[0,u])=\T(F|[0,u])=\T(F|[0,v])\le
\frac{(v-F(v))(2B-0)}{1+2\cdot 2B - 3\cdot 0} \le
\frac{u(1/2-B)2B}{1+4B}.
$$
Calculus shows that
$\frac{(1/2-B)2B}{1+4B}\le\frac{2-\sqrt 3}4=\Phi(0)$ for $B\in [0,1]$,
so we obtain $\T(f|[0,u])\le u\Phi(0)$.

Unfortunately, $f(x)/x$ may not have
a maximum on $H\cap (0,1]$ and, even if it does, we might get an $u$ which is too small.
To avoid these problems we replace $f(x)/x$
by $f(x)/x +\delta \log x$. Then we can show that $u$ exists, is not too small, and it still satisfies the claim of the proposition.

We now continue with the actual proof. Note that $H$ is a closed set, and let $h=\max H$.
By Lemma~\ref{l:hardpoints},
$\T(f)=\T(f|[0,h])$.

If $h<3a\Phi(D)$ then, applying Proposition~\ref{p:basic} on $[0,h]$ with $C=1$,
we get
$$
\T(f)=\T(f|[0,h])\le \frac{(h-f(h))(1-2D)}{3-3D}
\le
\frac{h}3 <a\Phi(D),
$$
so $u=a$ is a good choice in this case.

Therefore in the rest of the proof we can suppose that $h\ge 3a\Phi(D)$.
Let
$$
\phi(x)=\frac{f(x)}{x}+\delta\log x.
$$
(Recall that in this paper $\log$ denotes $\log_2$.)
Since $f$ is nonnegative and $1$-Lipschitz, $0\le f(x)/x\le 1$ on $(0,a]$, so
  for any $x\in(0,2^{-1/\delta}h)$ we have
  \begin{equation}\label{e:phi}
 \phi(x)=\frac{f(x)}{x}+\delta\log x < 1 +\delta \log(2^{-1/\delta}h)
\le \delta\log h \le \frac{f(h)}{h}+\delta\log h = \phi(h).
    \end{equation}

Now we claim that
\begin{equation}\label{e:u}
  \left(\exists u\in H\cap [2^{-1/\delta}h, h]\right)\
  (\forall x\in H\cap [\delta u, u])\
  \phi(x)\le \phi(u).
\end{equation}

To prove this we define a sequence $u_0>u_1>\ldots\in H$ inductively. Let $u_0=h$. Suppose that $u_n\in H$ is already defined.
Let $v\in H\cap[\delta u_n, u_n]$ be the largest number such that
$\phi(v)=\max_{H\cap [\delta u_n, u_n]}\phi$.
If $v=u_n$ then let $N=n$ and the procedure is terminated.

Otherwise letting $u_{n+1}=v$ we have $u_{n+1}<u_n$,
so the procedure can be continued.
Note that it follows from the construction that
$\phi(h)=\phi(u_0)\le\ldots\le \phi(u_n)$ and
$u_{n+2}<\delta u_n$ ($n=0,1,\ldots$).
Thus \eqref{e:phi} implies that the procedure must be terminated
in finitely many steps and \eqref{e:u} holds for $u=u_N$.

Let $u$ be chosen according to \eqref{e:u}.
Then, using that $h\ge 3a\Phi(D)$, we have
$u\ge 2^{-1/\delta}h\ge 2^{-1/\delta}\cdot 3a\Phi(D)$,
so the requirement $u\in[3a\Phi(D)2^{-1/\delta},a]$
is satisfied.
Thus it remains to prove \eqref{e:initiale}.

Let
$$
B=\frac{f(u)}{u}-\delta\log\delta.
$$
Since $u$ is chosen according to \eqref{e:u}, we have
\begin{equation}\label{e:Bx}
 (\forall x\in H\cap[\delta u,u]) \quad
 f(x)\le x\left(\frac{f(u)}{u}+\delta \log\frac{u}{x}\right)
 \le x\left(\frac{f(u)}{u}-\delta \log\delta\right) = Bx.
\end{equation}
Let $F(x)=\min(f(x),2Bx)$ ($x\in[0,u]$).

Now we claim that every $p\in H\cap[\delta u,u]$ is also a hard point of $F$.
Suppose, on the contrary, that $p\in H\cap[\delta u,u]$ is not a hard point of $F$.
Then there exists a $q\in[p/2,p]$ such that $F(q)<F(p)$.
By \eqref{e:Bx} we have $f(p)\le Bp\le 2Bp$, so by definition $F(p)=f(p)$,
and consequently we have
$$
F(q)<F(p)=f(p)\le Bp\le 2Bq,
$$
which implies that $f(q)=F(q)$. Thus $f(q)<f(p)$, so $p$ cannot be a hard point of $f$, which is
a contradiction.

Note that, by Lemma~\ref{l:hardpoints} and
since $F(p)=f(p)$ for any hard point of $F$,
the above claim
and the trivial estimate $\T(f|[0,\delta u])\le \delta u$
imply
\begin{equation}\label{e:fF}
\T(f|[0,v])\le \T(F|[0,v])+\delta u \qquad
\text{for any } v\in[0,u].
\end{equation}

First we consider the case when $f(u)/u<-\delta \log \delta$.

Then $B<-2\delta \log\delta$, and so
\[
0\le F(x)\le 2Bx < (-4 \delta\log\delta) x \quad\text{on }[0,u].
\]
If $-4\delta\log\delta >1$ then,
since $\Phi(D)\ge 0$ for $D\le 1/2$,
the righthand-side of \eqref{e:initiale} is
larger than $u$.
Since clearly $\T(g)\le u$ for
any $1$-Lipschitz function $g:[0,u]\to\R$
we are done if $-4\delta\log\delta >1$. So we may suppose that $-4\delta\log\delta \le 1$.
By Proposition~\ref{p:basic} applied to $F$, with $a=u, C=-4\delta\log\delta$ and $D=0$,
 we obtain
$$
\T(F)\le
u\cdot \frac{-4 \delta\log\delta}{1-8 \delta\log\delta}
< -4 u\delta\log\delta.
$$
By \eqref{e:fF} (applied to $v=u$) this implies that
\begin{equation}\label{e:large eps}
\T(f|[0,u])\le -4 u\delta\log\delta +\delta u.
\end{equation}

Since $D\le 1/2$, we have $\Phi(D)\ge 0$, so
\eqref{e:large eps}
implies \eqref{e:initiale}, which completes
the proof in the case when $f(u)/u<-\delta \log \delta$.

So in the rest of the proof we may assume that
\begin{equation}\label{e:fudelta}
f(u)/u\ge -\delta \log \delta.
\end{equation}

Since $\delta<1/2$ this also implies that $f(u)/u>\delta$.
Putting this together with the fact that $u$ was chosen according to \eqref{e:u}, and with the inequality $\log y\le y-1$, we get that
if $x\in H\cap[\delta u, u)$, then
\begin{equation}\label{e:f(x)f(u)}
f(x)\le x\frac{f(u)}{u}+x\delta\log\frac{u}{x}<
x\frac{f(u)}{u}+x\frac{f(u)}{u}\left(\frac{u}{x}-1\right)
= f(u).
\end{equation}
Since $u$ is a hard point,  $f(x)\ge f(u)$
on $[u/2,u]$, and so \eqref{e:f(x)f(u)} implies
that $H\cap[u/2,u)=\emptyset$.

Again because $u$ is a hard point, $f(u/2)\ge f(u)$. Using this, $\delta<1/2$
and the fact that $f$ is $1$-Lipschitz, we get
$$
B=\frac{f(u)}{u}-\delta \log\delta <
\frac{f(u/2)}{u}+\frac{1}{2}\le 1.
$$
Using again that $f$ is $1$-Lipschitz and $f(u/2)\ge f(u)$,
we get
$$
f(x)\ge
f\left(\frac{u}{2}\right)-\left(\frac{u}{2}-x\right)\ge
f(u)-\left(\frac{u}{2}-x\right) \quad (x\in [0,u/2]).
$$
Thus
\begin{equation}\label{e:newnumberedline}
f(x)\ge f(u)-\left(\frac{u}{2}-x\right)>Bx
\quad\textrm{if}\quad
\frac{\frac{u}{2}-f(u)}{1-B} < x \le \frac{u}{2}.
\end{equation}
Let $v_0= \frac{u/2-f(u)}{1-B}$.
Note that $f(x)>Bx$ also holds on the closed interval $[v_0,u/2]$ unless
$f(v_0)=Bv_0$.
The definition $B=\frac{f(u)}{u}-\delta \log\delta$
and the assumption \eqref{e:fudelta} imply that
$B\le 2f(u)/u$, hence $v_0\le u/2$.
Let $v=\max\{x\in[0,u/2]\ :\ f(x)=Bx\}$
(the maximum over a nonempty compact set).
By \eqref{e:newnumberedline} we have $v\le v_0$
and $f(x)>Bx$ on $(v,u/2]$.
By \eqref{e:Bx},
this implies that
$H \cap [\delta u,u]\cap (v, u/2)=\emptyset$.
Since above we obtained $H\cap[u/2,u)=\emptyset$
we get
$H\cap(v,u)\su[0,\delta u]$.
Hence, using Lemma~\ref{l:hardpoints}
and the trivial estimate $T(f|[0,\de u))\le \de u$, we get
\begin{equation}\label{e:Tuv}
\T(f|[0,u]) \le \T(f|[0,v]) +\delta u.
\end{equation}

Since $v\le v_0=\frac{u/2-f(u)}{1-B}$ and $f(v)=Bv$,
$$
2(v-f(v))=2(1-B)v\le u-2f(u) =
u\left(1-2\frac{f(u)}{u}\right)=
u(1-2B-2\delta\log\delta).
$$
Let $C=\min(2B,1)$. We have just seen that
\[
2(v-f(v)) \le u(1-C-2\delta\log\delta).
\]
Note also that $D\le f(u)/u=B+\delta\log\delta<B$, and so $D\le C/2$ since we assumed that $D\le 1/2$. Then $Dx\le F(x)\le Cx$ on
$[0,v]\su [0,u]$, so we can apply Proposition~\ref{p:basic} to get
$$
\T(F|[0,v])\le \frac{(v-f(v))(C-2D)}{1+2C-3D}\le
\frac{u(1-C-2\delta\log\delta)(C/2-D)}{1+2C-3D}.
$$
Note that $\frac{C/2-D}{1+2C-3D}< 1$. Using calculus, we get that
$\frac{(1-C)(C/2-D)}{1+2C-3D}\le\Phi(D)$ for $C\in [2D,1]$.
Therefore
$$
\T(F|[0,v])< u(\Phi(D)-2\delta\log\delta).
$$

Combining the above inequality with \eqref{e:fF} and
\eqref{e:Tuv},
we get \eqref{e:initiale}.
\end{proof}

\subsection{Stability results}
\label{subsec:stability}

The results of this subsection are only needed for the proof of Theorem \ref{thm:full-distance-set}. Moreover, to get the bound $\hdim(\Delta(A))\ge 37/54$ whenever $\hdim(A)>1$, one only needs to consider the case $D=0$ below. While there is no conceptual difference between the cases $D=0$ and $D>0$, the calculations are easier in the former case, so the reader may want to assume that $D=0$ in a first reading.

In the $C=1$ special case of Proposition~\ref{p:basic}, we get that if
$D\in [-1,1/2]$ and $f:[0,1]\to\R$ is a  $1$-Lipschitz function such that
 $f(0)=0$ and
$f(x)\ge Dx$ on $[0,1]$, then $\T(f)\le (1-2D)/3$.
As we will see in Section \ref{s:sharpness},
and is not hard to check, this estimate is sharp: if
\[
f(x) =  \left\{ \begin{array}{ccc}
          x & \text{ if } & x\in [0,(D+1)/2] \\
          1+D-x & \text{ if } & x\in [(D+1)/2,1]
        \end{array} \right.,
\]
then $\T(f)=(1-2D)/3$. In this section we prove a quantitative stability result
(Proposition~\ref{p:stability})
for $D\in[0,1/3]$, stating that if $\T(f)$ is close to $(1-2D)/3$ then $f(x)$ must be close to the
above function when $x$ is not too far from $0$ or
from $1$. 

The general plan to get this result is the following.
Let $b=\min_{[1/2,1]} f$ and choose $a\in[1/2,1]$
such that $f(a)=b$. It is easy to see that $\T(f)=\T(f|[0,a])$,
so it is enough to study $f|[0,a]$ instead of $f$.
We need to get an
upper estimate on $T(f)$ when $f$ is not close enough to
the function defined in the previous paragraph.
This upper estimate will be obtained by
finding a point $p\in [0,a]$ such that
in the good partition in the definition of $T(f)$,
the points $a_n$ in $[p,a]$
can be chosen such that
$\min_{[a_n,a_{n-1}]}f = f(a_{n-1})$, and so
for these indices the sum of the terms
$f(a_n)-\min_{[a_n,a_{n-1}]}f$ is
$f(p)-f(a)$ or, in other words, the smallest possible.
Combining this with a near optimal good partition for $f|[0,p]$
guaranteed by Proposition~\ref{p:basic}, we get a near optimal
lower bound for $T(f)$ for all $f$ with such a special point
$p$ and value $f(p)$. These points $p$ will be called
simple points, and after proving the above described
near optimal upper estimate, most of the proof will be about
hunting a simple point such that the estimate we obtain
for $\T(f)$ is the upper estimate we claim.

First we collect some assumptions and define precisely
the above mentioned notion of simple points.
\begin{definition}\label{d:simple}
  Suppose that
 \begin{equation}\label{e:conditions on f}
  \begin{aligned}
    \ a\in (0,1],\ b\in \R,\ D\in[0,1/2), \
   f:[0,a]\to\R \text{ is $1$-Lipschitz, }\\
   f(0)=0,\
    f(x)\ge Dx \ (x\in [0,a]) \
  \text{ and }
    \min_{[a/2,a]} f=f(a)=b.
  \end{aligned}
 \end{equation}
  A point $p\in[0,a]$ is called \emph{simple} if there exists a finite
  sequence $p=p_0<p_1<\ldots<p_k=a$ such that
  \begin{equation}\label{e:simple condition}
    \frac{p_i}{p_{i-1}} \le 2 \quad \text{ and }
    \quad f(p_i)=\min_{[p_{i-1},p_i]} f \qquad (i=1,\ldots,k).
  \end{equation}
\end{definition}

\begin{lemma}\label{l:ifsimple}
  If \eqref{e:conditions on f} holds and $p\in[0,a]$ is a simple point then
  $$
  \T(f)\le \al p+ (1- \al) f(p) -b,
  \text{ where }
  \al = \frac{1-2D}{3(1-D)}.
  $$
\end{lemma}

\begin{proof}
  Applying Proposition~\ref{p:basic} to $f|[0,p]$ with $C=1$ we
  get $\T(f|[0,p])\le \al (p-f(p))$.
  Hence for any $\de>0$ there exists a good partition $(a_n)$ of $[0,p]$
  such that
  $$
  \T(f|[0,p],(a_n))\le \al (p-f(p)) + \de.
  $$
  Since $p$ is simple there exists a finite sequence
  $p=p_0<p_1<\ldots<p_k=a$ such that \eqref{e:simple condition} holds.

  For $n\le k$ let $a'_n=p_{k-n}$ and for $n>k$ let $a'_n=a_{n-k}$.
  Then
$( a'_n)$
    is a good partition of $[0,a]$ and
  \begin{eqnarray*}
    \T(f,( a'_n))
    &=& \T(f|[0,p],( a_n))
    +\sum_{i=1}^k f(p_{i-1})-f(p_i) \\
    &\le &  \al (p-f(p)) + \de + f(p) - f(a)
    = \al p+ (1- \al) f(p)-b +\de,
  \end{eqnarray*}
  which completes the proof.
\end{proof}

\begin{lemma}\label{l:trivi simple}
  Suppose that \eqref{e:conditions on f} holds and let $p\in [0,a]$.
  If
\begin{equation}\label{e:lesstrivi}
\text{ for every } z\in[p,a/2)
\text{ there exists } y \in (z,2z]
\text{ such that } f(y)\le f(z)
\end{equation}
then $p$ is simple.
\end{lemma}

\begin{proof}
  Let $p_0=p$. Suppose that $n\ge 0$ and
   $p_0<\ldots<p_n$ are defined such that
  \eqref{e:simple condition} holds for $k=n$.
  If $p_n\ge a/2$ then let $p_{n+1}=a$ and we are done.
  Otherwise, let $p_{n+1}\in[p_n,2p_n]$ be the largest number such that
  $f(p_{n+1})=\min_{[p_n,p_{n+1}]} f$.
  By \eqref{e:lesstrivi} we also have $p_{n+1}>p_n$.
  It remains to check that the procedure terminates, which follows
  from the simple observation that $p_{n+2}\ge \min(2 p_n,a)$ by definition.
\end{proof}

\begin{lemma}\label{l:thensimple}
  Suppose that \eqref{e:conditions on f} holds.
  If $p\in[a/2,a]$, or if $p\in[0,a/2]$ and $f(p)\ge -2p+a+b$,
  then $p$ is a simple point.
\end{lemma}

\begin{proof}
  The case $p\in[a/2,a]$ is clear, so suppose that
  $p\in[0,a/2]$ and $f(p)\ge -2p+a+b$.
  Then the $1$-Lipschitz property of $f$ implies that for any $x\in[p,a]$
  we also have $f(x)\ge -2x+a+b$.
  Since $f$ is $1$-Lipschitz and $f(a)=b$ we have $f(y)\le -y+a+b$
  for any $y\in [0,a]$.
  Thus $f(x)\ge -2x+a+b\ge f(2x)$ for any $x\in [p,a/2]$,
  so Lemma~\ref{l:trivi simple} completes the proof.
\end{proof}

\begin{lemma}\label{l:1-a+2b-2D}
  Condition \eqref{e:conditions on f} implies that $1-a+2b-2D\ge 0$.
\end{lemma}

\begin{proof}
Note that
$b=f(a)\ge Da$, so
$$
1-a+2b-2D \ge 1 -a + 2aD - 2D = (1-a)(1-2D) \ge 0.
$$
\end{proof}

\begin{lemma}\label{l:lower}
  If \eqref{e:conditions on f} holds  and
  $\T(f)> \frac{1-2D}3-\de$ for some
  $\de\in (0,a/3)$
  then
 \[
 f(x) > x - 3\de(1-D) \qquad \text{ on } [0,t_0],
 \text{ where } t_0=\frac{a+b}{3} +\de(1-D).
 \]
\end{lemma}

\begin{proof}
 First note that $\de<a/3$ implies that $t_0<a$.
  Since $f(0)<-2\cdot 0+a+b$ and $f(a)\ge -2\cdot a + a +b$ there exists
  a $t\in(0,a]$ such that $f(t)=-2t+a+b$.
 By Lemma~\ref{l:thensimple}, $t$ is a simple point,
  so writing $\al = \frac{1-2D}{3(1-D)}$ and using Lemma~\ref{l:ifsimple},
   we get
\begin{align*}
  \T(f) &\le \al t + (1-\al) f(t) - b \\
  &= \al t + (1-\al) (-2t+a+b) - b \\
  &= \frac{-t}{1-D}+\frac{2-D}{3(1-D)}(a+b)-b.
\end{align*}
Combining this with the assumption
  $\T(f)> \frac{1-2D}3 - \de$ and multiplying
  through by $3(1-D)$,  we get
\[
3t < (2-D)(a+b)-3(1-D)b-(1-2D)(1-D)+3\de(1-D),
\]
which can be rewritten as
\[
3t < a+b+ (1-D)(3\de - (1-a+2b-2D)).
\]
By Lemma~\ref{l:1-a+2b-2D}, this implies
$t < \frac{a+b}{3} +\delta(1-D)=t_0$.
Using this and the $1$-Lipschitz property of $f$, we obtain
\[
f(t_0) \ge
f(t)-(t_0-t) >f(t)-2(t_0-t) =a+b-2t_0=t_0-3\de(1-D).
\]
Using again that $f$ is $1$-Lipschitz, this gives the claim.
\end{proof}

\begin{lemma}\label{l:newthensimple}
  Suppose that \eqref{e:conditions on f} holds,
  $0\le p\le \frac{a+b-v}{2} < u \le a$, $f(u)=v$
 and $f(x)\ge v$ on $[p,\frac{a+b-v}{2}]$.
If $v\ge u/2$ or $f(p)=-2p+u+v$,
 then $p$ is simple.
\end{lemma}

\begin{proof}
It is useful to note that by the $1$-Lipschitz property
of $f$, the assumptions $u\le a$, $f(u)=v$ and $f(a)=b$ imply that
$u+v\le a+b$,
 and so $\frac{u}{2}\le\frac{a+b-v}{2}$.

By Lemma~\ref{l:trivi simple} it is enough to check
\eqref{e:lesstrivi}. So let $z\in[p,a/2)$.
We distinguish three cases.

First suppose that $z\ge \frac{a+b-v}{2}$.
Then, using that $f(\frac{a+b-v}{2})\ge v$,
$f$ is $1$-Lipschitz, $2z<a$ and $f(a)=b$, we get
$$
f(z)\ge v-\left(z-\frac{a+b-v}2\right)
\ge  v-2\left(z-\frac{a+b-v}2\right)= -2z+a+b\ge f(2z).
$$
Therefore \eqref{e:lesstrivi} holds in this case.

Now suppose that $z\in[\frac{u}{2}, \frac{a+b-v}{2}]$.
Since we consider only $z\in[p,a/2)$ we also have
$z\in [p, \frac{a+b-v}{2}]$.
Then $f(z)\ge v$,
$u\in(z,2z]$ and $f(u)=v\le f(z)$,
so  \eqref{e:lesstrivi} holds in this case as well.

Finally, suppose that $z\in[p,\frac{u}{2})$.
  Then $v \le f(z)\le  z < u/2$, hence we cannot
  have $v \ge u/2$, so we must have $f(p)=-2p+u+v$.
  Using that $f$ is $1$-Lipschitz and $z\ge p$,
  this implies $f(z)\ge -2z+u+v$.
Since $f$ is $1$-Lipschitz and $f(u)=v$ we have
$f(x)\le u+v-x$ on $[0,u]$.
Thus $f(2z)\le u+v-2z\le f(z)$, which completes the proof.
\end{proof}

\begin{lemma}\label{l:middle}
  If \eqref{e:conditions on f} holds  and
  $\T(f)> \frac{1-2D}3-\de$ for some
  $\de\in (0,a/3)$
then
  \[f(x)> \frac{a+b}3-2\de(1-D) \quad \text{ on } [t_0,2t_0-6\delta(1-D)],
  \text{  where }
  t_0=\frac{a+b}{3} +\delta(1-D).
 \]
\end{lemma}

\begin{proof}
  Let $v=\frac{a+b}3-2\de(1-D)$.
  If $v<0$ then the claim is clear, so we can suppose that $v\ge 0$.
  By Lemma~\ref{l:lower}, $f(t_0)>v$. 
  Thus if the claim is false then there exists a 
  $u\in (t_0,2t_0-6\delta(1-D)]$ such that $f(u)=v$.

By \eqref{e:conditions on f},
we have $b\le \frac{a}{2}$, which implies
\[
2t_0-6\de(1-D)=\frac{2(a+b)}{3}-4\de(1-D)<a.
\]
Since $f(0)\le v<f(t_0)$ we also have a largest $p\in [0,t_0)$ such
  that $f(p)=v$.
  Then $f(x)\ge v$ on $[p,t_0]$.
  Since $\frac{a+b-v}2=t_0$ and $u/2\le t_0-3\de(1-D)=v$, all the
  assumptions of Lemma~\ref{l:newthensimple} hold, so we get that
  $p$ is simple.

  Then by Lemma~\ref{l:ifsimple} we have
  $\T(f)\le \al p + (1-\al) v - b$, where $\al = \frac{1-2D}{3(1-D)}$.
  Since $p<t_0=v+3\delta(1-D)$, this implies that $\T(f)\le v+(1-2D)\de-b$.
  Combining this with the assumption $\T(f)> \frac{1-2D}3-\de$ we get
  \[
  v> b+\frac{1-2D}3-2\de(1-D).
  \]
  Note that Lemma~\ref{l:1-a+2b-2D} implies that $b+\frac{1-2D}3\ge \frac{a+b}{3}$,
  so we obtain $ v> \frac{a+b}3-2\de(1-D)$,
  which is a contradiction.
\end{proof}

From the last lemma and the Lipschitz property of $f$ one can easily derive a good lower estimate also on
$[2t_0-6\de (1-D),a]$. However, the next lemma will lead to an even better (and, as we will see later, sharp) estimate on the right part of $[0,a]$.
\begin{lemma}\label{l:loweruv}
  Suppose that \eqref{e:conditions on f} holds,
 $$
 \T(f) > \frac{1-2D}3-\de,\
   \de\in (0,a/3),\
  u \in (a/2, a],
  u \ge 2v + 6\de(1-D),\
  \text{ and }
  f(u)=v.
 $$
  Then
  $$
  u+v> 1+D -3\de\frac{1+D}{1-2D}.
  $$
\end{lemma}

\begin{proof}
 Since $f(0)<-2\cdot 0+u+v$ and $f(a)\ge -2\cdot a + u+v$ there exists
 a $p\in(0,a]$ such that $f(p)=-2p+u+v$.
First we prove that $p$ is a simple point.
To get this, by Lemma~\ref{l:newthensimple}, it is enough to check that
$p\le\frac{a+b-v}2<u$ and  $f(x) \ge v$ on $[p,\frac{a+b-v}2]$.

Since $u\in(a/2,a]$, $v=f(u)$ and $b=\min_{[a/2,a]} f$,
we have $b\le v$, so $\frac{a+b-v}2\le a/2<u$.

Note (as in Lemma~\ref{l:newthensimple}) that $u+v\le a+b$.
By Lemma~\ref{l:lower}, we have $f(t_0)>t_0-3\de(1-D)$, where $t_0=\frac{a+b}{3} +\delta(1-D)$.
Then
\begin{equation}\label{e:2t_0+f(t_0)}
2t_0+f(t_0)>3t_0-3\de(1-D)=a+b\ge u+v=2p+f(p).
\end{equation}
  Since $f$ is $1$-Lipschitz this implies that $p< t_0$.
  Using this, $u+v\le a+b$ and finally
  the assumption $u\ge 2v+6\de(1-D)$, we get
\begin{align*}
  p< t_0 &= \frac{a+b}3+\de(1-D)\\
  &=\frac{a+b}2-\frac{a+b}6+\de(1-D)\\
  &\le \frac{a+b}2-\frac{u+v}6+\de(1-D)\le \frac{a+b-v}2.
\end{align*}

  On $[0,t_0]$ we have $f(x)-x> -3\de(1-D)$ by Lemma~\ref{l:lower},
  on $[p,a]$ we have $2x+f(x)\ge 2p + f(p) = u+v$
  by the $1$-Lipschitz property of $f$.
   Taking the linear combination of these inequalities with weights $2/3$
  and $1/3$, we get
  $$
  f(x)> \frac{u+v}3 - 2\de(1-D) \quad \text{on } [p,t_0].
  $$
  By the assumption $u \ge 2v + 6\de(1-D)$, this gives $f(x)\ge v$
  on $[p,t_0]$.

  Using that $f$ is $1$-Lipschitz and then \eqref{e:2t_0+f(t_0)},
  we get that on $[t_0,a]$ we have $2x+f(x)\ge 2t_0 +f(t_0)>a+b$,
  which implies that $f(x)\ge v$ also on $[t_0,\frac{a+b-v}2]$.

  Therefore, by Lemma~\ref{l:newthensimple}, $p$ is indeed a simple point. Now Lemma~\ref{l:ifsimple} gives
\begin{align*}
 \T(f) &\le \al p + (1-\al)f(p)-b\\
& =\al p + (1-\al)(-2p+u+v)-b\\
& =(3\al - 2) p +(1-\al)(u+v)-b.
\end{align*}
Recalling that $\al=\frac{1-2D}{3(1-D)}$,
it is easy to check that $D<1$ implies that $3\al-2 <0$.
 The $1$-Lipschitz property of $f$ and $f(0)=0$ imply that
 $0\le p-f(p)=3p-(u+v)$, so $p\ge\frac{u+v}{3}$.
Using these facts, the last displayed equation yields
$$
\T(f)\le  \left(\frac{3\al-2}{3}+1-\al\right)(u+v)-b
=\frac{u+v}{3}-b.
$$
Combining this with the assumption $\T(f)>\frac{1-2D}{3}-\de$
we get $u+v> 1-2D-3\de+3b$.
Note that $u+v\le a+b$
and $b=f(a)\ge Da$
imply that $b \ge \frac{D}{D+1}(u+v)$. Combining these facts, we conclude that
$$
u+v> 1-2D-3\de+3b \ge 1-2D-3\de + \frac{3D}{D+1}(u+v),
$$
which implies (using also that $D<1/2$) the claim.
\end{proof}

The following proposition provides a global quantitative estimate for functions $f:[0,1]\to\R$ for which $\T(f)$ is close to the maximum possible value.
\begin{prop}\label{p:stability}
Fix $D\in[0,1/3]$, $\de\in(0,1/21]$ and let
$f:[0,1]\to\R$ be a $1$-Lipschitz function
such that $f(0)=0$, $f(x)\ge Dx$ on $[0,1]$ and
$\T(f) > \frac{1-2D}{3}-\de$.
Let
\begin{equation} \label{eq:t_1}
t_1= \frac{1+D}{3} - \de\left(\frac{1+D}{1-2D}-(1-D)\right).
\end{equation}

Then
\begin{align}
  \label{e:left}
         x-3\de(1-D) \ < \ & f(x) \  \le x \qquad \text{ on }  [0,t_1] \\
  \label{e:middle}
  t_1-3\de(1-D) \ < \ & f(x)
         \qquad \text{ on }  [t_1,2t_1-6\de(1-D)] \qquad \text{and} \\
  \label{e:right}
  3t_1-x-3\de(1-D) \ < \ & f(x) \ < \ 1+D-x + 3\de\frac{1-D}{1-2D}
   \qquad  \text{ on } [2t_1,1].
\end{align}
\end{prop}

\begin{proof}
Let $b=\min_{[1/2,1]} f$ and choose $a\in[1/2,1]$ such that
$f(a)=b$.
Then it is easy to see that $\T(f)=\T(f|[0,a])$.
So combining the assumption $\T(f) > \frac{1-2D}{3}-\de$
and Proposition~\ref{p:basic} for $f|[0,a]$ and $C=1$,
and then using $b=f(a)\ge Da$,
we get
\begin{equation}\label{e:ab}
\frac{1-2D}{3}-\de < \T(f) =\T(f|[0,a]) \le
\frac{(a-b)(1-2D)}{3(1-D)} \le a\frac{1-2D}{3}.
\end{equation}
This implies

\begin{equation}\label{e:a}
a> 1-\frac{3\de}{1-2D}
\end{equation}
and so
\begin{equation}\label{e:a+b}
a+b\ge a+Da > 1 + D - 3\de\frac{1+D}{1-2D}.
\end{equation}
By \eqref{e:ab},
\[
a-b> 1-D -3\de\frac{1-D}{1-2D}.
\]
Since $f$ is $1$-Lipschitz this implies that
$$
f(1)\le b+1-a< D+3\de\frac{1-D}{1-2D},
$$
which (using again that $f$ is $1$-Lipschitz) yields the upper estimate of
\eqref{e:right} on $[0,1]$.

By definition we have $\min_{[1/2,a]}f=f(a)=b$, but
in order to apply our lemmas to $f|[0,a]$ we have to show
$\min_{[a/2,a]}f=f(a)$.
Suppose then that $\min_{[a/2,a]}f<f(a)$. Then there exists an $a'\in[a/2,1/2)$
  such that  $\min_{[a/2,a]}f=f(a')$. Using Proposition~\ref{p:basic}
applied to $f|[0,a']$ and $C=1$  we get
  $$
  \frac{1-2D}{3}-\de < \T(f) =\T(f|[0,a])=\T(f|[0,a'])
  \le a' \frac{1-2D}{3} \le \frac12\cdot\frac{1-2D}{3},
  $$
  which is impossible, since we assumed $D\le 1/3$ and $\de\le 1/21$.

Therefore \eqref{e:conditions on f} holds for $f|[0,a]$.
Note that \eqref{e:a+b} implies that $t_0> t_1$,
where $t_0=\frac{a+b}3+\delta(1-D)$ (as in Lemma~\ref{l:middle}).

By \eqref{e:a}, $D\le 1/3$ and $\de\le 1/21$,
we get $a>4/7$, so $\de<a/3$ holds.
Then applying Lemmas~\ref{l:lower} and \ref{l:middle} to
$f|[0,a]$ and using that $f$ is $1$-Lipschitz
we get the lower estimate
of \eqref{e:left} and \eqref{e:middle}. The upper estimate
of \eqref{e:left} is clear.

It remains to prove the lower estimate of
\eqref{e:right}.
Lemma~\ref{l:loweruv} (for $f|[0,a]$) gives that
every point of the graph of $f|(a/2,a]$ must be above
either the $y=1+D-x-3\de\frac{1+D}{1-2D}$ line, or the
$y=\frac{x}{2}-3\de(1-D)$ line.
These two lines intersect at $(2t_1,t_1-3\de(1-D))$. On the other hand, a calculation using $\de\le 1/21$, $D\le 1/3$,
 $a\le 1$ and \eqref{e:a}
shows that $a/2\le 1/2 < 2t_1 \le 1-\frac{3\de}{1-2D}< a$. We deduce that $f(2t_1)> t_1-3\de(1-D)$.
Using the $1$-Lipschitz property of $f$, this gives
the lower estimate of \eqref{e:right}.
\end{proof}

\begin{remark}
  Let $D\in[0,1/3]$, and let
$f:[0,1]\to\R$ be a $1$-Lipschitz function
  such that $f(0)=0$, $f(x)\ge Dx$ on $[0,1]$
  and $\T(f)\ge \frac{1-2D}3$.
Letting $\delta\to 0^+$ in Proposition~\ref{p:stability}, we get
that $f(x)=x$ on $[0,\frac{1+D}3]$,
  $f(x)\ge \frac{1+D}{3}$ on $[\frac{1+D}3, \frac{2(1+D)}3]$,
  and $f(x)=1+D-x$ on $[\frac{2(1+D)}3,1]$.
  It is easy to see that, conversely, $\T(f)=\frac{1-2D}3$ for any such $f$.
  Recall that by the $C=1$ special case of Proposition~\ref{p:basic},
  we have  $\T(f)\le \frac{1-2D}3$ for any
  $1$-Lipschitz function $f:[0,1]\to\R$
  such that $f(0)=0$, $f(x)\ge Dx$ on $[0,1]$.
  Therefore the above observation gives a characterization of those
  functions for which we have equality in Proposition~\ref{p:basic} when $C=1$.
\end{remark}

The following corollary can be seen as a version of Proposition \ref{p:stability} that is closer to the kind of estimates we will need in the proof of Theorem \ref{thm:full-distance-set}.
\begin{corollary}\label{c:37/54}
Let $D\in[0,1/3)$ and
$$
\Lambda(D)=\frac{(1+D)(37-50D+60D^2)}{18(3-4D+5D^2)}
\ge \Lambda(0)=\frac{37}{54}=0.6851851\ldots.
$$
Then there exist $\eta>0$ and
$\xi\in(2/3,1]$ (depending continuously on $D$)
such that $\Lambda(D)=\xi(1-2\eta)$ and the following holds.

If $f:[0,1]\to\R$ is a $1$-Lipschitz function such that
$f(0)=0$ and $f(x)\ge Dx$ on $[0,1]$ then
$$
\T(f) > 1-\Lambda(D) \quad \Longrightarrow
\quad f(x) \ge \frac{x}{3}-\eta\xi \text{ on } [0,\xi].
$$
\end{corollary}

\begin{proof}
  Let
  $$
  \de=\frac{(1+D)(1-2D)}{18(3-4D+5D^2)}
  \in\left(0,\frac1{36}\right).
  $$
  Then $1-\Lambda(D)=\frac{1-2D}3-\de$. Let $t_1$ be the number given by Proposition~\ref{p:stability}. By the hypothesis $\T(f)> 1-\Lambda(D)$,
  Proposition~\ref{p:stability} implies that \eqref{e:left},  \eqref{e:middle}
  and \eqref{e:right} hold.

  Let
  \begin{align*}
  \xi&=\frac34\left(\de(1-3D)+1+D-3\de\frac{1+D}{1-2D}\right) =
  \frac{(1+D)(13-20D+24D^2)}{6(3-4D+5D^2)}\in (2/3,1),\\
  \eta&=\frac{\delta(1-3D)}{\xi}=
  \frac{(1-2D)(1-3D)}{3(13-20D+24D^2)}>0.
  \end{align*}
  Then the three lines $y=x-3\de(1-D)$, $y=x/3-\eta\xi$ and $y=Dx$ meet
  at $(3\de,3D\de)$. Thus, since $D<1/3$, on $[0,3\de]$ we have
  $f(x)\ge Dx\ge x/3-\eta\xi$ and, using \eqref{e:left}, on $[3\de,t_1]$ we have
  $f(x) \ge x-3\de(1-D) \ge x/3-\eta\xi$.

  One can also check that the lines $y= x/3-\eta\xi$ and
  $y=3t_1-x-3\de(1-D)$ intersect at $x=\xi$.
  Thus, by \eqref{e:right}, on $[2t_1,\xi]$ we also have
  $f(x) > 3t_1-x-3\de(1-D) \ge  x/3-\eta\xi$.

It remains to check $f(x)\ge x/3-\eta\xi$ on
$[t_1,2t_1]$.
By \eqref{e:right}, $f(2t_1)> t_1-3\delta(1-D)$.
Hence \eqref{e:middle} and the $1$-Lipschitz property
of $f$ imply that on $[t_1,2t_1]$ we have
$f(x)> g(x)-3\delta(1-D)$, where
\[
g(x)= \begin{cases}
   t_1 & \text{ on } [t_1,2t_1-6\de(1-D)] \\
   3t_1 -6\de(1-D) -x & \text{ on } [2t_1-6\de(1-D),2t_1-3\de(1-D)] \\
   x-t_1 & \text{ on } [2t_1-3\de(1-D),2t_1]
     \end{cases}.
\]

Now we claim that
\begin{equation}\label{e:gx0}
g(x_0)-3\delta(1-D) \ge x_0/3-\eta\xi, \qquad
\text{where } x_0=2t_1-3\de(1-D).
\end{equation}
Indeed, using the definition of $g$ and $x_0$ and
the equation $\eta\xi=(1-3D)\de$, we obtain that the
left-hand side of \eqref{e:gx0} is
$t_1 - 6\de (1-D)$, the right-hand side is
$2t_1/3 -\de (2-4D)$, so it is enough to prove that
$t_1/3 \ge \de(4-2D)$.
It is straightforward
to check that this last inequality follows
from the definition \eqref{eq:t_1} of $t_1$, $D\in[0,1/3)$
and $\de\in(0,1/36)$.

Note that the function $g(x)-3\de(1-D)$ has slope
$0$ or $-1$ on $[t_1,x_0]$ and it has slope $1$ on
$[x_0,2t_1]$, while $x/3-\eta\xi$ has slope $1/3$.
Thus \eqref{e:gx0} implies that
$g(x)-3\delta(1-D) \ge x/3-\eta\xi$ on $[t_1,2t_1]$,
which completes the proof.
\end{proof}

\subsection{Total drop for $\tau$-good partitions}

In this subsection we show that for small $\tau$,
allowing only $\tau$-good partitions
(recall Definition~\ref{d:many}) does not change
too much the smallest possible total drop, see
Corollary~\ref{c:tau} below.
We begin with a lemma that will allow us to obtain $\tau$-good partitions from partitions that satisfy a weaker property, with a controlled change in the total  drop.
\begin{lemma}\label{l:tau}
Let $f:[0,a]\to\R$ be a $1$-Lipschitz function and
$(a_n)$ be a good partition of $[0,a]$.
Suppose that $\tau>0$, $K>1$ is an integer,
$(1+\tau)^K<2$ and
$a_{n-K}/a_n\ge 2$ for every $n\ge K$.
Then
$$
\T_\tau(f) < \T(f,(a_n)) + 6K(K-1) \tau a .
$$
\end{lemma}

\begin{proof}
Fix $i\in \N_0$ and consider the numbers $\beta_j=a_{iK+j-1}/a_{iK+j}$ ($j=1,\ldots,K$).
Then $\beta_1\cdots\beta_K=a_{iK}/a_{iK+K}\ge 2$ and for every $j$ we have $1<\beta_j\le 2$. The goal is to make every $\beta_j$ at least $1+\tau$ so that each of them remains at most $2$, the product $\beta_1\cdots \beta_K$ stays fixed, and the numbers $a_{iK+j}$ are changed by only a small amount.

So let $\beta'_j=1+\tau$ if $\beta_j\le 1+\tau$, and to get the remaining
$\beta'_j$'s decrease some of the corresponding $\beta_j$'s (and choose $\beta'_j=\beta_j$ for the rest), so that still $\beta'_j\ge 1+\tau$ and $\beta'_1\cdots\beta'_K=\beta_1\cdots\beta_K$; this is possible since $(1+\tau)^K<2$.
Then
 let $a'_{iK}=a_{iK}$ and
for each $j=1,\ldots, K$
let $a'_{iK+j}=a_{iK}/(\beta'_1\cdots \beta'_j)$. Note
that
 $a'_{iK+K}=a_{iK+K}$,
$1+\tau\le a'_{iK+j-1}/a'_{iK+j} \le 2$ for every $j=1,\ldots,K$
and each $a_{iK+j}$ was multiplied by a factor between $(1+\tau)^{-K}$ and $(1+\tau)^{K}$ to get $a'_{iK+j}$. This implies that for every $j=1,\ldots,K-1$,
\begin{equation}\label{e:translate}
  |a'_{iK+j}-a_{iK+j}|\le a_{iK} ((1+\tau)^{K}-1) \le
  2(a_{iK}-a_{(i+1)K})((1+\tau)^{K}-1).
\end{equation}

Note that $(a'_n)_{n=0}^{\infty}$ obtained by applying this procedure for every
$i\in\N_0$
is a $\tau$-good partition of $[0,a]$.

Let $\tau_0=2^{1/K}-1$. Since $\frac{(1+x)^K-1}{x}$ is increasing on $(0,\infty)$ (being a polynomial with positive coefficients) and
$(1+\tau)^K< 2$, we have
$$
\frac{(1+\tau)^K-1}{\tau}<
\frac{(1+\tau_0)^K-1}{\tau_0}=
\frac{1}{2^{1/K}-1} \le \frac{K}{\ln 2}<\frac{3K}{2},
$$
where we used the inequality
$e^t-1 \ge t$.
Thus
$$
(1+\tau)^K-1< \frac{3K\tau}2.
$$
Combining this with \eqref{e:translate}
 and $a'_{iK}=a_{iK}$, then
adding up, we get
\begin{eqnarray*}
  \sum_{n=0}^{\infty}|a'_n-a_n| <
  \left(\sum_{i=0}^{\infty}(K-1)\cdot 2(a_{iK}-a_{(i+1)K})
                    \frac{3K\tau}{2}\right) \\
  = (K-1)(a_0-\lim_{n\to\infty}a_n)\cdot 3K\tau
  =   3K(K-1)\tau a.
\end{eqnarray*}
Since $f$ is $1$-Lipschitz, changing one $a_n$ by $\eta$
can change $\T(f,(a_n))$ by at most $2\eta$, so the above inequality
implies
$$
\T(f,(a'_n))< \T(f,(a_n)) + 6K(K-1) \tau a,
$$
which completes the proof of the lemma.
\end{proof}

The next lemma shows that we can replace an arbitrary good partition by one satisfying the assumptions of Lemma \ref{l:tau}, without increasing the total drop.

\begin{lemma}\label{l:merging}
For any $\de>0$ and
$1$-Lipschitz function $f:[0,a]\to\R$
there exists a good  partition $(a'_n)$ such that
$a'_{n-3}/a'_n > 2$ for every $n\ge 3$ and
$\T(f,(a'_n))\le \T(f)+\delta$.
\end{lemma}

\begin{proof}
It is enough to show that
for any good partition $(a_n)$
there exists a good  partition $(a'_n)$ such that
$a'_{n-3}/a'_n > 2$ for every $n\ge 3$ and
$\T(f,(a'_n))\le \T(f,(a_n))$.

First we claim that we can suppose that every interval
$[a_n,a_{n-1}]$ is increasing or decreasing
(recall Definition~\ref{d:many}).
Indeed, for each $n\ge 1$
if on the interval $[a_n,a_{n-1}]$ the minimum of $f$ is
taken at $p\in(a_n,a_{n-1})$ then inserting $p$ to the partition
(in between $a_{n-1}$ and $a_n$) we get a new good partition
such that $[a_n,p]$ is decreasing and $[p,a_{n-1}]$ is increasing and
it is easy to see that $\T(f,(a_n))$ is not changed.

Suppose that $a_{n-2}/a_n<2$ $(n\ge 2)$.
If $[a_n,a_{n-1}]$ and $[a_{n-1},a_{n-2}]$ are both increasing
or both decreasing,  then by merging these intervals we get an interval of the same type,
and $\T(f,(a_n))$ remains unchanged.
If $[a_n,a_{n-1}]$ is increasing and $[a_{n-1},a_{n-2}]$
is decreasing then after merging the two intervals
the minimum of $f$ on $[a_n,a_{n-2}]$ is still
achieved at one of the endpoints of the interval,
and  $\T(f,(a_n))$ does not increase.

Applying the above merging procedure inductively (starting with $n=2$) whenever possible, we get a good partition $(a'_n)$ such that
whenever $a'_{n-2}/a'_n<2$ $(n\ge 2)$ then $[a'_n,a'_{n-1}]$ is decreasing and $[a'_{n-1},a'_{n-2}]$
is increasing. Since this cannot happen for both $n$ and
$n-1$ we get that $a'_{n-2}/a'_n\ge 2$ or
$a'_{n-3}/a'_{n-1}\ge 2$ for any $n\ge 3$, which clearly
implies that $a'_{n-3}/a'_n> 2$.
\end{proof}

\begin{corollary}\label{c:tau}
For any $1$-Lipschitz function $f:[0,a]\to\R$ and
any $0<\tau<1$,
$$
\T_\tau(f)\le \T(f) + 36 \tau a.
$$
\end{corollary}

\begin{proof}
Note that for any $1$-Lipschitz function $f:[0,a]\to\R$
and any partition $(a_n)$ of $[0,a]$, by definition, we have
$0\le \T(f,(a_n))\le a$.
Thus the claim  holds trivially if
$\tau\ge \root 3 \of 2 - 1$.
Otherwise we can apply Lemma~\ref{l:merging}, and then
Lemma~\ref{l:tau} (for $K=3$).
\end{proof}

\subsection{Discretizing the estimates}

Recall from Definition \ref{def:integer-partition} the notion of $\tau$-good partition of an integer interval $(0,\ell]$, and the notation $\M(\sigma,(N_i))$. Sometimes we refer to these as \emph{integer} partitions for emphasis. Note that the requirement
\eqref{e:integertau}
for a $\tau$-good integer partition slightly
differs from the requirement
$1+\tau\le a_{k-1}/a_k \le 2$ for a $\tau$-good
partition (see Definition~\ref{d:many}), which
is equivalent to $\tau a_k \le a_{k-1}-a_k\le a_k$.
These two notions are connected by the following
lemma.

\begin{lemma}\label{l:discretizing}
  Assume that $L\le \ell$ are positive integers.  Let $f:[0,L/\ell]\to\R$ be a $1$-Lipschitz function and let $(a_n)$ be a
  $(2\tau)$-good partition of $[0,L/\ell]$.
  Then there exists a $\tau$-good integer partition
  $0=N_0<\ldots<N_q=L$ of $(0,L]$ such that
  $$
  \sum_{j=0}^{q-1} f(N_j/\ell)-\min_{[N_j/\ell,N_{j+1}/\ell]} f \le \T(f,(a_n))
  +O_{\tau}(\log \ell/\ell).
  $$
\end{lemma}

\begin{proof}
Let $N_0<\ldots<N_q$ be the values taken by the sequence
  $\lfloor \ell a_n \rfloor$. Since $a_n\to 0$ we get $N_0=0$. Thus $0=N_0<\ldots<N_q=L$ is an integer partition of $(0,L]$.

Using that $(a_n)$ is a good partition we get
$$
\lfloor \ell a_n \rfloor \le \ell a_n \le 2 \ell a_{n+1} <
2\lfloor \ell a_{n+1} \rfloor + 2,
$$
hence $\lfloor \ell a_n \rfloor - \lfloor \ell a_{n+1} \rfloor \le
\lfloor \ell a_{n+1} \rfloor + 1$.
Thus to prove that $(N_j)$ is a $\tau$-good integer partition of $(0,L]$
it is enough to show that
$\lfloor \ell a_n \rfloor - \lfloor \ell a_{n+1} \rfloor \ge
\tau \lfloor \ell a_{n+1} \rfloor$ if
$\lfloor \ell a_n \rfloor > \lfloor \ell a_{n+1} \rfloor$.
This is clear if $\tau\lfloor \ell a_{n+1} \rfloor\le 1$.
Otherwise, using also that $(a_n)$ is a $(2\tau)$-good partition, we get
$$
\lfloor \ell a_n \rfloor - \lfloor \ell a_{n+1} \rfloor >
\ell a_n - \ell a_{n+1} - 1
\ge 2\tau \ell a_{n+1} - 1 \ge
2\tau \lfloor \ell a_{n+1} \rfloor -1
> \tau \lfloor \ell a_{n+1} \rfloor.
$$

Let $K=\max\{ n: \ell a_n \ge 1\}$. Since $(a_n)$ is $(2\tau)$-good and $\ell a_0=L$,
$(1+2\tau)^{K}\le L\le \ell$, and so $K\le \log\ell /\log(1+2\tau)$.
Let $a'_n=\lfloor \ell a_n \rfloor / \ell$.
Since $f$ is $1$-Lipschitz
and $|a'_n - a_n|<1/ \ell$,
we deduce that
\[
\sum_{n=1}^{K+1} f(a'_n)-\min_{[a'_n,a'_{n-1}]}f
\le  \sum_{n=1}^{K+1} \left(f(a_n)-\min_{[a_n,a_{n-1}]}f + 2/\ell\right) \le \T(f,(a_n)) + 2(K+1)/\ell.
\]

By definition $\ell a_{K+1}<1$,
hence $a'_{K+1}=0$.
Thus
$$
\sum_{j=0}^{q-1} f(N_j/\ell)-
     \min_{[N_j/\ell,N_{j+1}/\ell]} f =
\sum_{n=1}^{K+1} f(a'_n)-
     \min_{[a'_n,a'_{n-1}]}f
\le \T(f,(a_n)) +O_{\tau}(\log\ell /\ell),
$$
which completes the proof.
\end{proof}

The following lemma will help us translate the results for Lipschitz functions to results for $[-1,1]$-sequences.
\begin{lemma}\label{l:sigmatof}
Let $\gamma,\Gamma\in[-1,1]$, $\tau\in(0,1/2)$,
$\zeta\in(0,1)$ and
let $\sigma\in [-1,1]^\ell$ satisfy
\[
\gamma j - \zeta\ell \le
\sigma_1 +\ldots + \sigma_j \le
\Gamma j + \zeta\ell  \quad (1\le j\le \ell).
\]

Then there exists a piecewise linear $1$-Lipschitz
function $f:[0,1]\to\R$ such that
\begin{enumerate}[(i)]
\item \label{i:fdef}
$f(j/\ell)=\frac{1}{\ell}(\sigma_1+\ldots+\sigma_j)$
if $\sqrt\zeta \ell \le j \le \ell$,
\item \label{i:estimates}
    $(\gamma-\sqrt\zeta) x \le f(x) \le
       (\Gamma +\sqrt\zeta) x$ on $[0,1]$ and
\item \label{i:MT}
  for any integer $0<L\le \ell$,
$$
\frac{1}{\ell} \M_\tau(\sigma|(0,L]) \le
\T(f|[0,L/\ell])
+2\sqrt{\zeta}+ 144\tau + O_{\tau}(\log\ell /\ell).
$$
\end{enumerate}
\end{lemma}

\begin{proof}
Let $f_1:[0,1]\to\R$ be the piecewise linear function which is linear on each interval $[j/\ell,(j+1)/\ell]$, and at the points $j/\ell$ takes the values
\[
f_1(j/\ell) = \frac{1}{\ell}(\sigma_1+\ldots+\sigma_j)
\qquad (j=0,1,\ldots,\ell).
\]
Since $\sigma_i\in [-1,1]$, this is a $1$-Lipschitz function. Moreover, it follows from the assumption on $\sigma$ that
\[
\gamma x - \zeta \le f_1(x) \le \Gamma x + \zeta
\quad
(x\in[0,1]),
\]
and so
\[
(\gamma - \sqrt\zeta)x \le
f_1(x) \le (\Gamma + \sqrt\zeta)x
\quad
\text{ for } x\in [\sqrt\zeta,1].
\]
Let $f$ agree with $f_1$
on $[\sqrt\zeta,1]$, $f(0)=0$ and let $f$ be
linear on $[0,\sqrt\zeta]$.
Then $f:[0,1]\to\R$ is also a piecewise linear $1$-Lipschitz function and \eqref{i:fdef} and
\eqref{i:estimates} hold.

Therefore it remains to prove \eqref{i:MT}.
Let $0<L\le \ell$ be an integer.
By Corollary~\ref{c:tau} we have
$\T_{2\tau}(f)\le \T(f)+72\tau$.
Thus it is enough to show that for any
$\de>0$ and
$(2\tau)$-good partition $(a_n)$ of $[0,L/\ell]$
there exists a $\tau$-good
integer partition $\mathcal{P}$ of $\sigma|(0,L]$ such
that
\begin{equation}\label{e:1/l M enough}
\frac{1}{\ell} \M(\sigma|(0,L], \mathcal{P}) \le
\T(f|[0,L/\ell], (a_n)) + 2\sqrt\zeta + 72\tau +
O_{\tau}(\log\ell / \ell)
 +\de.
\end{equation}

Let $N$ be the largest index such that
$a_N\ge \sqrt\zeta$.
Then $a_N\le 2a_{N+1} < 2\sqrt\zeta$.
By applying Corollary~\ref{c:tau} and
Proposition~\ref{p:basic} to
$f_1|[0,a_N]$
with $D=-1$ and $C=1$,
we get
$$
\T_{2\tau}(f_1|[0,a_N])
\le \T(f_1|[0,a_N]) + 72\tau
\le \frac{(a_N-f(a_N))\cdot 3}{6} + 72\tau
\le 2\sqrt\zeta + 72\tau.
$$
Hence for any $\delta>0$ there exists a
a $(2\tau)$-good partition $(b_n)$
of $[0,a_N]$ such that
$\T(f_1|[0,a_N],(b_n))\le 2\sqrt\zeta + 72\tau + \delta$.

Let $a'_n=a_n$ if $n\le N$ and $a'_n=b_{n-N}$
otherwise.
Then $(a'_n)$
is a $(2\tau)$-good partition of $[0,L/\ell]$
such that
\begin{equation}\label{e:Tf_1 and Tf}
\T(f_1|[0,L/\ell],(a'_n))\le
\T(f|[0,L/\ell], (a_n))
+2\sqrt\zeta + 72\tau + \delta.
\end{equation}

Applying Lemma~\ref{l:discretizing} for $f_1|[0,L/\ell]$
and the $(2\tau)$-good partition $(a'_n)$ of $[0,L/\ell]$,
we get a $\tau$-good integer partition
$0=N_0<\ldots<N_q=L$ of $(0,L]$ such that
$$
\sum_{j=0}^{q-1} f_1(N_j/\ell)-\min_{[N_j/\ell,N_{j+1}/\ell]} f_1 \le
\T(f_1|[0,L/\ell],(a_n'))
  +O_{\tau}(\log \ell/\ell).
$$
Noting that the left-hand side of the above
expression is exactly
$\frac{1}{\ell}\M(\sigma,(N_j))$,
and the right-hand side is at most the right-hand side
of \eqref{e:1/l M enough} by \eqref{e:Tf_1 and Tf},
the proof is complete.
\end{proof}

The next proposition is a version of Proposition \ref{p:basic} for sequences, and will play a central role in the proof of Theorem \ref{thm:main}.
\begin{prop} \label{prop:combinatorial}
For any $\gamma,\Gamma\in [-1,1]$,
$\tau\in (0,1/2)$, $\zeta\in(0,1)$ such that $\gamma\le\Gamma$ and $2\gamma\le \Gamma$, the following holds.

Let $\sigma\in [-1,1]^\ell$ satisfy
\[
\gamma j - \zeta\ell \le
\sigma_1 +\ldots + \sigma_j \le
\Gamma j + \zeta\ell  \quad(1\le j\le \ell).
\]
Then
\[
\frac{1}{\ell} \M_\tau(\sigma) \le
\frac{(1-\gamma)(\Gamma-2\gamma)}{1+2\Gamma-3\gamma}
+14\sqrt{\zeta}
+ 144\tau + O_{\tau}(\log\ell/\ell).
\]
\end{prop}
\begin{proof}
Let $f$ be the function provided by
Lemma~\ref{l:sigmatof}.
Applying Proposition \ref{p:basic} to
$f$ with $D=\max(-1,\gamma-\sqrt{\zeta})$,
$C=\min(1,\Gamma+\sqrt{\zeta})$,
and using that $D\in [\gamma-\sqrt{\zeta},\gamma]$ and $C\in [\Gamma,\Gamma+\sqrt{\zeta}]$,
then using that $1+2\Gamma-3\gamma\ge 1$, $\gamma,\Gamma\in[-1,1]$ and $\zeta\in(0,1)$,
we obtain
\begin{align*}
\T(f) &\le
\frac{(1-\gamma+\sqrt{\zeta})
      (\Gamma-2\gamma+3\sqrt{\zeta})}
    {1+2\Gamma-3\gamma}
 \\
&\le
\frac{(1-\gamma)(\Gamma-2\gamma)}{1+2\Gamma-3\gamma}
+12\sqrt\zeta.
\end{align*}
By applying \eqref{i:MT} of Lemma~\ref{l:sigmatof} for $L=\ell$
we get the desired inequality.
\end{proof}

The following proposition will be used (only) in the proof of Theorem \ref{thm:packing}; it is essentially a consequence of Proposition \ref{p:initial},
\begin{prop} \label{prop:combinatorial-packing}
For any $\gamma, \tau \in (0,1/2)$,
$\zeta\in (0,\gamma^2]$, $\delta>0$  there is $\eta=\eta(\delta,\gamma)>0$ such that the following holds
for any positive integer $\ell$.

Let $\sigma\in [-1,1]^\ell$ satisfy
\[
\gamma j - \zeta\ell \le
\sigma_1 +\ldots + \sigma_j \quad (1\le j\le\ell).
\]

Then there exists an integer $L\in [\eta \ell,\ell]$ such that
\[
\frac{1}{L} \M_\tau(\sigma|(0,L]) \le
\Phi(\gamma)
+\frac{1}{\eta}\left(O(\sqrt\zeta) + O(\tau) + O_\tau(\log \ell / \ell) \right) +\delta.
\]
where
$$
\Phi(x)=\frac{2-x-\sqrt{3-3x^2}}{4}.
$$
\end{prop}

\begin{proof}
Let $f$ be the function provided by
Lemma~\ref{l:sigmatof} for $\Gamma=1$.
Choose $\wt{\delta}\in(0,1/2)$ such that $\wt{\delta}(2-4\log\wt{\delta})<\delta$
and let $\eta=3\Phi(\gamma)2^{-1/\wt{\delta}}>0$.

Let  $D=\gamma-\sqrt{\zeta}\ge 0$. Note that $\Phi$ is decreasing on $[0,1/2]$, so
$3\Phi(D)2^{-1/\wt{\delta}}\ge 3\Phi(\gamma)2^{-1/\wt{\delta}} = \eta$. Using this
and applying Proposition \ref{p:initial},
we obtain a $u\in[\eta,1]$ such that
\begin{eqnarray*}
\frac{1}{u} \T(f|[0,u]) & < &
\Phi(D)+\delta =
\frac{2-D-\sqrt{3-3D^2}}{4} + \delta  \\
&\le &
\frac{2-(\gamma-\sqrt\zeta)-\sqrt{3-3\gamma^2}}{4}  +\delta
= \Phi(\gamma)+\sqrt\zeta/4 +\delta.
\end{eqnarray*}

Let $L$ be the smallest integer such that $u\le L/\ell$.
Then clearly $L\in[\eta \ell, \ell]$.

It is easy to see that for any $1$-Lipschitz function
$g:[0,a]\to\R$ and any $0<u_1<u_2\le a$ we have
$\T(g|[0,u_2])\le \T(g|[0,u_1]) + (u_2-u_1)$.
Thus
$$
\frac{1}{u}\T(f|[0,L/\ell]) \le
\frac{1}{u}(\T(f|[0,u])+1/\ell)
\le \Phi(\gamma)+\sqrt\zeta/4 +\delta+\frac{1}{\ell u}.
$$
Combining this with \eqref{i:MT} of Lemma~\ref{l:sigmatof},
we conclude
\begin{align*}
\frac{1}{L} \M_\tau(\sigma|(0,L])
&\le
\frac{1}{u} \cdot \frac{1}{\ell} \M_\tau(\sigma|(0,L]) \\
& \le
\frac{1}{u}\left(\T(f|[0,L/\ell]) + 2\sqrt\zeta + 144\tau + O_\tau(\log \ell / \ell) \right) \\
&\le
\Phi(\gamma)+\sqrt\zeta/4 +\delta+\frac{1}{\ell\eta}
+\frac{1}{\eta}\left( 2\sqrt\zeta + 144\tau + O_\tau(\log \ell / \ell) \right) \\
&\le
\Phi(\gamma) +\frac{1}{\eta}\left(O(\sqrt\zeta) + O(\tau) + O_\tau(\log \ell / \ell) \right) +\delta.
\end{align*}
\end{proof}

Finally, we get a version for sequences and integer partitions of Corollary \ref{c:37/54}, which will be applied to prove Theorem \ref{thm:full-distance-set}.
\begin{prop} \label{prop:combinatorial-wolff}
Let $\tau\in(0,1/2)$, $\gamma\in (0,1/3)$, $\zeta\in(0,\gamma^2)$ and
\[
\Lambda(x) = \frac{(1+x)(37-50x+60x^2)}{18(3-4x+5x^2)}
\qquad (x\in [0,1/3]).
\]
Then there exist $\eta>0,\xi\in (2/3,1]$ (depending on $\gamma-\sqrt\zeta$)
such that
\[
 \xi (1-2\eta)  = \Lambda(\gamma-\sqrt\zeta)\ge \Lambda(\gamma)-\sqrt\zeta
\]
and the following holds:

For any sequence $(\sigma_1,\ldots,\sigma_\ell) \in [-1,1]^\ell$ such that
\[
\sigma_1+ \ldots +\sigma_j \ge  \gamma j - \zeta \ell
\qquad (j=1,\ldots,\ell),
\]
one of the following alternatives is satisfied:
\begin{enumerate}[(i)]
\item
\[
\frac{1}{\xi\ell} \max_{j=1}^{\xi \ell} \sum_{i=1}^j \left(1/3-\sigma_i\right) \le \eta + 2\sqrt\zeta.
\]
\item
\[
\frac{1}{\ell} \M_\tau(\sigma) \le 1- \Lambda(\gamma)
+3\sqrt{\zeta}+ 144\tau + O_{\tau}(\log\ell /\ell).
\]
\end{enumerate}
\end{prop}

\begin{proof}
We begin by noting that $\Lambda(\gamma-\sqrt\zeta)\ge \Lambda(\gamma)-\sqrt\zeta$ since $\Lambda'(x)\le 1$ on $[0,1/3]$.

Let $\eta>0,\xi\in (2/3,1]$ be the numbers given in Corollary \ref{c:37/54} for $D=\gamma-\sqrt{\zeta}\in (0,1/3)$. Suppose that (i) is false, so there is a
$j\le \xi\ell$ such that
$$
\frac{1}{\xi\ell} \sum_{i=1}^j
\left(1/3-\sigma_i\right) > \eta +
2\sqrt\zeta,
$$
and therefore
\[
\frac{1}{\ell}(\sigma_1+\ldots+\sigma_j) <
\frac{j}{3\ell} - \eta \xi
- 2 \xi  \sqrt\zeta.
\]
Note that the left-hand side is at least $-j/\ell$ and that, since $\eta>0$ and $\xi\ge 2/3$, we have
$\frac34(\eta+2\sqrt\zeta)\xi > \sqrt\zeta$. This implies that $j/\ell > \sqrt{\zeta}$.

Let $f$ be the function provided by
Lemma~\ref{l:sigmatof} for $\Gamma=1$,
and let $x_0=j/\ell$.
Then $x_0\in [\sqrt\zeta,\xi]$ and
$f:[0,1]\to\R$ is a $1$-Lipschitz function
such that $f(0)=0$, $f(x)\ge (\gamma-\sqrt{\zeta})x$
on $[0,1]$, $f(x_0)<x_0/3 - \eta \xi
- 2\xi \sqrt\zeta\le x_0/3-\eta\xi$,
and
\begin{equation}\label{e:MT}
\frac{1}{\ell} \M_\tau(\sigma) \le
\T(f)
+2\sqrt{\zeta}+ 144\tau + O_{\tau}(\log\ell /\ell).
\end{equation}
Applying Corollary~\ref{c:37/54} to
$f$, we get
\begin{align*}
\T(f) \le 1-\Lambda(\gamma-\sqrt{\zeta}) \le 1- \Lambda(\gamma)+\sqrt{\zeta}.
\end{align*}
Combining this with \eqref{e:MT} we obtain (ii),
which completes the proof.
\end{proof}

\section{Proofs of main theorems}
\label{sec:proofs}

\subsection{Proof of Theorem \ref{thm:main}}
\label{subsec:proof-of-theorem-main}

In this section we prove Theorem \ref{thm:main}. Write
\[
\psi(s,u)= \frac{s(2+u-2s)}{2+2u-3s},
\]
and recall that $\chi(s,u)=\min(\psi(s,u),1)$ for $0\le s\le u\le 2$ and $s<2$, and moreover $\chi(s,u)=1$ if and only if $u\le 2s-1$ (which forces $s\ge 1$).

The next proposition encapsulates some preliminary reductions towards the proof of Theorem \ref{thm:main}. We first explain how to deduce the theorem from the proposition; the rest of the section is then devoted to the proof of the proposition.
\begin{prop} \label{prop:main}
For every $0<s\le u\le 2$ with $u>2s-1$, the following holds.

Let $\mu\in\cP([0,1)^2)$ satisfy $\cE_s(\mu)<\infty$ and $\ubdim(\supp(\mu))\le u$.  If $B\subset [0,1)^2$ is a compact set disjoint from $\supp(\mu)$ with $\hdim(B)>\min(1,2-s)$, then
\[
\sup_{y\in B} \hdim(\Delta_y (\supp\mu)) \ge \psi(s,u).
\]
\end{prop}

\begin{proof}[Proof of Theorem \ref{thm:main} (assuming Proposition \ref{prop:main})]
We proceed by contradiction. Assume, then,  that there exists a Borel set $A\subset\R^2$ such that $0< s\le \hdim(A)\le \pdim(A)\le u\le 2$ and
\[
\hdim \{ y\in\R^2 : \hdim(\Delta_y A) < \chi(s,u) \} > \max(1,2-s).
\]
By countable stability of Hausdorff dimension, there are $\eta>0$ and a set $B\subset\R^2$ with $\hdim(B)>\max(1,2-s)$ such that
\begin{equation} \label{eq:small-dist-set-for-contradiction}
\hdim(\Delta_y A) < \chi(s,u) -\eta \quad\text{for all }y\in B.
\end{equation}
Since $\hdim(\Delta_y A)$ does not increase if we replace $A$ by any subset, every Borel set of dimension $s>0$ contains compact subsets of positive $s'$-dimensional Hausdorff measure for all $0<s'<s$, and $\chi(s,u)$ is continuous, at the price of replacing $\eta$ by $\eta/2$ we may assume that  in \eqref{eq:small-dist-set-for-contradiction} the set $A$ is compact and of positive $s$-dimensional Hausdorff measure. In turn, a routine verification shows that if $A$ is compact, then the set
\[
\{ y: \hdim(\Delta_y A) < \chi(s,u)-\eta/2   \}
\]
is Borel. Hence in \eqref{eq:small-dist-set-for-contradiction} we may also assume that $B$ is Borel.

Recall that $\chi(s,u)= 1$ if and only if  $u\le 2s-1$ (and $s\ge 1$ in this case). Hence, if $u\le 2s-1$, then we can pick $0< s'\le u'\le 2$ such that $u'\ge u$, $1\le s'\le s$, $u'>2s'-1$, and $\psi(s',u')>1-\eta/2$. This shows that in \eqref{eq:small-dist-set-for-contradiction}, we may further assume that $u>2s-1$ and replace $\chi(s,u)$ by $\psi(s,u)$ (with $\eta/2$ in place of $\eta$).

Let $\mu\in\cP(\R^2)$ be an $s$-Frostman measure on $A$, i.e. $\mu$ is a Radon measure supported on $A$ and $\mu(B(x,r)) \le C r^s$ for all $x\in \R^2$, $r>0$, where $C$ is independent of $x$ (recall that we assumed that $A$ has positive $s$-dimensional Hausdorff measure). By assumption, $\pdim(\supp(\mu)) \le u$. Using that packing dimension is equal to the modified upper box counting dimension (see e.g. \cite[Proposition 3.8]{Falconer14}), and that $\ubdim(A_0)=\ubdim(\overline{A}_0)$, we see that for every $\delta>0$ there is a compact set $A_0\subset A$ of positive $\mu$-measure such that $\ubdim(A_0)\le \min(u+\delta,2)$.

We can then find disjoint compact subsets $B'\subset B, A'\subset A_0$ such that still $\mu(A')>0$, $\hdim(B')>\max(1,2-s)$. Then (provided $\delta$ was taken small enough in terms of $s,u,\eta$)
\[
\hdim(\Delta_y A') < \psi(s-\delta,\min(u+\delta,2)) - \eta/2 \quad\text{for all }y\in B'.
\]
This inequality is preserved under (joint) scaling and translation of $\mu, A',B'$, so it holds in particular for some compact $A', B'\subset [0,1)^2$.  Since $\mu_{A'}(B(x,r)) \le C' r^s$ for some constant $C'>0$, we can check that $\cE_{s-\delta}(\mu_{A'})<\infty$. Since $\supp(\mu_{A'})\subset A'$, this contradicts Proposition \ref{prop:main} applied to $\mu_{A'}$ and $B'$, with $s-\delta$, $\min(u+\delta,2)$ in place of $s, u$ (provided $\delta$ was taken small enough in terms of $s,u,\hdim(B')$).
\end{proof}

In order to bound the Hausdorff dimension of $\Delta_y A$ from below, we will use the following standard criterion; although it is well known, we include the short proof for completeness.
\begin{lemma} \label{lem:calculation-dimension}
Let $F\subset\R^d$ be a Borel set and let $\rho\in\cP(\R^d)$ give full mass to $F$. Suppose that there are $M_0\in\N_{\ge 2}$ and $s>0$ such that for any $M\ge M_0$ and any Borel subset $F'\subset F$ with $\rho(F')> M^{-2}$, the number $\cN(F',M)$ of cubes in $\cD_M$ hitting $F'$ is at least $2^{s T M}$. Then $\mathcal{H}^s(F)\gtrsim_{T,d} 1$ and in particular $\hdim(F)\ge s$.
\end{lemma}
\begin{proof}
Let $\{ B(x_i,r_i)\}$ be a cover of $F$ where $r_i \le 2^{-T M_0}$ for all $i$. Our goal is to estimate $\sum_i r_i^s$ from below.

Write $F_M$ for the union of all the $B(x_i,r_i)$ for which $2^{- T(M+1)} \le r_i \le 2^{-TM}$. Pigeonholing, there is $M\ge M_0$ such that $\rho(F_M)> M^{-2}$. By assumption, one needs at least $2^{sTM}$ cubes in $\cD_M$ to cover $F_M$. It follows that the number of balls making up $F_M$ is $\gtrsim_{d,T} 2^{sTM}$, so that $\sum_i r_i^s \gtrsim_{d,T} 2^{s T M}  2^{-s T M}=1$. This gives the claim.
\end{proof}

We now begin the proof of Proposition \ref{prop:main}. Since $\mu, s,u$ are fixed, any (possibly implicit) constants appearing in the proof may depend on them. Let $\nu$ be a measure supported on $B$ with finite $u'$-energy where $u'>\max(1,2-s)$.  Let $\kappa=\kappa(\mu,\nu)>0$ be the number given by Proposition \ref{prop:radial}. We will show that (under the assumptions of the proposition) there exists $y\in B$ (possibly depending on $T,\e,\tau$) such that
\begin{equation} \label{eq:lower-bound-dim-to-prove}
\hdim(\Delta_y (\supp(\mu))) > \psi(s,u) - o_{T,\e,\tau}(1).
\end{equation}
Recall that $o_{T,\e,\tau}(1)$ stands for a function of $T,\e,\tau$ which tends to $0$ as $T\to\infty$ and $\e,\tau\to 0^+$. We will henceforth assume that $T,\e,\tau$ are given, and that the integer $\ell_0$ is chosen large enough in terms of $T,\e,\tau$ so that all the claimed inequalities hold. As a first instance of this, apply Lemma \ref{lem:badx} to get  that
\begin{align*}
|\bad''_{\ell_0}(\mu,x)| &\le \kappa  \text{ for all } x\in\supp_{\mathsf{d}}(\mu)
\end{align*}
provided $\ell_0$ was taken large enough (in terms of $T,\e,\tau$).

One can easily check that, given $\nu\in\cP([0,1)^2)$ and $j,k\in\N$, the set  $\{ (x,\theta): \theta\in \bad(\nu,x,j,k)\}$ is Borel (recall Definition \ref{def:bad-projections}). It follows that the set
\begin{equation} \label{eq:Borel-set}
\Theta = \{ (x,\theta): x\in\supp_{\mathsf{d}}(\mu), \theta\in \bad''_{\ell_0}(\mu,x) \}
\end{equation}
is Borel. Hence, applying Proposition \ref{prop:radial}, and using Fubini and the fact that $\mu$ is a Radon measure, we obtain a compact set $A_1\subset \supp_{\mathsf{d}}(\mu)\subset\supp(\mu)$ with $\mu(A_1)>2/3$ and a point $y\in \supp(\nu)\subset B$ such that
\begin{equation} \label{eq:good-projection}
P_y(x) \notin \bad''_{\ell_0}(\mu,x) \text{ for all } x\in A_1.
\end{equation}
Making $\e$ smaller (in terms of $\dist(B,\supp(\mu))$ only) and $\ell_0$ larger, we may assume that
\begin{equation} \label{eq:separation}
 \dist(B,A_1) \ge \e + \sqrt{2} \cdot 2^{-\ell_0}.
\end{equation}
We will show that, in fact, $\hdim(\Delta_y(A_1))\ge \psi(s,u)-o_{T,\e,\tau}(1)$, which clearly implies \eqref{eq:lower-bound-dim-to-prove}. To do this, our aim is to apply Lemma \ref{lem:calculation-dimension} with $F=\Delta_y(A_1)$, $\rho=\Delta_y(\mu_{A_1})$.  Note that if $\rho(F')\ge \ell^{-2}$, then $A_2=\Delta_y^{-1}(F')$ satisfies that $\mu_{A_1}(A_2)=\rho(F')\ge \ell^{-2}$. Hence, in order to complete the proof of Proposition \ref{prop:main}, it is enough to establish the following.

\textbf{Claim}. If the Borel set $A_2\subset [0,1)^2$ satisfies $\mu_{A_1}(A_2) \ge \ell^{-2}$ with $\ell\ge \ell_0$, where $\ell_0$ is taken sufficiently large in terms of $T,\e,\tau$, then
\begin{equation} \label{eq:lower-bound-covering-number}
\log \cN(\Delta_y A_2,\ell) \ge (\psi(s,u)-o_{T,\e,\tau}(1))T\ell.
\end{equation}

Fix, then, $A_2$ as above. Since the set $\Delta_y(R_\ell A_2)$ is contained in the $(\sqrt{2}\cdot 2^{-T\ell})$-neighborhood of $\Delta_y A_2$, the numbers $\log \cN(\Delta_y A_2,\ell)$ and $\log \cN(\Delta_y  R_\ell A_2,\ell)$ differ by at most a constant. Hence we can, and do, assume that $A_2=R_\ell A_2$ from now on. Moreover, we may assume that $A_2\subset R_\ell(A_1)$, since whenever $A_2=R_\ell A_2$ and $\mu_{A_1}(A_2) \ge \ell^{-2}$, the same holds for $A_2 \cap R_\ell(A_1)$.

Consider the sets given by Corollary \ref{cor:bourgain} applied to $R_{\ell}\mu$. Applying the corollary with $A=A_2$, and using that
\[
2^{-\e T\ell}\ll  \tfrac{2}{3}\ell^{-2} \le  \mu(A_1)\mu_{A_1}(A_2) \le \mu(A_2)=R_\ell\mu(A_2)
\]
for large enough $\ell$, we can find a further $2^{-T\ell}$-set $X$ such that, setting $\rho=(R_\ell\mu)_X$,
\begin{enumerate}[(\rm i)]
\item $\rho(A_2) \ge \ell^{-2}/2$.
\item $R_\ell\mu(X) \ge 2^{-o_{T,\e}(1) T\ell}$ and therefore, using that $\cE_s(R_\ell\mu) \lesssim_T \cE_s(\mu)$ by  Lemma \ref{lem:energy},
\[
\cE_s(\rho) \le (R_\ell\mu(X))^{-2}\cE_s(R_\ell\mu) \lesssim_T 2^{o_{T,\e}(1) T\ell} \cE_s(\mu) \lesssim 2^{o_{T,\e}(1) T\ell}.
\]
\item $\rho$ is $\sigma$-regular for some sequence $\sigma=(\sigma_1,\ldots,\sigma_\ell)$, $\sigma_j\in [-1,1]$.
\item $X$ is contained in $R_\ell\supp_{\mathsf{d}}(\mu)$.
\end{enumerate}

By Lemma \ref{lem:energy-regular} and (ii), (iii) above,  and assuming that $\ell_0$ was taken large enough in terms of $T$, we have
\begin{equation} \label{eq:sigma-lower}
  \sum_{i=1}^j \sigma_i \ge (s-1)j - \ell o_{T,\e}(1) \qquad(j=1,\ldots,\ell).
\end{equation}
On the other hand, we have assumed that $\ubdim(\supp(\mu))\le u$, so that $\mathcal{N}(\supp(\mu),j) \le O_\e(1) 2^{(u+\e)T j}$ for all $j\in\N$. By (iv) above, this also holds for $X$ in place $\supp(\mu)$ if $j\le \ell$. On the other hand, using that $\rho$ is $\sigma$-regular as in \eqref{eq:box-counting-regular}, we get
\[
\mathcal{N}(X,j) = |\cD_j(\rho)| \ge 2^{T(\sigma_1+1)} \cdots 2^{T(\sigma_j+1)}  \quad (1\le j\le  \ell).
\]
Combining these estimates, we deduce that $2^{T(\sigma_1+\ldots+\sigma_j+j)} \le  O_\e(1) 2^{(u+\e)T j}$, and hence
\begin{equation} \label{eq:sigma-upper}
  \sum_{i=1}^j \sigma_i  \le \frac{O_\e(1)}{T} + (u-1+\eps)j \le
(u-1)j+\ell o_{T,\eps} (1) \qquad (j=1,\ldots,\ell),
\end{equation}
provided $\ell_0$ was taken large enough in terms of $\e$.

Combining \eqref{eq:sigma-lower} and \eqref{eq:sigma-upper}, we see that the assumptions of Proposition \ref{prop:combinatorial} are satisfied with $\gamma=s-1$, $\Gamma=u-1$, and $\zeta=o_{T,\e}(1)$. After another short calculation, and starting with $\ell_0$ large enough in terms of $\tau$, we deduce that
\begin{equation} \label{eq:good-partition}
\frac{1}{\ell} \M_\tau(\sigma) \le  1 - \psi(s,u) + o_{T,\e,\tau}(1).
\end{equation}

Recall from \eqref{eq:good-projection} that if $x\in  A_1$, then $\theta(x,y)\notin \bad''_{\ell_0}(\mu,x)$. Hence, according to the definition of the sets $\bad'_{\ell_0\ssto\ell}(R_\ell\mu,x)$ and $\bad''_{\ell_0}(\mu,x)$ in \eqref{eq:def-badprime} and \eqref{eq:def-badprimeprime} respectively, we have $\theta(x,y)\notin \bad'_{\e\ell\ssto\ell}(R_\ell\mu,x)=\bad_{\e\ell\ssto\ell}(\rho,x)$ for all $x\in A_1\cap X$. Since we have assumed that $A_2\subset R_\ell(A_1)$, the hypotheses of Proposition \ref{prop:lower-bound-box-counting} are met by $\rho$ and $A_2$, with $\beta=\eps$ (the separation assumption follows from \eqref{eq:separation}). Recalling (i), we see that if $\ell_0$ was taken even larger in terms of $T,\e,\tau$ we can make the error term in Proposition \ref{prop:lower-bound-box-counting} equal to $o_{T,\e,\tau}(1)$. In light of  \eqref{eq:good-partition}, Proposition \ref{prop:lower-bound-box-counting} gives exactly \eqref{eq:lower-bound-covering-number}.

This completes the proof of the claim and, with it, of Proposition \ref{prop:main} and Theorem \ref{thm:main}.

\subsection{Proof of Theorem \ref{thm:packing}}
\label{subsec:proof-of-theorem-packing}

In this section we prove Theorem \ref{thm:packing}. The proof goes along the same lines as the proof of Theorem \ref{thm:main}, except that we rely on Proposition \ref{prop:combinatorial-packing} instead of Proposition \ref{prop:combinatorial} to choose the scales in the multi-scale decomposition. The need to deal with two different scales $2^{-TL}$ and $2^{-T\ell}$ also creates some additional challenges. Write
\[
\psi(s) = \frac{1+s+\sqrt{3s(2-s)}}{4}.
\]
(This should not be confused with the function $\psi(s,u)$ from \S\ref{subsec:proof-of-theorem-main}.) The next proposition contains the core of Theorem \ref{thm:packing}. 

\begin{prop} \label{prop:packing}
For every $1<s<3/2$, the following holds.

Let $\mu\in\cP([0,1)^2)$ satisfy $\cE_s(\mu)<\infty$.  If $B\subset [0,1)^2$ is a compact set disjoint from $\supp(\mu)$ with $\hdim(B)>1$, then
\[
\sup_{y\in B} \ubdim(\Delta_y (\supp\mu)) \ge \psi(s).
\]
\end{prop}

\begin{proof}[Proof of Theorem \ref{thm:packing} (assuming Proposition \ref{prop:packing})]
Reasoning as in the deduction of Theorem \ref{thm:main} from Proposition \ref{prop:main}, we get that if $U$ is a Borel subset of $\R^2$ with $\hdim(U)\ge t\in (1,3/2)$, then
\begin{equation} \label{eq:claim-box-dim}
\hdim\{ y\in\R^2: \ubdim(\Delta_y(U)) < \psi(t) \} \le 1.
\end{equation}
The reason we need to go via box dimension is that the map $y\mapsto \ubdim\Delta_y(U)$ is Borel if $U$ is compact, while it is unclear whether the map $y\mapsto \pdim\Delta_y(U)$ is Borel, since it was proved in \cite{MattilaMauldin97} that packing dimension is \emph{not} a Borel function of the set if one considers the Hausdorff metric on the compact subsets of $\R$.

Now suppose the claim of Theorem \ref{thm:packing} does not hold. Then we can find a Borel set $A\subset\R^2$ with $\hdim(A)=s\in (1,3/2)$ and $\eta>0$ such that
\[
\hdim\{ y\in \R^2: \pdim(\Delta_y(A)) < \psi (s)-\eta \} > 1.
\]
Let $\nu$ a Frostman measure on $A$ of exponent $t\in (1,s)$, sufficiently close to $s$ that $\psi(t)\ge \psi(s)-\eta$, and note that
\begin{equation} \label{eq:claim-packing-for-contradiction}
\hdim\{ y\in \R^2: \pdim(\Delta_y(\supp(\nu))) < \psi (t) \} > 1.
\end{equation}
Fix a countable basis $(U_i)$ of open sets of $\supp(\nu)$ (in the relative topology). Note that $\hdim(U_i) \ge t$ for all $i$ since $\nu$ is a Frostman measure. Hence, from \eqref{eq:claim-box-dim} we get that $\hdim(E)\le 1$, where
\[
E = \{ y \in\R^2: \ubdim(\Delta_y(U_i)) < \psi(t) \text{ for some } i\}.
\]
Fix $y\in\R^2\setminus E$. Let $(F_j)$ be a countable cover of $\Delta_y(\supp(\nu))$. By Baire's Theorem, some $\Delta_y^{-1}(\overline{F}_j)$ has nonempty interior in $\supp(\nu)$, and hence contains some $U_i$. By the definition of $E$,
\[
\ubdim(F_j) = \ubdim(\overline{F}_j) \ge \ubdim(\Delta_y(U_i)) \ge \psi(t).
\]
By the characterization of packing dimension as modified upper box counting dimension (\cite[Proposition 3.8]{Falconer14}), we conclude that $\pdim(\Delta_y(\supp(\nu)) \ge \psi(t)$ whenever $y\in\R^2\setminus E$. Since $\hdim(E)\le 1$, this contradicts \eqref{eq:claim-packing-for-contradiction}, finishing the proof.
\end{proof}

We now start the proof of Proposition \ref{prop:packing}. Let $\nu$ be a measure supported on $B$ with finite $u$-energy for some $u>1$, and let $\kappa=\kappa(\mu,\nu)>0$ be the number given by Proposition \ref{prop:radial}.  Apply Lemma \ref{lem:badx} to obtain the bound
\begin{align*}
|\bad''_{\ell_0}(\mu,x)| &\le \kappa  \text{ for all } x\in\supp_{\mathsf{d}}(\mu)
\end{align*}
provided $\ell_0$ was taken large enough in terms of $T,\e,\tau$. Recall that the set $\Theta$ in Equation \eqref{eq:Borel-set} is Borel. Applying Proposition \ref{prop:radial} to $\Theta$ and Fubini, we obtain a compact set $A\subset \supp_{\mathsf{d}}(\mu)$ with $\mu(A)> 2/3$ and a point $y\in B$ such that
\begin{equation} \label{eq:directions-good}
P_y(x) \notin \bad''_{\ell_0}(\mu,x) \text{ for all } x\in A.
\end{equation}

Fix a number $\delta>0$. We will show that
\[
\ubdim(\Delta_y A) \ge \psi(s)- \error_{T,\e,\tau}(\delta),
\]
where $\error_{T,\e,\tau}(\delta)$ can be made arbitrarily small by first taking $\delta$ small enough, and then taking $T$ large enough and $\e,\tau$ small enough, all in terms of $\delta$. This error term may also depend on $s$.

Fix a large integer $\ell \gg \ell_0$. We claim that it is enough to find a scale $L\in [\ell_0,\ell]$, tending to infinity with $\ell$, such that
\begin{equation} \label{eq:box-counting-to-show}
\frac{\log\mathcal{N}(\Delta_y(R_L A),L)}{TL} \ge \psi(s) - \error_{T,\e,\tau}(\delta),
\end{equation}
where the error term has property detailed above. Indeed, since $\Delta_y(R_L A)$ is contained in the $O(2^{-TL})$-neighborhood of $\Delta_y(A)$, this implies the corresponding lower bound for $\ubdim(\Delta_y(A))$.

Apply Corollary \ref{cor:bourgain} to $R_{\ell}\mu$.  Taking $\ell$ large enough that $2^{-\e T\ell}\ll 2/3<\mu(A)\le \mu(R_\ell A)$, there is a $2^{-T\ell}$-set $X$  such that, setting $\rho=(R_\ell\mu)_{X}$,
\begin{enumerate}[(\rm i)]
\item $\rho(R_\ell A) \ge 1/2$.
\item $R_\ell\mu(X)   \ge 2^{-o_{T,\e}(1) T\ell}$ whence, as we saw in the proof of Proposition \ref{prop:main},
\[
\cE_s(\rho) \lesssim_T 2^{o_{T,\e}(1) T\ell}.
\]
\item $\rho$ is $\sigma$-regular for some sequence $\sigma=(\sigma_1,\ldots,\sigma_\ell)$, $\sigma_j\in [-1,1]$.
\end{enumerate}
Note that, provided $\ell_0$ was taken large enough, \eqref{eq:sigma-lower} still holds, since it only depends on (ii) and (iii). We are then in the setting of Proposition \ref{prop:combinatorial-packing} with $\gamma=s-1$, $\zeta=o_{T,\e}(1)$. Let $\eta=\eta(\delta,s-1)>0$ be the number given by the proposition; we underline that, since $\delta$ is chosen before $T,\e,\tau$, the number $\eta$ is also independent of $T,\e,\tau$ (it is useful to keep in mind that $\eta$ does depend on $\delta$). A short calculation shows that $1-\psi(s)=\Phi(s-1)$. From now we assume that $\ell$ is taken large enough (in terms of $\delta$ and $s$) that $\eta\ell \ge \ell_0$. Then, applying Proposition \ref{prop:combinatorial-packing} and making $\ell$ even larger, we get an integer $L \in [\eta\ell,\ell]\subset [\ell_0,\ell]$ such that
\begin{equation} \label{eq:good-partition-initial-segment}
\frac{1}{L} \M_\tau(\sigma|(0,L]) \le 1-\psi(s) + \eta^{-1} o_{T,\e,\tau}(1) +\delta.
\end{equation}

Note that $R_L\rho$ is $(\sigma_1,\ldots,\sigma_L)$-regular. Also, if $x\in A\cap X$, then
\begin{align*}
\theta(x,y) &\notin \bad'_{\e\ell\ssto\ell}(R_\ell\mu,x) & \text{(by \eqref{eq:directions-good} and \eqref{eq:def-badprimeprime})}\\
& = \bad_{\e\ell\ssto\ell}(\rho,x) & \text{(by  \eqref{eq:def-badprime}, since $X$ came from Cor. \ref{cor:bourgain})}  \\
& \supset \bad_{(\e/\eta) L\ssto L}(\rho,x) & \text{(by Def.  \ref{def:bad-projections}, $\e\ell\le (\e/\eta)L$ and $L\le \ell$)}
 \\
& = \bad_{(\e/\eta) L\ssto L}(R_L\rho,x) & \text{(by Def.  \ref{def:bad-projections})}.
\end{align*}
Note that $x\mapsto \bad_{(\e/\eta) L\ssto L}(R_L \rho,x)$ is constant on each square of $\cD_L(R_L\rho)$. Hence, for each $x\in R_L(A\cap X)$ there is $\wt{x}\in A\cap X\subset R_L(A\cap X)$ such that $\theta(\wt{x},y)\notin \bad_{(\e/\eta)L\ssto L}(R_L\rho,\wt{x})$. Assume $\e<\eta$. If $\eps$ was taken small enough and $\ell$  large enough that $\dist(B,\supp(\mu))\ge \e +  \sqrt{2}\cdot 2^{-\ell_0}$, then all the hypotheses of Proposition \ref{prop:lower-bound-box-counting} are satisfied for $R_L\rho$, $R_L(A\cap  X)$ and $L$ in place of $\rho, A$ and $\ell$, with $\beta=\e/\eta$. Using (i) above (which implies $R_L\rho(R_L A) \ge 1/2$) and the bound $L\ge\eta\ell$, the error term in Proposition \ref{prop:lower-bound-box-counting} can be bounded by
\[
 \frac{2\eps}{\eta} +  o_{T,\e}(1) + O_{T,\e,\tau}\left(\frac{\log^2(\eta \ell)}{\eta\ell}\right) .
\]
Making $\ell$ large enough in terms of $T,\e,\tau,\delta$ and $s$, this error term can be made $o_{T,\e}(1)+2\eps \eta^{-1}$. Hence Proposition \ref{prop:lower-bound-box-counting} together with \eqref{eq:good-partition-initial-segment} ensure that \eqref{eq:box-counting-to-show} holds, with the error behaving as claimed.

Since $L\ge \eta \ell$ and $\ell$ is arbitrarily large, $L$ is also arbitrarily large. Hence we have shown that \eqref{eq:box-counting-to-show} holds for arbitrarily large $L$ and, as explained above, this completes the proof of Proposition \ref{prop:packing} and, with it, of Theorem \ref{thm:packing}.

\subsection{Proof of Theorem \ref{thm:full-distance-set}}

In this section we prove Theorem \ref{thm:full-distance-set}. Throughout this section, we let $\Delta:\R^4\to\R$, $(x,y)\mapsto |x-y|$. We start by recalling a more quantitative version of the Mattila-Wolff bound \eqref{eq:bound-wolff}.

\begin{theorem} \label{thm:mattila-wolff}
Suppose $\mu_1,\mu_2\in\cP([0,1)^2)$ have $\e$-separated supports. If $\cE_{4/3}(\mu_1)<+\infty$, $\cE_{4/3+\e}(\mu_2)<\infty$, then $\Delta(\mu_1\times\mu_2)$ has an $L^2$ density, and
\[
\|\Delta(\mu_1\times \mu_2)\|_2^2  \lesssim_\e \cE_{4/3}(\mu_1) \cE_{4/3+\e}(\mu_2).
\]
\end{theorem}
\begin{proof}
Given $\mu\in\cP([0,1)^2)$, let
\begin{align*}
\boldsymbol{\sigma}(\mu,r) &= \int_{S^1} |\widehat{\mu}(\theta r)|^2 \,d\theta,\\
\boldsymbol{\sigma}_\alpha(\mu) &= \sup \{ r^\alpha \boldsymbol{\sigma}(\mu,r) : r> 0 \}.
\end{align*}
Mattila \cite[Corollary 4.9]{Mattila87} proved that
\[
\|\Delta(\mu_1\times \mu_2)\|_2^2 \lesssim_\e \cE_\alpha(\mu_1) \boldsymbol{\sigma}_{2-\alpha}(\mu_2).
\]
We remark that in \cite{Mattila87} this is proved for a weighted version of the distance measure (see \cite[Eq. (4.1)]{Mattila87}), but the weight $u^{-1/2}$ lies in the interval $[(\sqrt{2})^{-1/2}, \e^{-1/2}]$ by our assumption that the supports of $\mu_1,\mu_2$ are $\e$-separated and contained in $[0,1)^2$. Later Wolff \cite[Theorem 1]{Wolff99} proved that for any $\mu\in\cP([0,1)^2)$,
\[
\boldsymbol{\sigma}_{\beta/2}(\mu) \lesssim_{\beta,\e} \cE_{\beta+\e}(\mu),
\]
and this is sharp up to the $\e$ when $\beta\in (1,2)$. See also \cite[Chapters 15 and 16]{Mattila15} for an exposition of these arguments. Combining these estimates with $\alpha=\beta=4/3$ yields the claim.
\end{proof}

In the proof we will also require the following well-known lemma, whose proof we include for completeness.
\begin{lemma} \label{lem:box-counting-from-L2-norm}
Let $f\in L^2(\R)$ satisfy $\int f dx=1$. Then $\cN(\supp(f),L) \ge 2^{T L}/\|f\|_2^2$ for all $L\in\N$.
\end{lemma}
\begin{proof}
Using Cauchy-Schwarz and Jensen's inequality, we estimate
\begin{align*}
1 &= \left( \sum_{I\in\cD_L} \int_I f \right)^2 \\
&\le \cN(\supp(f),L) \sum_{I\in\cD_L} \left(\int_I f\right)^2 \\
&\le \cN(\supp(f),L) \sum_{I\in\cD_L} 2^{-T L} \int_I f^2\\
&= 2^{-T L} \cN(\supp(f),L) \|f\|_2^2.
\end{align*}
\end{proof}

\begin{proof}[Proof of Theorem \ref{thm:full-distance-set}]
 As usual fix $T\gg 1, \e,\tau\ll 1$. Let $\Lambda(x)$ be the function defined in Proposition \ref{prop:combinatorial-wolff}. A calculation shows that
\[
\Lambda(s-1) =\frac{s(147-170s+60s^2)}{18(12-14s+5s^2)}.
\]
As $\Lambda$ is continuous, it is enough to show that if $A\subset\R^2$ is a Borel set with $\hdim(A)>s$, then $\hdim(\Delta(A\times A))\ge \Lambda(s-1)$. It is enough to consider the case in which $A$ is bounded. After translating and rescaling $A$, we may further assume that $A\subset [0,1)^2$.

Let $\mu_1,\mu_2\in\cP([0,1)^2)$ be measures supported on $A$ such that $\cE_{s}(\mu_1),\cE_{s}(\mu_2)<\infty$, and their supports are $(2\e)$-separated (making $\e$ smaller if needed). Any implicit constants arising in the proof may depend on $\mu_1$, $\mu_2$ and $s$.

Let $\kappa_1,\kappa_2>0$ be the numbers given by Proposition \ref{prop:radial} applied to $\mu_1,\mu_2$ and $\mu_2,\mu_1$ in place of $\mu,\nu$ respectively, and set $\kappa=\min(\kappa_1,\kappa_2)$.

Pick $\ell_0$ large enough in terms of $T,\e,\tau$ that, invoking Lemma \ref{lem:badx},
\[
\bad''_{\ell_0}(\mu_i,x) \le \kappa \text{ for all } x\in\supp_{\mathsf{d}}(\mu_i),\quad i=1,2.
\]
Let
\[
\Theta_i = \{ (x,\theta): x\in\supp_{\mathsf{d}}(\mu_i), \theta\in \bad''_{\ell_0}(\mu_i,x)\}, \quad i=1,2.
\]
Applying Proposition \ref{prop:radial} first with  $\mu_1,\mu_2$ and $\Theta_1$ in place of $\mu,\nu,\Theta$ and then with $\mu_2,\mu_1$ and $\Theta_2$ in place of $\mu,\nu,\Theta$, we get that there exists a compact set $G\subset\supp_{\mathsf{d}}(\mu_1)\times\supp_{\mathsf{d}}(\mu_2)$ such that $(\mu_1\times\mu_2)(G)>1/3$ and
\begin{equation} \label{eq:def-G}
\theta(x,y)\not\in \bad''_{\ell_0}(\mu_1,x)\,\text{ and } \,\theta(y,x)\not\in \bad''_{\ell_0}(\mu_2,y)\quad\text{for all }(x,y)\in G.
\end{equation}
We write $\mu=\mu_1\times\mu_2$ from now on. Denote $s_0=\Lambda(s-1)$. Our goal is to show that
\[
\hdim(\Delta(G)) \ge s_0.
\]
Since $\Delta(G)\subset \Delta(A\times A)$, this will establish the theorem. In turn, since a Borel set $F'\subset\R$ satisfies $(\Delta\mu_G)(F')> \ell^{-2}$ if and only if $B=\Delta^{-1}(F')$ satisfies $\mu_G(B)> \ell^{-2}$, according to Lemma \ref{lem:calculation-dimension}, in order to complete the proof it is enough to prove the following claim.

\textbf{Claim}. The following holds if $\ell$ is large enough in terms of $\mu,T,\e,\tau$: if $B$ is a Borel subset of $[0,1)^2\times [0,1)^2$ such that $\mu_G(B) > \ell^{-2}$, then
\begin{equation} \label{eq:claim-full-dist-thm}
\log\cN(\Delta(B),\ell) \ge T\ell (s_0-o_{T,\e,\tau}(1)).
\end{equation}

We start the proof of the claim. Firstly, replacing $B$ by a compact subset of almost the same measure we may assume that $B$ is compact. We may assume also that $B\subset G$. Note that $\mu(B) = \mu_G(B)\mu(G)\ge \ell^{-2}/3$.

Let $(X_k^{(i)})_{k=1}^{N_i}$ be the $2^{-T\ell}$-sets given by Corollary \ref{cor:bourgain} applied to  $R_\ell(\mu_i)$. Note that we have a disjoint union
\begin{equation} \label{eq:decomposition-into-regular}
\supp_{\mathsf{d}}(R_\ell\mu_i)= \left( \bigcup_{k=1}^{N_i} X_k^{(i)} \right) \cup \wt{X}^{(i)},\quad\text{where }\mu_i(\wt{X}^{(i)}) =R_\ell\mu_i(\wt{X}^{(i)}) \le 2^{-\e T\ell}.
\end{equation}

Write $\rho_k^{(i)} =  (R_\ell\mu_i)_{X_k^{(i)}}$. Note that $\rho_k^{(i)}$ is  $\sigma_k^{(i)}$-regular for some $\sigma_k^{(i)} \in [-1,1]^\ell$; in particular, it is a $2^{-T\ell}$-measure. Also, by Lemma \ref{lem:energy} and Corollary \ref{cor:bourgain}(ii), and using our assumption that $\cE_{s}(\mu_i)<\infty$,
\[
\cE_{s}(\rho_k^{(i)}) \le \left(R_\ell\mu_i(X_k^{(i)})\right)^{-2} \cE_{s}(R_\ell\mu_i) \lesssim_T 2^{o_{T,\e}(1)T\ell}.
\]
Hence, using Lemma \ref{lem:energy-regular} and increasing the value of $\ell_0$ again,  any $\sigma=\sigma_k^{(i)}$ satisfies
\begin{equation} \label{eq:sigma-positive}
\sigma_1+\ldots+\sigma_j \ge (s-1)j - \zeta\ell  \qquad(j=1,\ldots,\ell),
\end{equation}
where $\zeta=o_{T,\e}(1)$. By starting with appropriate $T,\e$, we may assume that $\zeta<(s-1)^2$.

If $\rho$ is $\sigma$-regular, we write
\[
\D(\rho) =  1 - \frac{1}{\ell} \M_\tau(\sigma).
\]
Let $F_i\subset\supp_{\mathsf{d}}(R_\ell\mu_i)$ be union of the sets $X_k^{(i)}$ over all $k$ such that $\D(\rho_k^{(i)})\ge s_0-\delta$, where
\begin{equation} \label{eq:def-delta}
\delta = 3\sqrt{\zeta} + 145\tau.
\end{equation}
Note that, since the $X_k^{(i)}$ are $2^{-T\ell}$-sets, then so if $F_i$.

Consider two (non mutually exclusive) cases:
\begin{enumerate}[(a)]
\item Either $\mu_{B}(F_1\times\R^2)\ge 1/3$ or $\mu_{B}(\R^2 \times F_2)\ge 1/3$ (or both).
\item $\mu_{B}((\R^2\setminus F_1)\times (\R^2\setminus F_2))\ge  1/3$.
\end{enumerate}
Roughly speaking, in the first case we will argue as in the proof of Theorem \ref{thm:main}, while in case (b) we will appeal to Proposition \ref{prop:combinatorial-wolff}.

Assume then that (a) holds. Without loss of generality, suppose $\mu_{B}(F_1\times\R^2)\ge 1/3$. Instead of showing that \eqref{eq:claim-full-dist-thm} holds directly for $B$, we will show that it holds for the set
\[
B' = \bigcup_{y\in [0,1)} (R_\ell B_y \times \{y\})
\]
where, for the rest of this section, given $A\subset\R^2\times\R^2$ we denote its ``horizontal'' sections by $A_y = \{ x: (x,y)\in A\}$ (for $y\in\R^2$). In other words, to form $B'$ we make each horizontal fiber of $B$ into a union of squares in $\cD_\ell$. One can check that $B'$ is Borel (in fact, $\sigma$-compact). Since $B\subset B'\subset R_\ell B$, the numbers $\cN(\Delta(B'),\ell)$ and $\cN(\Delta(B),\ell)$ differ by at most a multiplicative constant so that proving \eqref{eq:claim-full-dist-thm} for $B'$ implies it also for $B$. Since we are assuming that $B\subset G$, we have that $B'\subset G'$, where $G'$ is defined analogously to $B'$.

Using Fubini, that $F_1 = R_\ell F_1$, our definition of $B'$, the assumption $\mu_{B}(F_1\times\R^2)\ge 1/3$, and the fact that $\mu(B)\ge \ell^{-2}/3$, we get
\begin{align*}
(R_\ell \mu_1 \times \mu_2) ((F_1\times\R^2) \cap B') &=
\int R_\ell \mu_1 (F_1 \cap B'_y) \,d\mu_2 (y) \\ &=
\int  \mu_1 (F_1 \cap R_\ell B_y) \,d\mu_2 (y) \\ &=
\mu ( (F_1 \times \R^2) \cap B' ) \\ &\ge
\mu_B (F_1 \times \R^2) \mu (B) \ge
\ell^{-2} / 9.
\end{align*}

Applying \eqref{eq:decomposition-into-regular} with $i=1$, we can decompose
\[
R_\ell\mu_1 = (R_\ell\mu_1)|_{\wt{X}^{(1)}} + \sum_{k=1}^{N_1} \mu_1(X_k^{(1)}) \rho_k^{(1)}.
\]
Hence, using that $R_\ell \mu_1(\widetilde X^{(1)}) \le 2^{-\eps T\ell}$ and taking $\ell$ large enough, there exists $k$ such that
\[
(\rho_k^{(1)}\times \mu_2)( (F_1\times\R^2) \cap B') \ge \ell^{-2}/9 - 2^{-\e T\ell} \ge \ell^{-2}/10.
\]
By the definition of $F_1$, we must have $\D(\rho_k^{(1)})\ge s_0-\delta$, and  $\supp_{\mathsf{d}}(\rho_k^{(1)})\subset F_1$. By Fubini, we can find $y\in\supp_{\mathsf{d}}(\mu_2)$ such that
\begin{equation} \label{eq:B'_y-large-measure}
\rho_k^{(1)} (B'_y) \ge \ell^{-2}/10.
\end{equation}
Since $B'\subset G'\subset R_\ell G$, we know that if $x\in B'_y$ then there exists $\wt{x}\in \cD_\ell(x)$ such that $(\wt{x},y)\in G$. By \eqref{eq:def-G}, this implies that $\theta(\wt{x},y)\notin \bad''_{\ell_0}(\mu_1,\wt{x})$.  Recalling the definitions \eqref{eq:def-badprime}, \eqref{eq:def-badprimeprime}, we have shown that the hypotheses of Proposition \ref{prop:lower-bound-box-counting} hold for $\rho_k^{(1)}$ and $B'_y$, with $\beta=\eps$ (the separation between $y$ and $\supp(\rho_k^{(1)})$ follows from the fact that the supports of $\mu_1$ and $\mu_2$ are $(2\e)$-separated, making $\ell$ larger again). Recalling \eqref{eq:B'_y-large-measure}, we see that the error term in Proposition \ref{prop:lower-bound-box-counting} can be made $\le o_{T,\e,\tau}(1)$ by making $\ell$ even larger.  Applying the proposition, and recalling that $\D(\rho_k^{(1)}) \ge s_0-\delta$, where $\delta=o_{T,\e,\tau}(1)$ was defined in \eqref{eq:def-delta}, we conclude that (for this fixed value of $y$)
\[
\log \cN(\Delta_y(B'_y),\ell) \ge T\ell(s_0-o_{T,\e,\tau}(1)),
\]
and hence the same lower bound holds for $\log\cN(\Delta(B'),\ell)$. This concludes the proof of the claim in case (a).

We now consider case (b). Since $\cN(\Delta(B),\ell)$ and $\cN(\Delta(R_\ell B),\ell)$ differ by at most a multiplicative constant, it is enough to prove \eqref{eq:claim-full-dist-thm} for $R_\ell B$ in place of $B$. It follows from  the assumption of case (b), the decomposition \eqref{eq:decomposition-into-regular} for both $\mu_1, \mu_2$,  and the definitions of the sets $F_i$ that
\begin{align*}
\frac{1}{3}\mu(B) &\le \mu\left( \big((\R^2\setminus F_1)\times (\R^2\setminus F_2)\big) \cap B\right)  \\
&\le 2\cdot 2^{-\e T\ell} + \sum_{(k_1,k_2):\D(\rho_{k_i}^{(i)})<s_0-\delta} R_\ell\mu_1(X_{k_1}^{(1)}) R_\ell\mu_2(X_{k_2}^{(2)}) \left( \rho_{k_1}^{(1)}\times  \rho_{k_2}^{(2)} \right)(R_\ell B)  .
\end{align*}
Hence, using that $\mu(B)\ge \ell^{-2}/3$, we can find $k_1,k_2$ such that $\D(\rho_{k_i}^{(i)})<s_0-\delta$ for $i=1,2$, and
\begin{equation} \label{eq:measure-Bprime-large}
\left(\rho_{k_1}^{(1)}\times \rho_{k_2}^{(2)}   \right)(R_\ell B) \ge \frac{1}{3}\mu(B) - 2\cdot 2^{-\e T\ell} \ge \frac{1}{10}\ell^{-2},
\end{equation}
if $\ell$ is large enough in terms of $\e$.

Write $\rho_{k_i}^{(i)} = \rho'_i$ for simplicity. In light of \eqref{eq:sigma-positive} and our earlier assumption $\zeta<(s-1)^2$, the hypothesis of Proposition  \ref{prop:combinatorial-wolff} holds for the sequences $\sigma$ arising from both $\rho'_1$ and $\rho'_2$, with $\gamma=s-1$. Let $\eta>0, \xi\in (2/3,1]$ be the numbers given in the proposition (they depend on $s-1-\sqrt{\zeta}$) . Since $\D(\rho'_i)<\Lambda(s-1)-\delta$, where $\delta$ was defined in \eqref{eq:def-delta}, if we take $\ell$ sufficiently large, then the alternative (i) in Proposition \ref{prop:combinatorial-wolff} must hold.

Let $\rho''_i=R_{\lfloor \xi \ell \rfloor}(\rho'_i)$. Note that if $\rho'_i$ is $(\sigma_1,\ldots,\sigma_{\ell})$-regular, then $\rho''_i$ is $(\sigma_1,\ldots,\sigma_{\lfloor \xi\ell\rfloor})$-regular.  Using Lemma \ref{lem:energy-regular} (with $\lfloor \xi\ell \rfloor$ in place of $\ell$), recalling that $\zeta=o_{T,\e,\tau}(1)$ and that the alternative (i) in Proposition \ref{prop:combinatorial-wolff} holds, and making $\ell$ larger if needed, we get
\[
\log\cE_{4/3}(\rho''_i) \le \xi T\ell(\eta+o_{T,\e,\tau}(1)) \quad (i=1,2).
\]
On the other hand, we see from Lemma \ref{lem:energy} that
\[
\cE_{4/3+\e}(\rho''_i) \lesssim_{T,\e} 2^{\e \xi T\ell} \cE_{4/3}(\rho''_i).
\]
We apply Theorem \ref{thm:mattila-wolff} (together with the last two displayed equations) to get
\begin{align*}
\|\Delta(\rho''_1\times \rho''_2)\|_2^2  & \lesssim_{T,\e} 2^{\e \xi T\ell} \cE_{4/3}(\rho''_1) \cE_{4/3}(\rho''_2)\\
&\le 2^{o_{T,\e,\tau}(1) \xi T\ell} 2^{2 \eta\xi T\ell}.
\end{align*}
It follows from \eqref{eq:measure-Bprime-large} that $(\rho''_1\times\rho''_2)(R_{\lfloor \xi \ell\rfloor}B)\ge \ell^{-2}/10$. We deduce that, for $\ell$ large enough,
\[
\log \left\|\Delta\big( \left(\rho''_1 \times \rho''_2\right)_{R_{\lfloor \xi \ell\rfloor}B} \big)\right\|_2^2  \le \xi T\ell(2\eta+o_{T,\e,\tau}(1)).
\]
Applying Lemma \ref{lem:box-counting-from-L2-norm} to $f=\Delta\big( \left(\rho''_1 \times \rho''_2\right)_{R_{\lfloor \xi \ell\rfloor}B} \big)$ and $L=\lfloor \xi\ell\rfloor$, we conclude that
\begin{align*}
\log\cN(\Delta(B),\ell) &\ge \log\cN(\Delta(B),\lfloor \xi \ell\rfloor) \\
&\gtrsim \log\cN(\Delta(R_{\lfloor \xi \ell\rfloor} B),\lfloor \xi \ell\rfloor)
\ge \xi T\ell(1-2\eta-o_{T,\e,\tau}(1))
\end{align*}
for $\ell$ sufficiently large. Since $\xi(1-2\eta)\ge s_0-\sqrt{\zeta}$ by Proposition \ref{prop:combinatorial-wolff} and $\zeta=o_{T,\e,\tau}(1)$, this concludes the proof of case (b) of the claim, which completes the proof of Theorem \ref{thm:full-distance-set}.
\end{proof}

\section{Sharpness of the results}
\label{s:sharpness}

It is natural to ask what parts of our approach are sharp and which are not.  In this section we show that  the results of Section \ref{sec:combinatorial} are sharp, up to error terms. Hence, if the main results are not sharp (which seems likely), this is not due to the estimates for $\M_\tau(\sigma)$, but rather to the fact that Proposition \ref{prop:lower-bound-box-counting} (which connects the value of $\M_\tau(\sigma)$ to the size of distance sets) is itself not sharp. 

We begin by showing that Proposition~\ref{p:basic} is sharp for
all parameter values (and even the value of $f(a)$ can be chosen
as an arbitrary $b\in[Da,Ca]$). This is illustrated by the following functions.

First consider the case when $C=1$.
Let $x_0=y_0=0$,
$x_1=y_1=y_2=\frac{(1+D)(a-b)}{3(1-D)}$,
$x_2=2x_1$,
$x_3=\frac{a-b}{1-D}$, $y_3=Dx_3$, $x_4=a$ and $y_4=b$,
let $f(x_i)=y_i$ $(i=0,\ldots,4)$ and let $f$ be linear
on every interval $[x_{i-1},x_i]$.
(See Figure~\ref{fig:sharp} for $D=0$,
$a=1$ and $b=0$ and note that in the most important $b=Da$ case $x_3=x_4$, so the graph consists of only
three linear segments.)
It is clear that
$f(a)=b$ and
$ Dx\le f(x) \le  Cx$
on $[0,a]$.
It is easy to check that $f$ is $1$-Lipschitz.
The fact that the first inequality of \eqref{e:basict}
holds with equality (and for $b=aD$ also the second one)
follows from the observation that the set of hard
points of $f$ is $[x_2,x_3]$  (recall Definition \ref{def-hardpoint}),  Lemma~\ref{l:hardpoints}, and a straightforward calculation.

\begin{figure}
\includegraphics[width=0.9\textwidth]{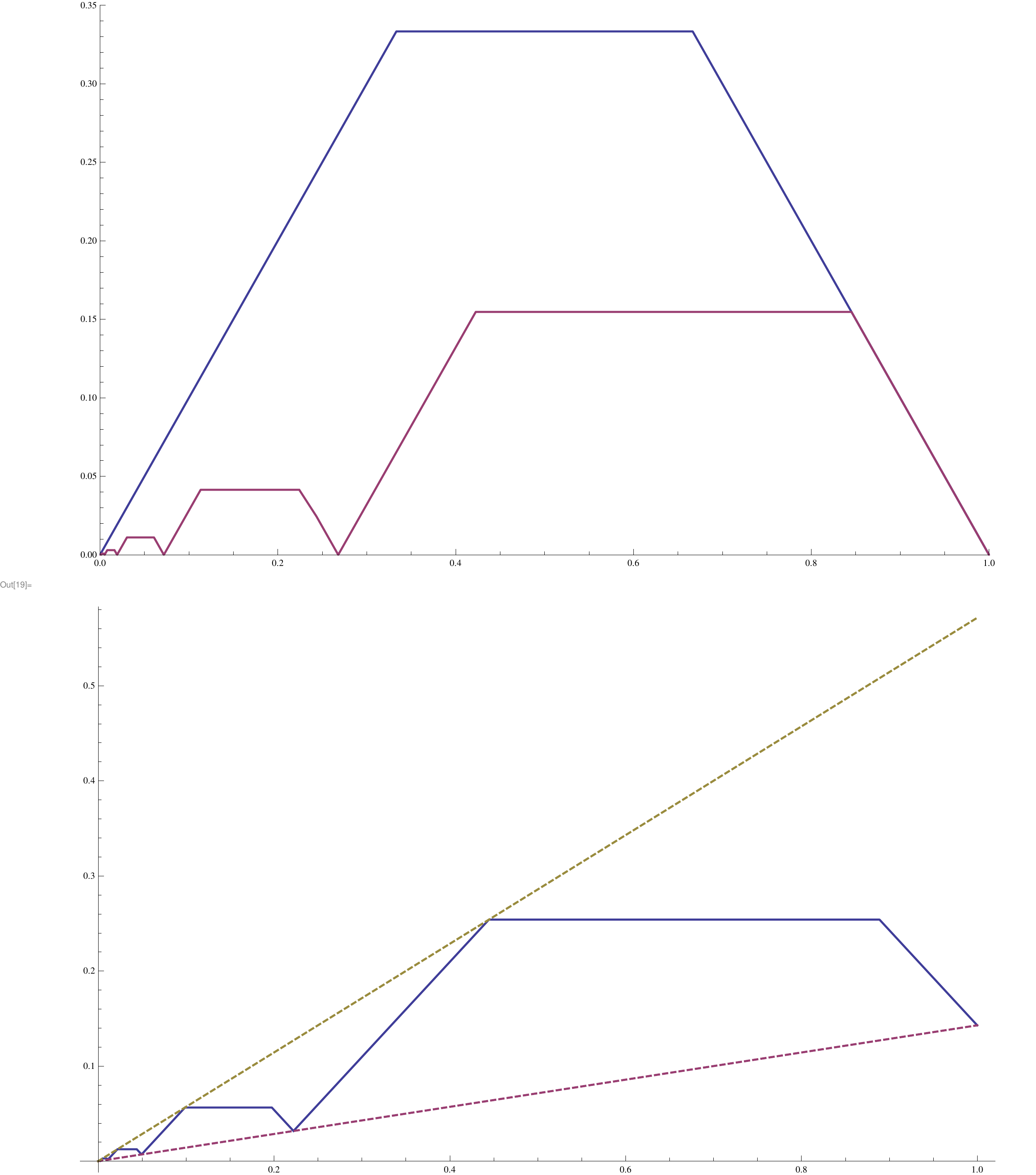}
\caption{These are graphs of functions that
witness the sharpness of Propositions~\ref{p:basic}
and \ref{p:initial} for $a=1$, $f(a)=D$ and
various parameters $C$ and $D$.
In the top graph, the larger function is for Proposition~\ref{p:basic} with $D=0$, $C=1$,
and the smaller function is for
Proposition~\ref{p:basic} with
$D=0$, $C=(\sqrt{3}-1)/2$
and also for Proposition~\ref{p:initial} with $D=0$.
The graph on the bottom shows the function that witnesses the sharpness of Proposition \ref{p:basic} with $D=1/7$ and $C=4/7$
and of Proposition~\ref{p:basic} with $D=1/7$,
together with the dashed lines
$y=x/7$ and $y=4x/7$.}
\label{fig:sharp}
\end{figure}

Now we consider the case when $C<1$.
Let $q=\frac{(1+D)(1-C)}{(1-D)(2+C)}$.
For $k=0,1,\ldots$ let
$$
x_{3k}=\frac{a-b}{1-C}q^k, \
 \quad
x_{3k+1}=\frac{a-b}{1-D}q^k, \
 \quad
x_{3k+2}=2x_{3k+3}, \
$$
$y_{3k}=Cx_{3k}$, $y_{3k+1}=Dx_{3k+1}$
and $y_{3k+2}=y_{3k+3}$.
Let $f(0)=0$, $f(a)=b$, $f(x_j)=y_j$ ($j=1,2,\ldots$)
and let $f$ be linear on $[x_1,a]$ and on each
$[x_{j+1},x_j]$.
(See Figure~\ref{fig:sharp} for $a=1$, $b=D=0$, $C=(\sqrt{3}-1)/2$, and for
$a=1$, $b=D=1/7$, $C=4/7$, and note again that
because of $b=Da$, the segment $[x_1,a]$ is degenerated in both cases.)
Again, it is clear that
$f(a)=b$ and
$Cx\le f(x) \le Dx$ on $[0,a]$, and
it is easy to check that $f$ is $1$-Lipschitz.
Now, observing that the set of hard points is
$\cup_{k=1}^\infty [x_{3k+2},x_{3k+1}]$,
Lemma~\ref{l:hardpoints} and another straightforward calculation
show that indeed
$\T(f)=\frac{(a-f(a))(C-2D)}{1+2C-3D}$.

Proposition~\ref{p:initial}  is also sharp up to the error term:
  for any given
  $D\in[0,1/2)$,
  we construct an $f$ that satisfies
  the conditions and for which for any $u\in(0,a]$ we have
  $\T(f|[0,u])\ge u\Phi(D)$.

  Recall that at the end of the proof of Proposition \ref{p:initial} we claimed
 that $\frac{(1-C)(C/2-D)}{1+2C-3D}\le\Phi(D)$ on $[2D,1]$.
 The function $\Phi(D)$ was of course chosen so that this
 is sharp, in fact for
 $C=\frac{3D-1+\sqrt{3-3D^2}}{2}\in
 ( 2D, 1 )$
 we have equality. Let $C$ be chosen this way, and
 let $f$ be the function we obtained above when we showed the sharpness of Proposition~\ref{p:basic} for these values of $C$ and $D$.
 (See Figure~\ref{fig:sharp} for $a=1$, $b=D=0$, and for
$a=1$, $b=D=1/7$.)
 As it was already mentioned above,
 the set of hard points of $f$ is
 $\cup_{k=1}^\infty [x_{3k+2},x_{3k+1}]$.
 By Lemma~\ref{l:hardpoints}, this implies that if we
 want to minimize $\T(f|[0,u]) / u$, then $u$ must be
 of the form $u=x_{3k+2}$.
 Lemma~\ref{l:hardpoints} and a simple calculation shows
 that for every $k$ we get
$$
\frac{\T(f|[0,x_{3k+2}])}{x_{3k+2}}=
\frac{(1-C)(C/2-D)}{1+2C-3D}=\Phi(D),
$$
which establishes the claimed sharpness.

Now we show that the lower estimates in \eqref{e:left} and \eqref{e:right}
are sharp in Proposition~\ref{p:stability}: for any $D\in[0,1/3]$ and
$\de\in[0,1/12]$
we construct $1$-Lipschitz functions $f_1$ and $f_2$ on $[0,1]$
such that $f_i(x)\ge Dx$ on $[0,1]$, $f_i(0)=0$,
$\T(f_i)=\frac{1-2D}3-\de$ ($i=1,2$),
$f_1(x)=x-3\de(1-D)$ on $[3\de,t_1]$ and
$f_2(x)=3t_1-x-3\de(1-D)$ on $[2t_1,1-\frac{3\de}{1-2D}]$,
where $t_1$ is given by \eqref{eq:t_1}.

Let
\[
f_1(x) =
\begin{cases}
   \min(x,3\de(1+D)-x) & \text{ on } [0,3\de] \\
   \min(x-3\de(1-D),1+D-x) & \text{ on } [3\de,1]
     \end{cases}.
\]
Then $f_1(x)\ge Dx$ on $[0,1]$, and $f_1(x)=x-3\de(1-D)$ on
$[3\de,\frac{1+D}2+\frac{3\de(1-D)}2]\supset  [3\de,t_1]$. One can check that set of hard points of $f_1$ is
$[2\de(1+D),3\de]\cup [\tfrac{2}{3}(1+D+3\de(1-D)),1]$,
where $\de\in[0, 1/12]$ ensures that both intervals
have nonnegative length,
and so
Lemma~\ref{l:hardpoints} gives that $\T(f_1)=\frac{1-2D}3-\de$.

Now, let
\[
f_2(x) =
\begin{cases}
   \min(x,3t_1-x-3\de(1-D)) & \text{ on } [0,1-\frac{3\de}{1-2D}] \\
   Dx & \text{ on } [1-\frac{3\de}{1-2D},1]
     \end{cases}.
\]
Then $f_2(x)\ge Dx$ on $[0,1]$, and
\[
f_2(x)=3t_1-x-3\de(1-D) \text{ on } \left[\frac{3}{2}(t_1-\de(1-D)),1-\frac{3\de}{1-2D}\right]\supset \left[2t_1,1-\frac{3\de}{1-2D}\right].
\]
After checking that the set of hard points of $f_2$ is
$[2t_1-2\de(1-D),1-\frac{3\de}{1-2D}]$, Lemma~\ref{l:hardpoints}  yields $\T(f_2)=\frac{1-2D}3-\de$.

We claim that Corollary~\ref{c:37/54} is sharp in the following sense:
If $D\in[0,0.26]$, $\eta>0$, $\xi\in(0,1]$, $\Lambda_1=\xi(1-2\eta)$ and
 for every $1$-Lipschitz function $f:[0,1]\to \R$ such that $f(0)=0$ and
 $f(x)\ge Dx$ on $[0,1]$ we have
\begin{equation} \label{eq:drop-implies-large-f}
\T(f)\ge 1-\Lambda_1\Longrightarrow f(x) > \frac{x}{3}-\eta\xi \text{ on } [0,\xi],
\end{equation}
 then $\Lambda_1 < \Lambda(D)$.

Indeed, if $\eta> 1/3-D$ then $\Lambda_1=\xi(1-2\eta)< 1-2(1/3-D) = 1/3+2D$, which
 is less than $\Lambda(D)$ when $D\in[0,0.26]$.
So we can suppose that $\eta\le 1/3-D$.
Let $x_1=\xi(4/3-\eta)/(1+D)$,
$x_2=\min(x_1,1)$ and
\[
f_3(x) =
\begin{cases}
  \min(x,-x+
   x_2(1+D)) & \text{ on }
  [0, x_2] \\
   Dx & \text{ on } [ x_2,1]
     \end{cases}.
\]
Then $f_3$ is $1$-Lipschitz,
$f_3(0)=0$, $f_3(x)\ge Dx$ on $[0,1]$ and
$f_3(x)=-x+
 x_2(1+D)$
on $[x_0, x_2]$, where

\[
x_0=x_2\frac{1+D}{2}\le x_1\frac{1+D}{2}=
\xi(2/3-\eta/2)<\xi.
\]
It is easy to see that the set of hard points of $f_3$
is $[\frac23 x_2(1+D),x_2]$ and so by
Lemma~\ref{l:hardpoints} we have $\T(f_3)=x_2(1-2D)/3.$
The assumptions $\eta\le1/3-D$
 and $\xi\le 1$
imply that $\xi\le x_2$, hence
$\xi\in[x_0,x_2]$, and we have

\[
f_3(\xi)=-\xi+x_2(1+D)\le -\xi+x_1(1+D)=\xi/3-\eta\xi.
\]
Thus by our assumption \eqref{eq:drop-implies-large-f} we must have $x_2(1-2D)/3=\T(f_3) < 1-\Lambda_1$.
If $x_1\ge 1$ then $x_2=1$, so we obtain
$\Lambda_1 < 2(1+D)/3$.
It is easy to check that $\Lambda(x)\ge 2(1+x)/3$
on $[0,1/2]$, so in this case we obtained
$\Lambda_1 < \Lambda(D)$ as we claimed.
So we can suppose that $x_1<1$ and so $x_2=x_1$.
Then $x_2(1-2D)/3 < 1-\Lambda_1$ gives
\begin{equation}\label{e:f_3 at xi}
\xi(4/3-\eta)\frac{1-2D}{3(1+D)} < 1-\xi(1-2\eta).
\end{equation}
Let
\[
\de=\Lambda_1-\frac23(1+D)=\xi(1-2\eta)-\frac23(1+D).
\]
We can clearly suppose that
$\Lambda_1\ge \Lambda(D)$.
Since $\Lambda(D)\ge \frac23(1+D)$, we obtain
$\de\ge 0$.
We also have $\de\le 1/12$ since this is clear
if $\xi\le 3/4$ and follows from
\eqref{e:f_3 at xi} and the assumption
$\eta\le 1/3-D$ if $\xi>3/4$.
For this value of $\de$ let $f_1$
be the $1$-Lipschitz function defined above
(to show the sharpness of \eqref{e:left} of Proposition~\ref{p:stability}).
Then $f_1(0)=0$, $f_1(x)\ge Dx$ on $[0,1]$ and
$\T(f_1)=\frac{1-2D}3-\de=1-\Lambda_1$, so by \eqref{eq:drop-implies-large-f} we must have
$f_1(x)>\frac{x}3-\eta\xi$ on $[0,\xi]$.
Since $\xi\le 1$ and $\eta>0$ we have
$3\de=3\xi(1-2\eta)-2(1+D)\le 3\xi-2\le\xi$. Thus we get $f_1(3\de)>\de-\eta\xi$.
Since $f_1(3\de)=3\de D$ this gives $\eta\xi > \de(1-3D)$.

From the definition of $\de$ we get
$\xi=\de + 2\eta\xi + \frac{2}{3}(1+D)$.
Considering $\de$, $\xi$ and $\eta\xi$ as
variables and $D$ as a parameter,
substituting the above expression into
\eqref{e:f_3 at xi}, and then using that
$\eta\xi > \de(1-3D)$,
after some calculations one gets
$\de < \frac{(1+D)(1-2D)}{18(3-4D+5D^2)}$, which yields
$\Lambda_1=\de+\frac23(1+D)<\Lambda(D)$, as we claimed.

We remark that if $\eta=1/3-D, \xi=1,
\Lambda_1=\xi(1-2\eta)=1/3+2D$ then \eqref{eq:drop-implies-large-f} holds, since
then $f(x)\ge Dx$ on $[0,1]$ already implies that
 $f(x)\ge x/3 - \eta\xi$ on $[0,\xi]$. So in order to make
 Corollary~\ref{c:37/54} sharp for every $D\in[0,1/3)$, the function $\Lambda(D)$ has to
   replaced by $\max(\Lambda(D),1/3+2D)$, which is equal to $\Lambda(D)$ if and
   only if $D\le 0.2609\ldots$. However, this version would not improve any of our distance set estimates.

Finally, we claim that Propositions \ref{prop:combinatorial},
\ref{prop:combinatorial-packing} and \ref{prop:combinatorial-wolff}
are also sharp, up to the error terms.
Indeed, given a $1$-Lipschitz function $f:[0,1]\to\R$ and $\ell\in\N$,
let $\sigma=\sigma_{f,\ell}\in [-1,1]^\ell$ be the sequence
\[
\sigma_i=\ell\left(f(i/\ell)-f((i-1)/\ell)\right).
\]
It is not hard to show that for any positive integer $L\le \ell$ and
good integer partition $\cP$ of $(0,L]$ there exists a good partition
$(a_n)$ of $[0,L/\ell]$ such that
 $\T(f|[0,L/\ell],(a_n))\le \frac1\ell \M(\sigma|(0,L],\cP)+O(\log\ell/\ell)$,
   thus
   $$
   \T(f|[0,L/\ell])\le \frac1\ell \M_\tau(\sigma|(0,L]) +O(\log\ell/\ell).
   $$
Thus, starting with the functions defined in this section that witness the
     sharpness of Propositions~\ref{p:basic} and \ref{p:initial} and
     Corollary~\ref{c:37/54}, this way we get sequences that show
     the sharpness of Propositions \ref{prop:combinatorial},
     \ref{prop:combinatorial-packing} and \ref{prop:combinatorial-wolff},
     up to the error terms.

\bibliographystyle{plain}
\bibliography{distancesetsgeneral}

\end{document}